\documentclass{article}

\usepackage{arxiv}

\usepackage[utf8]{inputenc} 
\usepackage[T1]{fontenc}    
\usepackage{hyperref}       
\usepackage{url}            
\usepackage{booktabs}       
\usepackage{amsfonts}       
\usepackage{nicefrac}       
\usepackage{microtype}      
\usepackage{xcolor}         
\usepackage{float}
\usepackage{algorithm}

\usepackage{graphicx}
\usepackage{natbib}
\usepackage{doi}
\usepackage{amsmath}
\usepackage{amssymb}
\usepackage{mathtools}
\usepackage{amsthm}
\usepackage{booktabs}
\usepackage{array}
\usepackage{multirow}
\usepackage{makecell}
\graphicspath{{images/}}
\usepackage{comment}

\usepackage{algorithm}
\usepackage{algpseudocode}

\newcommand{\red}[1]{\textcolor{black}{#1}}

\newcommand{\R}{\mathbb{R}}
\newcommand{\E}{\mathbb{E}\,}
\newcommand{\norm}[1]{\left\|
#1\right\|}
\def\wtgg{\widetilde{\nabla}}

\newcommand{\inner}[2]{\left\langle #1, #2\right\rangle}

\DeclareMathOperator*{\argmin}{arg\,min}

\providecommand{\E}{\mathbb{E}}

\newtheorem{theorem}{Theorem}
\newtheorem{proposition}{Proposition}
\newtheorem{lemma}{Lemma}
\newtheorem{corollary}{Corollary}

\theoremstyle{definition}
\newtheorem{definition}{Definition}
\newtheorem{assumption}{Assumption}

\theoremstyle{remark}
\newtheorem{remark}{Remark}

\title{Wall-Clock Complexity for Zeroth-Order Optimization with Tunable Oracle Fidelity}

\author{ Alexandra Suvorikova$^{*,\dagger,1}$ \quad Igor Pavlov$^{*, \dagger, 2}$ \quad Artem Vasin$^{2}$ \quad Georgii Bychkov$^{3}$\\
\textbf{Anastasia Antsiferova$^{3}$ \quad  Darina Dvinskikh$^{4}$ \quad Alexander Gasnikov$^{2,5,6}$}\\[0.5em] $^{1}$Weierstrass Institute for Applied Analysis and Stochastics, Berlin, Germany\\ 
$^{2}$Moscow Independent Research Institute of Artificial Intelligence, Moscow, Russia\\
$^{3}$MSU Institute for Artificial Intelligence, Moscow,
Russia\\
$^{4}$HSE University, Moscow, Russia\\
$^{5}$Trusted AI Research Center, RAS, Moscow, Russia\\
$^{6}$Innopolis University, Kazan, Russia
}

\begin{document}
\maketitle

\begingroup \renewcommand{\thefootnote}{\fnsymbol{footnote}} \setcounter{footnote}{0} \footnotetext[1]{Equal contribution.} \footnotetext[2]{Correspondence: \texttt{suvorikova@wias-berlin.de}, \texttt{1g0rp4vl@gmail.com}.} \endgroup

\begin{abstract}
 Zeroth-order (black-box) optimization is applied when gradients are unavailable and objective evaluations rely on expensive simulations. In many such applications, the oracle fidelity is tunable: higher-accuracy queries reduce noise but incur higher computational costs. To capture this trade-off, we study an accuracy-aware wall-clock model where each query with fidelity $\delta$ has a cost $c(\delta)$, and we minimize the total time $T_{\mathrm{total}} = \sum_{k=1}^{N} c(\delta_k)$, subject to a target accuracy constraint. We show how the choice of oracle type, noise model, and optimization scheme induces explicit wall-clock-optimal choices for the algorithmic parameters. For instance, we demonstrate that accelerated methods can be wall-clock inferior to non-accelerated schemes. Furthermore, we characterize the conditions under which a constant fidelity strategy is optimal in the Big-O sense. Our framework provides a unified methodology to translate convergence guarantees into practical fidelity and batching recommendations.
\end{abstract}


\section{Introduction}

Black-box optimization (BBO) plays a central role in modern applications, ranging from simulation-based engineering to adversarial machine learning \citep{kadowaki2022lossy, williams2023black, liu2025simulator}. This setting arises when the objective is defined through a numerical simulation or an algorithmic procedure. 
Since gradients are unavailable, one usually relies on an oracle that returns approximate values. 

A key feature of BBO is that the oracle's error is controllable. In practice, the oracle's fidelity $\delta$ serves as a tunable parameter, determined by the computational budget allocated to each query. For example, in experimental physics, the oracle fidelity depends on the number of Monte Carlo particle trajectories \cite{shirobokov2020black}; in molecular discovery, the number of internal search restarts impacts the oracle fidelity \cite{hoffman2022optimizing}. Another example is supervised PageRank learning, where estimating the stationary distribution of a parameterized Markov chain by MCMC yields a tunable noisy function-value oracle with cost proportional to $\delta^{-2}$ \cite{bogolubsky2016learning}.

\emph{Focusing on Zeroth-Order (ZO) optimization as a primary framework for such problems, we address this limitation by adopting a wall-clock perspective.} We explicitly model the dependency between oracle accuracy, iteration count, and total computational time. Specifically, we introduce an accuracy-aware complexity framework in which each oracle query is associated with a cost function $c(\delta)$ that depends on the desired fidelity $\delta$. A similar ``tunable oracle'' framework has been recently proposed for first-order convex optimization, deriving optimal inexactness schedules to minimize computational budget \cite{van2024optimal}. We extend this wall-clock perspective to the zeroth-order setting, where the trade-off governs not only the gradient quality but also the function evaluation itself.
\textbf{Our contributions are as follows}:
\begin{itemize}
\item \textbf{A wall-clock framework for tunable-fidelity ZO.} We recast zeroth-order optimization as a joint design problem over the iteration count $N$ and the per-query fidelity schedule $\{\delta_k\}$, under a polynomial cost-fidelity dependence $c(\delta) \propto \delta^{-\gamma}$. We show this is the only regime that produces a non-trivial trade-off.
\item \textbf{A master lemma turning convergence bounds into wall-clock-optimal designs}. We isolate a structural property---fidelity separability---shared by various methods, and give closed-form optimal schedules in one step (Prop.~\ref{lemma:master}). We also pin down exactly when a uniform schedule is already optimal up to constants.
\item \textbf{The role of $\gamma$ depends on the noise model.} Under adversarial noise, $\gamma$ selects the algorithm (acceleration can hurt once $\gamma \geq 1$) but doesn't tune parameters within it. Under Tsybakov noise, the opposite: $\gamma$ tunes the parameters, separating standard tuning from overbatching.
\item 
\textbf{An intermediate gradient method (IGM)}. We propose a new first-order method (Alg.~\ref{alg re-agm}) interpolating between GM and FGM, and prove its convergence under inexact gradients with time-varying oracle inexactness (Thm.~\ref{igm convergence}). IGM is what enables the application of our master lemma to the adversarial setting (Prop.~\ref{prop:igm_adversarial_various}).

\item \textbf{End-to-end recipes.} We illustrate and validated empirically the framework.
\end{itemize}

\textbf{Paper organization.} Sec.~\ref{sec:problem_setup} introduces the wall-clock model for tunable-fidelity zeroth-order
oracles and formulates the fidelity-allocation problem. Sec.~\ref{sec:related_work}
explains how standard ZO convergence guarantees can be converted into
the fidelity-separable form. Sec.~\ref{sec:adversarial_noise} and~\ref{sec:tsybakov} illustrate the framework under
adversarial and Tsybakov noise, respectively. Sec.~\ref{sec:experiments} provides experimental results. Finally, Sec.~\ref{sec:discussion} discusses the main
takeaways and limitations. All missing proofs and additional experimental
details are deferred to the appendix.

\textbf{Notations.}
For an integer \(N\ge 1\), we write \([N]:=\{1,\ldots,N\}\). The Euclidean norm
is denoted by \(\|\cdot\|\), and \(\langle \cdot,\cdot\rangle\) denotes the inner product. We use
\(\mathcal F_k\) for the filtration generated by the algorithm up to iteration
\(k\). The terms $O(\cdot)$ and $\tilde{O}(\cdot)$ suppress constant and log factor, respectively.

\section{Wall-clock complexity}
\label{sec:problem_setup}
We consider the problem of minimizing an objective function $\min_{x\in \Omega}f(x)$, where $f: \Omega \to \mathbb{R}$. We assume that the algorithm accesses $f$ through a noisy black-box oracle: upon receiving a query $x$, the oracle returns an estimate corrupted by an error $\xi$,
\begin{equation}
\label{eq:noisy_oracle}
\hat{f}(x) = f(x) + \xi.
\end{equation}
\emph{We assume the oracle is tunable: it is characterized by a fidelity parameter $\delta > 0$ (representing the inexactness level) which controls the accuracy of the estimation.} To formalize the precise dependence of $\xi(x)$ on $\delta$, let $\mathcal{F}_{k}$ denote the filtration representing the algorithm's history up to step $k$, and let the current iterate $x_k$ be measurable with respect to $\mathcal{F}_{k}$. We consider two widely used models for the oracle inexactness that fit this framework: stochastic and deterministic noise.
\begin{definition}[Tsybakov noise]
\label{asm:tsybakov_noise} The noise satisfies $\mathbb{E}[\xi^2 \mid \mathcal{F}_{k}] \le \delta^2$. If the algorithm queries $x_k$ multiple times, the respective noise realizations are mutually independent conditionally on $\mathcal{F}_{k}$.
\end{definition}
\begin{definition}[Adversarial noise]
\label{asm:adv_noise}
The noise is determined by a deterministic (but unknown) function of the queried point bounded by the fidelity level, $\xi = \xi(x_k)$ and  $|\xi(x_k)| \le \delta$.
\end{definition}
Crucially, we assume a direct trade-off between accuracy and computation: a smaller $\delta$ yields less noise but incurs a higher per-query cost $c(\delta)$. Let $\{(x_k, \delta_k)\}_{k\in [N]}$ be a sequence of queries performed by the algorithm. The wall-clock time is 
\begin{equation}
T_{\mathrm{total}}(N, \{\delta_k\}_{k\in [N]}) := \sum^N_{k=1} c(\delta_k).  \label{def:wall_clock_time}
\end{equation}

\begin{remark}[Batching]
\label{rem:batching}
Let $g(x,h,r,\zeta)$ denote a single gradient estimate constructed from the
noisy oracle~\eqref{eq:noisy_oracle}. For a batch size $B\ge 1$, the batched
gradient estimate at iteration $k$ is defined as $   g_k^{(B)}
    :=
    \frac{1}{B}\sum_{i=1}^{B}
    g(x_k,h_k,r_{k,i},\zeta_{k,i}),$
where, conditionally on $x_k$, the random variables
$\{(r_{k,i},\zeta_{k,i})\}_{i=1}^{B}$ are independent and identically
distributed.

The averaging reduces variance by $1/B$ but leaves bias unchanged, while wall-clock time scales as $T_{\mathrm{total}}=B\sum_{k=1}^{N}c(\delta_k)$. Our master analysis treats $B$ as fixed, but case studies (Sec.~\ref{sec:adversarial_noise} and~\ref{sec:tsybakov}) optimize $(N, B, \{\delta_k\})$, showing phase transitions between $B=1$ and $B \gg 1$.
\end{remark}

\textbf{Choice of $c(\delta)$.} We assume the computational cost of a single query scales polynomially with the fidelity as $c(\delta) \propto \delta^{-\gamma}$, $\gamma > 0$. Power-law costs naturally fit BBO, as oracle costs scale as a power of accuracy: e.g., a Monte Carlo estimator with $M$ samples achieves $\delta\sim M^{-1/2}$ via the central limit theorem, yielding $c(\delta)\sim\delta^{-2}$. Beyond empirics, the power-law is unique among non-decreasing cost models. Consider three natural regimes:
\textbf{(i)} $c(\delta)=\Theta(\ln(1/\delta))$,
\textbf{(ii)} $c(\delta)=\Theta(\delta^{-\gamma})$, and
\textbf{(iii)} $c(\delta)=e^{\Theta(1/\delta)}$.
In~(i) improving fidelity is free, so one drives $\delta$ as
small as possible and $T_{\mathrm{total}}$ reduces to $N$ up to log-factors. In~(iii) any fidelity gain is expensive, so one
uses the coarsest admissible $\delta$ and compensates by increasing $N$.
Both regimes collapse the joint design of $(N,\{\delta_k\})$ into a
trivial one-parameter problem. Only~(ii) produces a trade-off in
which the optimal fidelity depends on the problem parameters, and the
exponent $\gamma$ itself drives qualitative phase transitions in the
wall-clock-optimal algorithm as shown in Sec.~\ref{sec:adversarial_noise} and \ref{sec:tsybakov}.

\textbf{Wall-clock optimization problem.} To connect $T_{\mathrm{total}}$ with the algorithm's choice, we use a generic convergence metric $\mathrm{Gap}_N$. For instance, $\mathrm{Gap}_N$ can be  $\mathbb{E}[f(x_N)-f(x^*)]$ for convex objectives, or $\frac{1}{N}\sum_{k=1}^N \mathbb{E}[\|\nabla f(x_k)\|^2]$ for non-convex ones. Usually $\mathrm{Gap}_N$ is bounded by a deterministic function $\mathcal{E}$ depending on the iteration count $N$, the fidelities $\{\delta_k\}_k$, and the algorithmic hyperparameters $\Theta$,
\begin{equation}
\label{eq:gap_bound}
\mathrm{Gap}_N \le \mathcal{E}\big(N, \{\delta_k\}_{k\in [N]}, \Theta\big).
\end{equation}
An example is ZO-GD under the adversarial noise model (Def.~\ref{asm:adv_noise}). If $\delta_k = \delta$, the gap $\mathrm{Gap}_N := f(x_N) - f^*$ is bounded by $\mathcal{E}(N, \delta, \Theta) = \frac{c_1}{N} + c_2 \delta^2$, with constants $c_1, c_2$ depending on $\Theta$ (e.g., step size and Lipschitz constant).

One achieves accuracy  $\mathcal{E}\big(N, \{\delta_k\}_{k\in [N]}, \Theta\big) \le \varepsilon$ through various parameter configurations resulting in a different cumulative wall-clock time $T_{\mathrm{total}}$. So, we formulate the following optimization problem,
\begin{align}
    &\quad \min_{N, \{\delta_k\}_{k\in [N]}} T_{\mathrm{total}}(N, \{\delta_k\}_{k\in [N]}) ~~~
    \text{subject to}~~~\mathcal{E}(N,\{\delta_k\}_{k\in [N]}, \Theta) \le \varepsilon. \nonumber
\end{align}

\subsection{Optimal Fidelity Allocation}
\label{sec:master_lemma}
Standard ZO analyses derive error bounds $\mathcal{E}(N, \{\delta_k\}, \Theta)$ assuming uniform fidelity ($\delta_k = \delta$). 
\emph{Sec.~\ref{sec:related_work} provides a methodology to extend these bounds to time-varying $\delta_k$}, revealing a common structural property we term \emph{fidelity-separability}.
\begin{definition}[Fidelity-separable bound]
\label{def:separable_error}
We say that the error bound $\mathcal{E}(N,\{\delta_k\}_{k\in[N]},\Theta)$ is \emph{fidelity-separable} if, for any $N$ and $\Theta$, it can be written as
\begin{equation*}
\mathcal{E}(N,\{\delta_k\}_{k\in[N]},\Theta)
=
\mathcal{E}_0(N, \Theta)+\sum_{k=1}^N \alpha_k(N,\Theta)\,\delta_k^p,
\end{equation*}
with $\mathcal{E}_0(\cdot)$ and $\alpha_k(\cdot)$ depending on the assumptions on $f$ and the optimization method.
\end{definition}
\begin{remark}
\label{remark:channels}
Many algorithms admit the decomposition of Def.~\ref{def:separable_error} either directly or through a finite sum of fidelity-separable terms with different exponents: $\mathcal{E}(\cdot) = \mathcal{E}_0(\cdot) + \sum_{p\in\mathcal P}\sum_{k=1}^N \alpha_{k,p} \delta_k^p$. We refer to each $\sum_{k=1}^N \alpha_{k,p} \delta_k^p$ as a \textit{fidelity channel}. In such cases, the wall-clock analysis reduces to determining which channel governs the complexity in the high-precision regime.
\end{remark}
The following result provides the closed-form solution for the optimal schedule and the resulting time complexity. Appx.~\ref{appx:proofs_master} contains all proofs for this section.
\begin{proposition}\label{lemma:master}
Fix $N$ and let $\mathcal{E}$ be as in Def.~\ref{def:separable_error}. For sufficiently large $N$ such that  $\mathcal{E}_0(N, \Theta) \le \varepsilon - \epsilon$ for some $\epsilon > 0$, the optimal fidelity allocation $\delta_i^*(N)$ and the total wall-clock time $T_{\mathrm{total}}(N)$ are
\begin{equation}
A_N := \sum_{j=1}^N \alpha_j^{\frac{\gamma}{p+\gamma}}(N, \Theta), \quad
\delta_i^*(N) = \left( \frac{\epsilon}{A_N} \right)^{\frac{1}{p}} \alpha_i^{-\frac{1}{p+\gamma}}(N, \Theta), \quad T_{\mathrm{total}}(N) = \epsilon^{-\frac{\gamma}{p}}  A_N^{\frac{p+\gamma}{p}}.
\label{eq:time_adapt}
\end{equation}
\end{proposition}
\begin{remark}In practice, one may require $c(\delta) = \max\{1, \delta^{-\gamma}\}$ to prevent the per-query cost from vanishing. However, this does not alter the asymptotic scaling of $T_{\mathrm{total}}$ (see Lem.~\ref{lemma:master_aux}).
\end{remark}

Prop.~\ref{lemma:master} assumes that
$\delta_k$ can be reset at every iteration at no extra cost. In
practice, changing the fidelity often incurs overhead (re-seeding a
Monte Carlo simulator, rebuilding an inner solver, recalibrating a
physical oracle), making piecewise-constant or uniform schedules
preferable. Cor.~\ref{corr:uniform} bounds the wall-clock penalty of
restricting to a uniform schedule; analysis of \eqref{eq:benifits} shows that the penalty is at most constant-order whenever the coefficients $\{\alpha_j\}$ are mildly varying.
\begin{corollary}\label{corr:uniform}Let $\tilde{A}_N :=\sum_{i=1}^N \alpha_i(N, \Theta) $. Under the assumptions of Prop.~\ref{lemma:master}, for a uniform fidelity schedule $\delta_k = \delta$, the optimal level $\delta^*$ and the resulting total wall-clock time $T^{unif}_{\mathrm{total}}$ are
\begin{equation*}\delta^* = \left( \frac{\epsilon}{\tilde{A}_N} \right)^{\frac{1}{p}}, \quad T^{\mathrm{unif}}_{\mathrm{total}}(N) = N \epsilon^{-\frac{\gamma}{p}} \tilde{A}_N ^{\frac{\gamma}{p}}.
\end{equation*}

\end{corollary}

\textbf{High-precision regime.}
The large-$N$ assumption is natural in the high-precision regime:
as $\varepsilon\to 0$, the constraint $\mathcal{E}_0(N,\Theta)<\varepsilon$ with
$\mathcal{E}_0(N,\Theta)\asymp N^{-\beta}$ forces $N\gtrsim\varepsilon^{-1/\beta}\to\infty$.

\begin{corollary}[Large $N$]
\label{cor:large_N} Under Prop.~\ref{lemma:master}, assume
$\mathcal{E}_0(N,\Theta)\asymp N^{-\beta}$,
$A_N\asymp N^{\rho}$, with \(\beta>0\) and \(\rho>0\). Suppose that all $\delta_k < 1$. Then
\[
N^{*}\asymp \varepsilon^{-1/\beta},
\quad
\delta_k^{*}\asymp\varepsilon^{(1+\rho/\beta)/p}\,
\alpha_k(N^{*},\Theta)^{-1/(p+\gamma)},
\]
For the uniform schedule $\delta_k = \delta$ assume  $\tilde A_N\asymp N^{\sigma}$, with $\sigma \in \mathbb{R}$. If $p+\sigma\gamma>0$,
\begin{equation}
\label{eq:ihonestlydunno}
\widetilde N^{*}\asymp\varepsilon^{-1/\beta},
\quad
\delta^{*}\asymp\varepsilon^{(1+\sigma/\beta)/p}.
\end{equation}
\end{corollary}
\begin{remark}
The condition \(p+\sigma\gamma>0\) is automatic for \(\sigma\ge0\), but is
restrictive when \(\sigma<0\): then it requires $\gamma<\frac{p}{-\sigma}.$
If this fails, \eqref{eq:ihonestlydunno} does not apply; and the optimum is attained at \(\delta^*=1\).
\end{remark}

\textbf{Reading the exponents.}
The exponent~$\beta$ governs the
deterministic part of the bound and pins down the iteration budget
$N^{*}\asymp\varepsilon^{-1/\beta}$, which is identical for both schedules.
The exponent~$\sigma$ describes the cumulative growth of the fidelity
coefficients $\sum_{k}\alpha_{k}$ under the uniform schedule, while $\rho$
captures the same growth under the Lagrangian-optimal allocation in
Prop.~\ref{lemma:master}. The wall-clock benefit of time-varying fidelity
over a uniform one is therefore controlled by a single scalar gap:
\begin{equation}
\label{eq:benifits}
R_{\mathrm{opt}}(\varepsilon,\gamma)
:=\tfrac{T_{\mathrm{total}}(N^{*})}
{T_{\mathrm{total}}^{\mathrm{unif}}(\widetilde N^{*})}
\;\asymp\;
\varepsilon^{\,\Delta/(\beta p)},
\qquad
\Delta:=p+\sigma\gamma-\rho(p+\gamma).
\end{equation}

\textbf{When does a nonuniform schedule help?}
The term $\Delta$ is always non-negative. Indeed, with
$q:=\tfrac{\gamma}{p + \gamma}\in(0,1)$, Jensen's inequality applied to
$x\mapsto x^{q}$ gives
$A_N\le\;
N^{1-q}\Bigl(\sum_{j=1}^{N}\alpha_j\Bigr)^{q}
=N^{1-q}\,\tilde A_N^{\,q},$
which, after substituting $A_N\asymp N^{\rho}$, $\tilde A_N\asymp N^{\sigma}$,
yields $\rho(p+\gamma)\le p+\sigma\gamma$, i.e.\ $\Delta\ge 0$.
Hence the optimal nonuniform schedule is never asymptotically worse than the
uniform one. A \emph{polynomial} speedup ($\Delta>0$) requires Jensen's
inequality to be strict at the asymptotic level, which demands a
sufficiently concentrated coefficient sequence~$\{\alpha_j(N,\Theta)\}$.
\emph{For mildly decaying sequences such as $\alpha_j\asymp j^{-a}$,
$a\in(0,1)$, we get $\rho=1-aq$ and $\sigma=1-a$, so $\Delta=0$: in this regime, time-varying fidelity buys at most a
constant-factor improvement.} Polynomial gains arise when $\{\alpha_j\}$ is
heavy-tailed or spike-like. Appx.~\ref{app:optimized_ratio} provides the formal proof. Sec.~\ref{sec:exp_comparison} validates the result empirically.

\section{From ZO error bounds to fidelity-separable form}
\label{sec:related_work}
To apply Prop.~\ref{lemma:master} to a zeroth-order method, one must
show that its convergence bound $\mathcal{E}$ falls into the fidelity-separable
class (Def.~\ref{def:separable_error}). We establish this reduction
in three steps: we first fix the ingredients of a ZO setup; we then formulate a universal error
decomposition shared by a broad family of ZO methods; finally, we illustrate this using the
kernel two-point estimator. 
We begin with characterizing a ZO scenario by three interacting components. \textbf{(i) Oracle interface.} The way of constructing the gradient approximation $g_k$, e.g., one-point, symmetric or forward two-point, $\ell$-point, and
    kernel-smoothed finite differences.
\textbf{(ii) Noise model.} This work focuses on Tsybakov's noise (Def.~\ref{asm:tsybakov_noise}) or adversarial (Def.~\ref{asm:adv_noise}). \textbf{(iii) Algorithmic template.} The update rule---e.g.\ (accelerated) GD, SGD---which governs how estimation errors propagate across iterations.

Combinations of (i)--(iii) produce structurally different
bounds $\mathcal{E}(\cdot)$. Our goal is to show that these bounds share a common decomposition as in Def.~\ref{def:separable_error}. This is routed
through the conditional bias $b_k(x)$ and the second
moment $V_k(x)$ of a gradient estimate $g_k(x, \cdot)$. Specifically, we assume that there are such functions $b_k(x) \in \mathbb{R}^d$ and $V_k(x) > 0$, that
\[
\E[g_k(x, \cdot)\mid \mathcal F_k] :=\nabla f(x)+b_k(x), \qquad
\E[\|g_k(x, \cdot)\|^2\mid \mathcal F_k] \leq V_k(x).
\]
We next show that the fidelity-separable bound
(see Def.~\eqref{def:separable_error}, Rem.~\ref{remark:channels})
follows from the intermediate error decomposition below,
\begin{equation}
\label{eq:err_accumulation_2}
\mathrm{Gap}_N
\;\le\;
\tilde{\mathcal{E}}_0(N,\Theta)
\;+\; \sum_{k=1}^{N} \Bigl[ a_k(N,\Theta)\,\E[\|b_k(x_k)\|]
+ c_k(N,\Theta)\,\E[\|b_k(x_k)\|^{2}] +
e_k(N,\Theta)\,\E[V_k(x_k)]\Bigr],
\end{equation}
where $\tilde{\mathcal{E}}_0, a_k,c_k,e_k$ depend on the
algorithmic hyperparameters $\Theta$ and number of steps $N$.
\begin{proposition}
\label{lem:master_decomp}
Gradient descent, accelerated gradient descent, SGD
schemes satisfy \eqref{eq:err_accumulation_2}. Explicit
expressions for $(a_k,c_k,e_k)$ are collected in
Tab.~\ref{tab:drift_coefficients} in the Appendix.
\end{proposition}
Appx.~\ref{sec:proof_master_decomp} provides the proof.  Crucially, \emph{all fidelity dependence is routed through $b_k$ and $V_k$, not $a_k, c_k, e_k$}. It
therefore remains to quantify how these two quantities scale with
$\delta_k$.

\textbf{Example: the kernel two-point estimator.}
We next show how to translate fidelity into concrete bounds on
$b_k,V_k$ using the kernel two-point estimator.
\begin{definition}[Kernel two-point estimator]
\label{def:two_point_estimator}
Let $h>0$ be the smoothing radius,
$\zeta\sim\mathrm{Unif}(S_2^{d-1})$ a random direction,
$r\sim\mathrm{Unif}([-1,1])$ an independent scalar, and
$K\colon[-1,1]\to\mathbb{R}$ a smoothing kernel
(Def.~\ref{def:smoothing_kernel}). Given access to the noisy
oracle~\eqref{eq:noisy_oracle}, the estimator at $x$ is
\begin{equation}
g_k(x;h,r,\zeta)
\;:=\;
\frac{d}{2h}\,
\bigl(\hat f(x+hr\zeta)-\hat f(x-hr\zeta)\bigr)\,K(r)\,\zeta.
\label{eq:two_point_kernel_estm}
\end{equation}
\end{definition}
\begin{assumption}[Higher-order Smoothness]
\label{asm:hoelder_condition}
Fix $\beta \ge 2$ and let $\ell = \lfloor \beta \rfloor$. We assume $f \in \mathcal{F}_{\beta}(L_\beta)$, meaning $f$ is $\ell$-times continuously differentiable and the $\ell$-th derivative satisfies the H\"older condition,
$
\| f^{(\ell)}(x) - f^{(\ell)}(z) \|\leq L_\beta \| x - z \|^{\beta - \ell}, \quad \forall x, z \in \mathbb{R}^d.
$
\end{assumption} 
\begin{proposition}[$\|b_k(x_k)\|$ and $V_k(x_k)$]
\label{prop:kernel_bias_variance}
Let $g_k$ be as in Def.~\ref{def:two_point_estimator}. If $f$ satisfy Asm~\ref{asm:hoelder_condition}, then for any fixed $h>0$,
\begin{align}
\|b_k(x_k)\|
&\;\lesssim\;
L\,h^{\beta-1}
\;+\;
\frac{d}{h}\delta_k,
\label{eq:kernel_bias}\quad
V_k(x_k) \;\lesssim\;
\E\!\bigl[\,|f(x_k^{+})-f(x_k^{-})|^{2}\,\big|\,\mathcal{F}_{k-1}\bigr]
\;+\;
\frac{d^{2}}{h^{2}}\delta^2_k.
\end{align}
\end{proposition}
Appx.~\ref{sec:proof_example_kernel_estimator} contains the proof. Taking expectation and substituting \eqref{eq:kernel_bias} into \eqref{eq:err_accumulation_2} and collecting powers of
$\delta_k$ gives a fidelity-separable bound with \emph{two channels} (see Rem.~\ref{remark:channels}). The next sections illustrate the framework. We conclude this section with regularity assumptions required for the subsequent analysis. 
\begin{assumption}[Regularity of $f$]
\label{asm:regularity}
For all $x, y \in \mathbb{R}^d$, the differentiable function $f: \mathbb{R}^d \to \mathbb{R}$ satisfies:
\begin{itemize}
\item[(i)] \textbf{$L$-Smoothness:} $\|\nabla f(x) - \nabla f(y)\| \le L \|x-y\|$.
\item[(ii)] \textbf{$\mu$-Strong Convexity:} $f(y) \ge f(x) + \langle \nabla f(x), y - x \rangle + \frac{\mu}{2} \|y - x\|^2$ (convex if $\mu=0$).
\end{itemize}
\end{assumption}

\section{Adversarial noise}
\label{sec:adversarial_noise}

This section studies wall-clock complexity under adversarial noise (Def.\ref{asm:adv_noise}). Our main finding is that the cost exponent $\gamma$ acts as a template selector: it decides whether acceleration is wall-clock-optimal ($\gamma=1$ phase transition) but does not change the leading-order choices of $N, B, \delta$.

\subsection{Wall-clock complexity of gradient methods}
\label{sec:gradient_methods}

Let $f$ satisfy Asm.~\ref{asm:regularity} and let the oracle's answer $\hat{f}(x)$ be corrupted by the adversarial noise (Def.~\ref {asm:adv_noise}). We consider forward finite differences gradient estimation
\begin{equation}
\label{eq:one_point_estimation}
g(x,h)
=
\frac{1}{h}
\sum_{k=1}^{d}
\bigl(\hat f(x+h e^k)-\hat f(x)\bigr)e^k,
\end{equation}
where \(\{e^k\}_{k=1}^d\) is an orthonormal basis. Results from~\citep{gasnikov2023randomized}, with $h = 2\sqrt{\nicefrac{\delta}{L}}$, yield $\| b(x) \|_2 \le 2 \sqrt{d \delta L}$. First, consider the case $\delta_k = \delta$.
If $f$ is strongly convex (Asm.~\ref{asm:regularity}, $\mu > 0$),
\begin{equation} \label{eq:err_strongly_convex}
    f(x_N) - f^* 
    \lesssim \mathcal{E}(N, \delta, \Theta)  := LR^2 \exp\left(-\left(\frac{\mu}{L}\right)^{\frac{1}{p}} N\right) + \left(\frac{L}{\mu}\right)^{\frac{2p-1}{p}}  d\delta,
\end{equation}
with $\Theta = (L, p)$.
Using gradient methods with intermediate convergence; see \cite{devolder2013intermediate}. The parameter $p \in [1, 2]$ controls the rate of convergence: $p=1$ recovers GM, and $p=2$ recovers fast GM (FGM) \citep{vasin2023accelerated}. 

\begin{proposition}[Uniform schedule $\delta$]
\label{prop:igm_adversarial_uniform}
Consider gradient estimate~\eqref{eq:one_point_estimation}. Let $f$ satisfy Assm.~\ref{asm:regularity}. The wall-clock time complexity required to achieve $\mathcal{E}(N, \delta, \Theta) \le \varepsilon$ is 
\begin{equation}
\label{eq:complexity_igm}
    T_{\mathrm{total}} = 
    \begin{cases}
        \tilde{O}\left( \left(\frac{L}{\mu}\right)^{\frac{1}{p} (1 - \gamma) + 2\gamma}\left(\frac{\varepsilon}{d}\right)^{-\gamma} \right), & ~\text{if}~ \mu > 0,\\
        O \left( \left( \frac{\varepsilon}{9 d} \right)^{-\gamma} 
         \left( \frac{3LR^2}{\varepsilon} \right)^{\gamma + \frac{1}{p}} \right), & ~\text{if}~ \mu = 0.
    \end{cases}
\end{equation}

\end{proposition}
Appx.~\ref{section:proof intermidiate} provides proof. Per \eqref{eq:complexity_igm}, the optimal $p$ for strongly convex problems ($\mu>0$) depends on $\gamma$: $p=2$ (FGM) is best for $\gamma<1$, and $p=1$ (GM) for $\gamma>1$. Thus, accelerated methods may actually be wall-clock slower if $\gamma>1$; Sec.~\ref{exp:gm_fgm}) provides experiments. In the convex regime ($\mu=0$, $LR^2/\varepsilon>1$), however, $p=2$ is optimal across all $\gamma$.

To analyse $\{\delta_k\}_k$ selected using Prop.~\ref{lemma:master},we introduce an intermediate gradient method (IGM) (Alg.~\ref{alg re-agm}) and prove its convergence (Thm.~\ref{igm convergence}),
\[
    f(x^N) - f^* \leqslant \left(1 - \frac{1}{16} \left(\frac{\mu}{L}\right)^{\frac{1}{p}} \right)^N LR^2
    + 6 d \left( \frac{\mu}{L} \right)^{2\frac{1 - p}{p}} \sum_{k = 0}^{N - 1} \left(1 - \frac{1}{16} \left(\frac{\mu}{L}\right)^{\frac{1}{p}} \right)^{N - k - 1} \delta_k.
\]

Thus we can obtain the following result.

\begin{proposition}[Time-varying $\delta_k$]
\label{prop:igm_adversarial_various}
Under assumptions of Prop.~\ref{prop:igm_adversarial_uniform},
if $\mu > 0$,
\[
T_{\mathrm{total}}^{\mathrm{IGM}} =
\begin{cases}
    \widetilde{O} \left( \left( \frac{\varepsilon}{12 d} \right)^{-\gamma} \left(\frac{L}{\mu}\right)^{\frac{1 - \gamma}{p} + 2\gamma} \right), & ~\text{if}~~\gamma \approx 0,\\
    O \left( (1 + \gamma^{-1}) \left( \frac{\varepsilon}{12 d} \right)^{-\gamma} \left(\frac{L}{\mu}\right)^{\frac{1 - \gamma}{p} + 2\gamma} \right) & ~\text{if}~~\gamma \not\approx 0,
\end{cases}
\]
The case $\gamma \approx 0$ corresponds constant time for any $\delta$ precision calculation, thus it makes sense chose $\delta_k = \delta$. Such case refers to Proposition~\ref{prop:igm_adversarial_uniform}. Further, if $\mu = 0$,
\[
T_{\mathrm{total}}^{\mathrm{IGM}} =
\begin{cases}
    \widetilde{O} \left( \left( \frac{\varepsilon}{24 d} \right)^{-\gamma} \left(\frac{2 LR^2}{\varepsilon}\right)^{\frac{1 - \gamma}{p} + 2\gamma} \right), & ~\text{if}~~\gamma \approx 0,\\
    O \left( (1 + \gamma^{-1}) \left( \frac{\varepsilon}{24 d} \right)^{-\gamma} \left(\frac{2 LR^2}{\varepsilon}\right)^{\frac{1 - \gamma}{p} + 2\gamma} \right) & ~\text{if}~~\gamma \not\approx 0.
\end{cases}
\]
\end{proposition}


\subsection{Accelerated ZO-SGD with adversarial noise}
\label{sec:adv_noise}
Set $f(x) := \mathbb{E}_{\zeta}[f(x, \zeta)]$, where $\zeta$ is a r.v. with an unknown distribution. Let the oracle output be corrupted by deterministic adversarial noise, $\hat{f}(x, \zeta) := f(x, \zeta) + \xi(x)$ and $|\xi(x)| \le \delta$. Let $f(x, \cdot)$ satisfy Assm.~\ref{asm:hoelder_condition} uniformly in $\zeta$ with some $\beta \ge 2$, so that the same
holds for $f(x) = \mathbb{E}_\zeta[f(x,\zeta)]$. Assume there exists $\sigma_* > 0$ such that $\mathbb{E}_{\zeta}\norm{\nabla f(x^*, \zeta) - \nabla f(x^*)}^2 \le \sigma^2_*,$ $x^* := \argmin_{x}f(x)$.

We use the accelerated zero-order SGD by~\cite{lobanov2023accelerated}. To estimate the gradient, we employ a two-point scheme (see Def.~\ref{def:two_point_estimator}) and adopt batching (see Rem.~\ref{rem:batching}). Let $B$ be the batch size. For the sake of transparency, we assume a constant smoothing parameter $h_t \equiv h$. Setting $Q:=BN$, we use Thm.~3.1 by \cite{bychkov2024accelerated}, and write
\begin{align*}
\mathrm{Gap}_{N} := \mathbb{E}f(x_N) - f^* \lesssim \mathcal{E}(N, B, \delta, \Theta) &:= \frac{LR^2}{N^2} + \frac{LR^2}{Q}  +
\frac{R}{\sqrt{Q}}\Big(\sqrt{d\kappa\sigma_*}
+\sqrt{d\kappa L h}+\sqrt{\kappa\,\frac{d\delta}{h}}\Big)
\\
&+\;
\underbrace{R\Big(\kappa_\beta L h^{\beta-1}+\frac{d\delta}{h}\Big)}_{(I)}
\;+\;
\underbrace{N\Big(\kappa_\beta^2L\,h^{2(\beta-1)}+\frac{d^2\delta^2}{Lh^2}\Big)}_{(II)},
\end{align*}
with $\kappa, \kappa_{\beta}$ defined in \eqref{eq:kappa}, and $\Theta := (d, \sigma_*, L, R, h)$. Note that for a fixed 
\(\delta\), 
(I) and (II) are balanced by choosing $  h^*(\delta) =\left(
        \frac{d\delta}{\kappa_\beta L}
    \right)^{1/\beta}$ (see Appx.~\ref{app:adv_design}).

\begin{proposition}
\label{prop:adv2}
Let \(Q:=BN\), \(p=\tfrac{\beta-1}{\beta}\), and
\(c(\Delta)=\Delta^{-\gamma}\). 
Define $ a_1 := (\kappa_\beta L)^{1/\beta}\,d^{(\beta-1)/\beta},
\\a_2 := \kappa_\beta^{2/\beta}\,L^{2/\beta-1}\,d^{2(\beta-1)/\beta}.$
In the large-\(N\), large-\(Q\) regime, the optimal design satisfies
\[
    N^*
    \asymp
    \frac{LR^2}{\varepsilon},
    \qquad
    Q^*
    \asymp
    \frac{R^2\kappa d\sigma_*}{\varepsilon^2},
    \qquad
    B^*
    =
    \frac{Q^*}{N^*}
    \asymp
    \frac{\kappa d\sigma_*}{L\varepsilon}.
\]
The optimal fidelity is $\delta^* \asymp \tfrac{\varepsilon^{\frac{\beta}{\beta-1}}}{d \kappa_\beta^{\frac{1}{\beta-1}} L^{\frac{1}{\beta-1}} R^{\frac{\beta}{\beta-1}}}$ and \(h^*=h^*(\delta^*)\). Consequently,
\[
  T_{\rm total}^* \asymp \frac{\kappa \sigma_* d^{1+\gamma} \kappa_\beta^{\frac{\gamma}{\beta-1}} L^{\frac{\gamma}{\beta-1}} R^{2+\frac{\gamma\beta}{\beta-1}}}{\varepsilon^{2+\frac{\gamma\beta}{\beta-1}}}
\]
\end{proposition}
\textbf{Take-away message.}
After balancing \(h\), the adversarial model reduces to two
fidelity channels, $R a_1\delta^p$ and $ N a_2\delta^{2p}.$
In the large-\(N\) regime the second channel is active, which determines
\(\delta^*\). The total number of oracle calls is instead dictated by $\sigma_*$, $ Q^*=BN^*
    \asymp
    \frac{R^2\kappa d\sigma_*}{\varepsilon^2},$ and $ B^*
    \asymp
    \frac{\kappa d\sigma_*}{L\varepsilon}.$
Thus the batching regime is controlled by comparing \(B^*\) with \(N^*\): if $B^*\le N^*$, then $\kappa d\sigma_*\lesssim L^2R^2.$ Hence, small $\sigma_*$ leads to moderate batching, while large $\sigma_*$ leads to overbatching.

In this adversarial large-\(N\) regime, \(\gamma\) affects the final
wall-clock cost through \((\delta^*)^{-\gamma}\), but it does not change
the leading-order choices of \(N^*,B^*,\delta^*\). \emph{The next section shows that
under another noise regime, the oracle-cost exponent \(\gamma\) can also change the
optimal tuning parameters.}

\section{Tsybakov's noise}
\label{sec:tsybakov}
This section considers batched setting (see Rem.~\ref{rem:batching}). So, the objective is $T_{\rm total}(N, \delta) = B N \delta^{-\gamma}$.
In contrast to Sec.~\ref{sec:adversarial_noise}, under Tsybakov noise (Def.~\ref{asm:tsybakov_noise}) the cost exponent $\gamma$ acts as a parameter tuner controlling batching, with threshold $\gamma=2$. 

\subsection{Strongly convex case: the batched baseline}
\label{sec:tsybakov_str_cvx}
Let \(f\) satisfy Asm.~\ref{asm:regularity} with \(\mu>0\),
Asm.~\ref{asm:hoelder_condition} with \(\beta\ge2\), and assume
\(\|f^{(2)}(x)-f^{(2)}(z)\|\le L\) for all \(x,z\). We use the btached
\(\ell_2\)-randomized estimator of~\citet{akhavan2024gradient}. Using
Cor.~19 by~\citep{akhavan2024gradient}, we derive the following bound for the batched case (see Appx.~\ref{app:akhavan_batched_proof} for detail),
\begin{equation}
\label{eq:akhavan_bound_batched}
\mathrm{Gap}_{N} := \mathbb E[f(\hat x_N)-f^*]
\lesssim
\underbrace{\frac{dL^2R^2}{\mu N}}_{\mathcal{E}_0(N, \Theta)}
+
\underbrace{\frac{L_\beta^{2/\beta}}{\mu}
\left(\frac{d^2\delta^2}{BN}\right)^{\frac{\beta-1}{\beta}}}_{\text{(I)}}
+
\underbrace{\frac{d^{1+2/\beta}L^2}{\mu N^{1+1/\beta}}
\left(\frac{\delta^2}{BL_\beta^2}\right)^{1/\beta}}_{\text{(II)}},
\end{equation}
where $\Theta= R$ and $L_{\beta}$ comes from Asm.~\ref{asm:hoelder_condition}.
Rem.~\ref{remark:channels} applies to \ref{eq:akhavan_bound_batched} with two channels. In the high-precision regime, (I) is an active channel;  see Appx.~\ref{app:akhavan_batched_proof}.

\begin{proposition}[Optimal batched design under strongly convex Tsybakov noise]
\label{prop:akhavan_walltime_batched}
Let the assumptions of Sec.~\ref{sec:tsybakov_str_cvx} hold, and let
\(\mathcal E(N,B,\delta)\) be given by~\eqref{eq:akhavan_bound_batched}.

If \(0<\gamma<2\), then batching is not beneficial at the leading order:
\(B^*=1\), and
\[
    N^*
    \asymp
    \frac{dL^2R^2}{\mu\varepsilon}
    \frac{2\beta-2+\gamma}{(\beta-1)(2-\gamma)},\qquad
    \delta^*
    \asymp
    d^{-1/2}L R\,
    \mu^{\frac{1}{2(\beta-1)}}L_\beta^{-\frac{1}{\beta-1}}
    \varepsilon^{\frac{1}{2(\beta-1)}}C_\beta(\gamma),
\]
where $C_\beta(\gamma):= \left(\tfrac{\gamma\beta}{2\beta-2+\gamma}\right)^{\frac{\beta}{2(\beta-1)}} \left(\tfrac{2\beta-2+\gamma}{(\beta-1)(2-\gamma)}\right)^{1/2}.$
Consequently,

$$T_{\rm total}^* \asymp d^{\frac{2+\gamma}{2}} L^{2-\gamma} R^{2-\gamma} L_\beta^{\frac{\gamma}{\beta-1}} \mu^{-\frac{2\beta-2+\gamma}{2(\beta-1)}} \varepsilon^{-\frac{2\beta-2+\gamma}{2(\beta-1)}} C_\beta(\gamma)$$

If \(\gamma\ge2\), the optimum is attained at the boundary \(\delta^*=1\), and
\[
    T_{\rm total}^*
    \asymp
    d^2L_\beta^{2/(\beta-1)}
    \mu^{-\beta/(\beta-1)}
    \varepsilon^{-\beta/(\beta-1)} .
\]
The same total-work order can be achieved either sequentially with
\(B=1\), or with a larger batch size and fewer sequential steps.
\end{proposition}

\textbf{Take-away message.}
The Tsybakov model couples the iteration budget, fidelity and batch size through
the main noise channel
\((BN)^{-(\beta-1)/\beta}\delta^{2(\beta-1)/\beta}\). For \(0<\gamma<2\),
reducing \(\delta\) is cheaper than batching, hence \(B^*\asymp1\). For
\(\gamma\ge2\), reducing \(\delta\) is too expensive and the optimum moves to
\(\delta^*=1\). In this regime, batching does not improve the leading total
oracle work, but it can trade parallelism for smaller sequential depth.

\subsection{Tsybakov noise, accelerated convex case: batching as a new tuning parameter}
\label{sec:tsybakov_acc_batching}

We use the accelerated ZO-SGD scheme \citep{lobanov2024black} with the batched
two-point estimator (Rem.~\ref{rem:batching}). The convergence bound is
\begin{align*}
\label{eq:acc_tsyb_unified}
\mathcal{E}(N,B,\delta,h, \Theta)
&\lesssim
\frac{LR^2}{N^2}
\max\left\{
    1,\,
    \left(\frac{\kappa d}{B}\right)^2
\right\}
+
N
\min\left\{
    \frac{B}{\kappa},\,
    \frac{\kappa d^2}{B}
\right\}
\left(
    \frac{Lh^2}{d}
    +
    \frac{\delta^2}{Lh^2}
\right)
\\
&+
\widetilde R\kappa_\beta L_\beta h^{\beta-1}
+
\frac{N(\kappa_\beta L_\beta h^{\beta-1})^2}{L}.
\end{align*}
with $\kappa_\beta, L_{\beta}$ defined in Def.~\ref{def:two_point_estimator} and $\Theta = (\kappa, R, \tilde{R})$.
\begin{proposition}[Optimal design under accelerated Tsybakov noise]
\label{prop:acc_tsyb_design}
Assume \(\widetilde R=O(R)\), and fix some suitable \(h > 0\). In the small-batch regime \(B\le 4\kappa d\), set \(Q:=BN\). Then we get
\begin{equation}
\label{eq:Q_opt}
    Q^*
    \asymp
    \kappa dR\sqrt L\,\varepsilon^{-1/2}
    \left(
        \frac{2+3\gamma}{2+\gamma}
    \right)^{1/2},
    \quad
        \delta^*
    \asymp
    hL^{1/4}(dR)^{-1/2}
    \varepsilon^{3/4}
    G(\gamma),    .
\end{equation}
with $G(\gamma)
    :=
    \frac{(2\gamma)^{1/2}(2+\gamma)^{1/4}}
         {(2+3\gamma)^{3/4}}.$ Let $C_\gamma$ be a constant depending only on $\gamma$. We get,
\[
    T^* \asymp \kappa d^{\frac{2+\gamma}{2}} R^{\frac{2+\gamma}{2}} L^{\frac{2-\gamma}{4}} \varepsilon^{-\frac{2+3\gamma}{4}} h^{-\gamma} C_\gamma.
\]
In the large-batch regime \(B>4\kappa d\), the active constraints
imply $ \delta^2
    \asymp
    \frac{\varepsilon BLh^2}{N\kappa d^2}$, $T^*(B,h)\propto B^{1-\gamma/2}.$
Hence \(\gamma=2\) is the batching threshold. For \(0<\gamma\le2\),
overbatching is not beneficial at leading order. For \(\gamma>2\), the
optimum increases \(B\) until \(\delta^*=1\), yielding
$$N^* \asymp RL^{\frac{1}{2}}\,\varepsilon^{-\frac{1}{2}}, \qquad B^* \asymp \kappa d^2 R L^{-\frac{1}{2}} \varepsilon^{-\frac{3}{2}} h^{-2}, \quad T^* \asymp \kappa d^2 R^2 \varepsilon^{-2} h^{-2}.$$
\end{proposition}
Appx.~\ref{app:acc_tsyb_design} provides the proof.

\textbf{Choice of $h$.} 
Prop.~\ref{prop:acc_tsyb_design} should be read conditionally on
an admissible smoothing radius \(h\). This parameter is not arbitrary: in the small-batch regime, the optimal choice is $  h_{\rm sm}^*(B)=\min\{H_1,H_2,H_3(B)\},$ with $H_1\asymp \varepsilon^{3/4},$ $ H_2\asymp \varepsilon^{1/(\beta-1)},$ and $    H_3(B)\asymp(\tfrac{B\varepsilon^{3/2}}{d})^{1/(2\beta-2)}$; Appx.~\ref{app:acc_tsyb_design} provides the details. 
Therefore the small-batch wall-clock bound should be
$$T_{\rm sm}^*(B) \asymp \frac{\kappa d^{\frac{2+\gamma}{2}} R^{\frac{2+\gamma}{2}} L^{\frac{2-\gamma}{4}} C_\gamma}{\varepsilon^{\frac{2+3\gamma}{4}} \bigl(h_{\rm sm}^*(B)\bigr)^\gamma}$$

In the overbatching regime \(\gamma>2\), the optimal smoothing radius is again the largest admissible one, $  h_{\rm lg}^*=\min\{K_1,K_2,K_3\},$
where $K_1 \asymp d^{\frac{1}{4}},$$ K_2 \asymp \varepsilon^{\frac{1}{\beta-1}},$ and $K_3 \asymp \varepsilon^{\frac{3}{2(2\beta-2)}}$; Appx.~\ref{app:acc_tsyb_design} provides the detail. The optimized overbatching complexity is $    T_{\rm overbatch}^*
    \asymp
    \kappa d^2R^2
    \varepsilon^{-2}
    \bigl(h_{\rm lg}^*\bigr)^{-2}.$

\textbf{Take-away message.}
This case has the same logic as Sec.~\ref{sec:tsybakov_str_cvx},but now
batching and smoothing provide additional tuning parameters. In the
small-batch regime \eqref{eq:Q_opt}, the optimal solution is $ Q^*\asymp dR\sqrt L\,\varepsilon^{-1/2},$ $\delta^*
    \asymp
    h^*G(\gamma)L^{1/4}(dR)^{-1/2}\varepsilon^{3/4},$ and 
 \(h^*\) is the largest smoothing
radius allowed by the constraints. Thus, as in
Sec.~\ref{sec:tsybakov_str_cvx}, the fidelity is not chosen from
\(\varepsilon\) alone: its prefactor depends on the oracle-cost exponent
\(\gamma\), and it is also coupled to the smoothing choice \(h^*\). Notably, in this regime, batching does not improve the leading active trade-off between $Q$ and $\delta$, because everything depends on $Q=BN$. However, batching can indirectly help if the smoothing constraint, $h^* = H_3(B)$, is active.

The difference appears when \(B\) is optimized. In the large-batch
regime, the wall-clock
cost scales as $T(B)\propto B^{1-\gamma/2}$. Therefore \(\gamma=2\) governs batching. For \(0<\gamma\le2\), increasing \(B\)
does not improve the leading active trade-off, so one may take
\(B^*=1\) unless a larger batch is needed to relax the admissible
smoothing radius. For \(\gamma>2\), fidelity is too expensive; instead
of decreasing \(\delta\), the method increases the batch size until
\(\delta^*=1\). Thus, in this setting, the
expensive-fidelity regime is handled by overbatching rather than by increasing \(N\).

\subsection{Regularization and comparison}
\label{sec:tsybakov_regularization_comparison}

We compare two convex Tsybakov designs. \textsc{Regularized ZO-SGD} applies
the strongly convex batched baseline of Sec.~\ref{sec:tsybakov_str_cvx} to
\(f_\mu(x)=f(x)+\frac{\mu}{2}\|x-x_0\|^2\), with
\(\mu\asymp\varepsilon/R^2\). \textsc{Accelerated ZO-SGD} is the accelerated
convex method of Sec.~\ref{sec:tsybakov_acc_batching}. Let \(T_{\rm reg}\) and
\(T_{\rm acc}\) denote their wall-clock complexities. Details are deferred to
Appx.~\ref{app:regularization_comparison}. 

Prop.~\ref{prop:akhavan_walltime_batched} gives, for
\textsc{Regularized ZO-SGD},
$$T_{\rm reg} \asymp \begin{cases} \frac{d^{1+\frac{\gamma}{2}} L^{2-\gamma} L_\beta^{\frac{\gamma}{\beta-1}} R^{4-\gamma+\frac{\gamma}{\beta-1}}}{\varepsilon^{2+\frac{\gamma}{\beta-1}}}, & 0 < \gamma < 2, \\[3mm] \frac{d^2 L_\beta^{\frac{2}{\beta-1}} R^{\frac{2\beta}{\beta-1}}}{\varepsilon^{2+\frac{2}{\beta-1}}}, & \gamma \ge 2. \end{cases}$$
For \textsc{Accelerated ZO-SGD}, Sec.~\ref{sec:tsybakov_acc_batching} gives
$$T_{\rm acc} \asymp \begin{cases} \frac{d^{1+\frac{\gamma}{2}}}{\varepsilon^{\frac{1}{2}+\frac{3\gamma}{4}+\gamma s}}, \quad s:=\max\left\{\frac{3}{4},\frac{1}{\beta-1}\right\}, & 0 < \gamma < 2, \\[3mm] \frac{d^2}{\varepsilon^{2+\frac{2}{\beta-1}}}, & \gamma \ge 2. \end{cases}$$

Thus, for \(0<\gamma<2\), \textsc{Accelerated ZO-SGD} has the better
\(\varepsilon\)-scaling when \(\beta\le7/3\), or when \(\beta>7/3\) and $\gamma<\gamma_{\rm crit}(\beta):=\frac{3(\beta-1)}{3\beta-5}.
$
For \(\beta>7/3\) and \(\gamma_{\rm crit}(\beta)<\gamma<2\),
\textsc{Regularized ZO-SGD} is better. For \(\gamma\ge2\), both methods satisfy $T_{\rm reg}\asymp T_{\rm acc}
    \asymp d^2\varepsilon^{-2-\frac{2}{\beta-1}},
$
but \textsc{Accelerated ZO-SGD} has smaller depth, $N_{\rm acc}\asymp R\sqrt L\,\varepsilon^{-1/2},$ $ N_{\rm reg}\asymp dL^2R^4\varepsilon^{-2}.$

\section{End-to-end recipe: supervised PageRank with an MCMC oracle}
\label{sec:experiments}
PageRank~\citep{page1999pagerank} is a method for ranking the nodes of a graph by
importance---originally the pages of the web, ordered by how likely a user
randomly following links is to land on each one. Formally, it scores each node
by its mass under the stationary distribution of a random walk that, at each
step, follows an outgoing edge with probability $1-\alpha$ and with probability
$\alpha$ restarts from a fixed distribution; a node is ranked highly when many
important nodes link to it.
PageRank~\citep{bogolubsky2016learning} instead makes the walk \emph{tunable}:
the probability of following a given edge, and of restarting at a given node,
become functions of node and edge features---properties of a page, the strength
of a link---combined through a weight vector $\phi$. These weights are fitted on
training queries for which the desired ordering of nodes is known, so that the
stationary distribution of the tuned walk reproduces those orderings and
generalizes to unseen graphs.

The training set is organized into queries: each query $q\in Q$ is a separate ranking instance. Given a directed graph $\Gamma_q = (V_q, E_q)$ for each query $q\in Q$, a seed set $U_q\subset V_q$, node features $V_i^q\in\R^{m_1}$, edge features $E_{ij}^q\in\R^{m_2}$, and a restart probability $\alpha\in(0,1)$, the parameter $\phi=(\phi_1,\phi_2)\in \R^m$ (where $m = m_1 + m_2$) induces a restart distribution $\pi^0_q(\phi)$ supported on $U_q$ and a transition matrix $P_q(\phi)$ defined by

$$[\pi^0_q(\phi)]_i \;=\; \frac{\inner{\phi_1}{V_i^q}}{\sum_{l\in U_q} \inner{\phi_1}{V_l^q}}\,, \qquad [P_q(\phi)]_{ij} \;=\; \frac{\inner{\phi_2}{E_{ij}^q}}{\sum_{l : i\to l\in E_q} \inner{\phi_2}{E_{il}^q}},$$

for $i\in U_q$ and $i\to j\in E_q$, respectively. The stationary distribution $\pi_q(\phi)\in\R^{p_q}$ satisfies $\pi = \alpha\, \pi^0_q(\phi) + (1-\alpha)\, P_q(\phi)^\top \pi$. The loss is 
\begin{equation}
    \label{eq:loss}
f(\phi) \;=\; \frac{1}{|Q|}\sum_{q=1}^{|Q|} \big\| (A_q\,\pi_q(\phi))_+ \big\|_2^2,
\end{equation}
where $A_q\in\R^{r_q\times p_q}$ encodes the labelled-pair comparisons. The feasible set is the Euclidean ball $\Phi = \{\phi\in\R^m : \norm{\phi - \widehat\phi}_2\le R\}$, which is chosen to lie entirely within $\R^m_{++}$. The optimization problem is:$$\min_{\phi\in\Phi}\, f(\phi).$$We operate under the local convexity assumption introduced by \citet[Theorem~2]{bogolubsky2016learning}, which states that $\Phi$ can be chosen as a sufficiently small neighborhood of a local minimizer $\phi^\ast$ over which $f$ is convex.

\subsection{The MCMC zero-order oracle}
\label{sec:oracle}
Rather than computing \(\pi_q(\phi)\) by power iteration, we estimate it by independent samples from the random-surfer representation of PageRank. One sample is generated by drawing
\[
    v_0\sim \pi_q^0(\phi),\qquad
    K\sim \mathrm{Geom}(\alpha)-1,
\]
and then applying \(K\) transitions according to \(P_q(\phi)\). The returned
state \(v_K\) has distribution
\[
    \alpha\sum_{t\ge 0}(1-\alpha)^t
    \left(P_q(\phi)^\top\right)^t \pi_q^0(\phi)
    = \pi_q(\phi),
\]
which is exactly the solution of $    \pi_q(\phi)=\alpha\pi_q^0(\phi)
    +(1-\alpha)P_q(\phi)^\top\pi_q(\phi).$
Thus the sampler is the geometric-stopping random-surfer sampler for the
PageRank stationary distribution. The resulting empirical histogram $\widehat\pi_q$ satisfies, by the
Central Limit Theorem,
\begin{equation}\label{eq:CLT}
  \E\bigl[ \widehat\pi_q \big| \phi\bigr] = \pi_q(\phi),
  \qquad
  \E\bigl[ \norm{\widehat\pi_q - \pi_q(\phi)}^2_2 \big| \phi\bigr]
  \;\le\;
  \frac{C}{M},
\end{equation}
for an explicitly computable constant $C$ independent of $M$. Plugging into \eqref{eq:loss}and using the Lipschitz-continuity of
\(u\mapsto (Au)_+\), the plug-in loss estimator satisfies a mean-square
oracle guarantee
\[
    \mathbb E\!\left[(\widehat f(\phi;\delta)-f(\phi))^2\bigg|\phi\right]
    \le \delta^2,
    \qquad
    M=\left\lceil \frac{C'}{\delta^2}\right\rceil .
\]
Thus, up to the calibration constant \(C'\), the MCMC sample budget required
for a level \(\delta\) scales as $ M(\delta)\asymp \delta^{-2}.$
This gives the oracle-cost exponent
\[
    c(\delta)\asymp \delta^{-2}
    \quad\Longrightarrow\quad
    \gamma=2.
\]
The constant $C'$ is at most a polynomial in the supremum of the loss on
$\Phi$ and the spectral gap; in our experiments, we calibrate it once.
Independent calls at the same $\phi$ produce independent noises,
matching Def.~\ref{asm:tsybakov_noise}.

\subsection{Instantiating the accelerated ZO-SGD of Sec.~\ref{sec:tsybakov_acc_batching}}
\label{ssec:pr-method}

We solve the problem with the accelerated
ZO-SGD scheme of Sec.~\ref{sec:tsybakov_acc_batching} (labelled ``Method~B'' in the
figures below), instantiated on the Euclidean ball
$\Phi=\{\phi:\|\phi-\widehat\phi\|\le R\}$ (see Alg.~\ref{alg:pr}). 
The main experiments use a calibrated Tsybakov proxy oracle,
\[
    \hat f(\phi;\delta)=f(\phi)+\xi,\qquad
    \xi\sim \mathcal N(0,\delta^2),
\]
with the cost \(c(\delta)=\delta^{-\gamma}\). This model
isolates the optimization effect predicted by Sec.~\ref{sec:tsybakov_acc_batching} while matching the
variance--cost scaling of the MCMC PageRank sampler, for which
\(\gamma=2\).

\paragraph{Bias and variance ($\beta=2$).} We fix the smoothness level $\beta=2$ (Asm.~\ref{asm:hoelder_condition}). Since the proxy oracle is centered, its noise contributes to the second moment but not to the conditional bias of the two-point estimator. Prop.~\ref{prop:kernel_bias_variance} therefore gives \begin{equation}\label{eq:pr-bias-var} \|b_k(\phi_k)\|\;\lesssim\;L\,h, \qquad V_k(\phi_k)\;\lesssim\;G^2+\frac{m^2\,\delta_k^2}{h^2}. 
\end{equation} 
With batching, the stochastic contribution scales as \[ V_k^{(B)}(\phi_k)\;\lesssim\;G^2+\frac{m^2\,\delta_k^2}{B h^2}. \] The actual plug-in MCMC oracle of Sec.~\ref{sec:oracle} is used to motivate the cost exponent \(\gamma=2\); its finite-sample bias is included in the mean-square oracle guarantee.


\begin{algorithm}[t]
\caption{Accelerated ZO-SGD for supervised-PageRank learning with a tunable zeroth-order oracle.}
\label{alg:pr}
\begin{algorithmic}[1]
  \Require initial point $\phi_0\in\Phi$; iterations $N$; batch size $B$;
           step size $\eta$; smoothing radius $h$; fidelity $\delta$;
           kernel $K$; momentum cap $\bar\beta\in[0,1)$.
  \State Set $\phi_{-1}\leftarrow\phi_0$;\ $t_0\leftarrow1$.
  \For{$k=0,1,\dots,N-1$}
    \State $t_{k+1}\leftarrow\tfrac12(1+\sqrt{1+4t_k^2})$,\
           $\beta_k\leftarrow\min(\bar\beta,(t_k-1)/t_{k+1})$.
    \State $y_k\leftarrow\mathrm{Proj}_\Phi(\phi_k+\beta_k(\phi_k-\phi_{k-1}))$.
    \For{$i=1,\dots,B$}
      \State Sample $r_{k,i}\sim\mathrm{Unif}([-1,1])$,\
             $u_{k,i}\sim\mathrm{Unif}(\mathbb{S}^{m-1})$.
      \State Query the tunable zeroth-order oracle at
             $\phi_{k,i}^{+}=y_k+h r_{k,i}u_{k,i}$ and
             $\phi_{k,i}^{-}=y_k-h r_{k,i}u_{k,i}$,
             obtaining $\widehat f(\phi_{k,i}^{+};\delta)$ and
             $\widehat f(\phi_{k,i}^{-};\delta)$.
      \State Form the two-point kernel estimate
      \[
        g_{k,i}
        =
        \frac{m}{2h}
        \left(
          \widehat f(\phi_{k,i}^{+};\delta)
          -
          \widehat f(\phi_{k,i}^{-};\delta)
        \right)
        K(r_{k,i})u_{k,i}.
      \]
    \EndFor
    \State $\bar g_k\leftarrow\tfrac1B\sum_{i=1}^B g_{k,i}$;\quad
           $\phi_{k+1}\leftarrow\mathrm{Proj}_\Phi(y_k-\eta\,\bar g_k)$.
  \EndFor
  \State \Return $\phi_N$.
\end{algorithmic}
\end{algorithm}

\subsection{Wall-clock-optimal design}
\label{ssec:pr-wallclock}

We now specialize the accelerated Tsybakov design of Prop.~\ref{prop:acc_tsyb_design}
to $\gamma=2$; in its notation the ambient dimension is $d=m=m_1+m_2$.
Throughout, $\widetilde{\mathcal O}(\cdot)$ hides logarithmic factors.

\paragraph{Small-batch regime ($0<\gamma<2$).}
With $Q=BN$, Prop.~\ref{prop:acc_tsyb_design} gives $B^\star=1$ and
\begin{equation}\label{eq:pr-smallB}
  Q^\star\asymp\kappa\,m\,R\sqrt{L}\,\varepsilon^{-1/2}
            \Bigl(\tfrac{2+3\gamma}{2+\gamma}\Bigr)^{1/2},\qquad
  \delta^\star\asymp h\,L^{1/4}(mR)^{-1/2}\varepsilon^{3/4}G(\gamma),\qquad
  G(\gamma)=\frac{(2\gamma)^{1/2}(2+\gamma)^{1/4}}{(2+3\gamma)^{3/4}},
\end{equation}
with wall-clock cost
$T^\star\asymp\kappa\,m^{1+\gamma/2}R^{1+\gamma/2}L^{1/2-\gamma/4}
\varepsilon^{-1/2-3\gamma/4}h^{-\gamma}C_\gamma$.

\paragraph{Boundary regime ($\gamma>2$): over-batching.}
Here $T^\star(B,h)\propto B^{1-\gamma/2}$ is decreasing in $B$ while the
interior fidelity satisfies $\delta^\star(B,h)<1$; the optimum raises $B$ until
$\delta^\star=1$, calling the oracle at the cheapest fidelity and absorbing its
variance through batching:
\begin{equation}\label{eq:pr-overbatch}
  N^\star\asymp R\sqrt{L}\,\varepsilon^{-1/2},\qquad
  B^\star\asymp\kappa\,m^2 R\,L^{-1/2}\varepsilon^{-3/2}h^{-2},\qquad
  T^\star\asymp\kappa\,m^2 R^2\,\varepsilon^{-2}h^{-2}.
\end{equation}

\paragraph{If $\gamma=2$.}
Substituting $\gamma=2$ into either~\eqref{eq:pr-smallB}
or~\eqref{eq:pr-overbatch} gives the same leading scaling
$T^\star\asymp\varepsilon^{-2}h^{-2}$ for our fixed-size instance.
Equivalently, $T^\star(B,h)\propto B^{1-\gamma/2}\equiv B^0$ is flat in
$B$: the same leading wall-clock order can be achieved either with smaller
batches and larger sequential depth, or by distributing the work across larger
batches at a constant-level fidelity. The MCMC-motivated cost model
\(\gamma=2\) thus lies at the batching threshold predicted by
Sec.~\ref{sec:tsybakov_acc_batching}.


\begin{remark}[Why the cheapest fidelity is optimal]\label{rem:pr-flat-batch}

Under the Tsybakov proxy, the oracle variance scales as
\(\delta^2\), while the per-call cost scales as \(\delta^{-2}\). Hence
\(\delta^2 c(\delta)\) is constant in \(\delta\). Once the smoothing bias is kept subdominant, paying for higher fidelity does not improve the
variance--cost trade-off; the surplus budget is better spent on increasing
\(B\) or \(N\).
\end{remark}


\begin{remark}[Connection to the \cite{bogolubsky2016learning} bound]
\label{rem:pr-2016}
\citet[Thm.~2]{bogolubsky2016learning} obtain a
\(\widetilde{\mathcal O}(\varepsilon^{-1})\) arithmetic complexity for their
gradient-free method. This rate relies on their lower-level PageRank solver:
by Lemma~1 of \cite{bogolubsky2016learning}, an accuracy \(\delta\) for the
function value is obtained with an inner cost logarithmic in \(1/\delta\).
Thus their result belongs to the log-cost oracle regime discussed in
Sec.~\ref{sec:problem_setup}.

Our setting addresses a different lower-level implementation: a sampling-based
PageRank oracle for which achieving mean-square accuracy \(\delta^2\) requires
\(M(\delta)\asymp\delta^{-2}\) samples. In this polynomial-cost regime, simply
combining an outer zeroth-order method with a high-accuracy inner oracle can be
suboptimal, because the inner sampling cost becomes a leading term. The
wall-clock design above instead optimizes the iteration count, fidelity, and
batch size jointly; for the variance-dominated Tsybakov model at \(\gamma=2\),
this yields the scaling
\(\widetilde{\mathcal O}(\varepsilon^{-2}h^{-2})\).
\end{remark}

\subsection{Numerical validation}\label{ssec:pr-exp}

We implement Alg.~\ref{alg:pr} in NumPy on a synthetic supervised-PageRank instance
built in the spirit of \cite{bogolubsky2016learning}. We take \(|Q|=4\) queries with
\(p_q=20\) nodes each, \(m_1=m_2=4\) (\(m=8\)), random directed graphs of
expected out-degree \(3\), positive node/edge features, \(\alpha=0.15\),
and a label set \(\{1,\ldots,5\}\) defining the pair-comparison matrices
\(A_q\). The feasible ball is
\[
    \Phi=\{\phi:\|\phi-\mathbf 1\|\le 0.45\}\subset \mathbb R^m_{++},
\]
the kernel is \(K(r)=3r\), and \(\bar\beta=0.95\). The optimum
\(f^\star\approx 7.282\) (initial gap \(\approx 0.217\)) is precomputed by
deterministic projected gradient descent on the exact loss. We report the wall-clock cost at which the running mean
\[
    \bar f_k = \frac1T \sum_{j=k-T+1}^k f(\phi_j)
\]
over \(T=20\) iterates first drops below \(\varepsilon\). We additionally
include a sanity check for the actual MCMC sampler to confirm its
\(\delta^{-2}\) sample-budget scaling.

\paragraph{Experiment 1 (convergence).}
Using the Tsybakov proxy oracle, we run
\(N=2000\) iterations at \(\delta=0.05\), \(h=0.1\), \(B=1\), and
\(\eta=3\cdot 10^{-4}\). The method reaches residual \(\sim 10^{-3}\) in
\(\sim 10^3\) iterations and \(\sim 10^6\) cost units, with the
characteristic FISTA staircase (Fig.~\ref{fig:pr-conv}).

\begin{figure}[t]
  \centering
  \includegraphics[width=\textwidth]{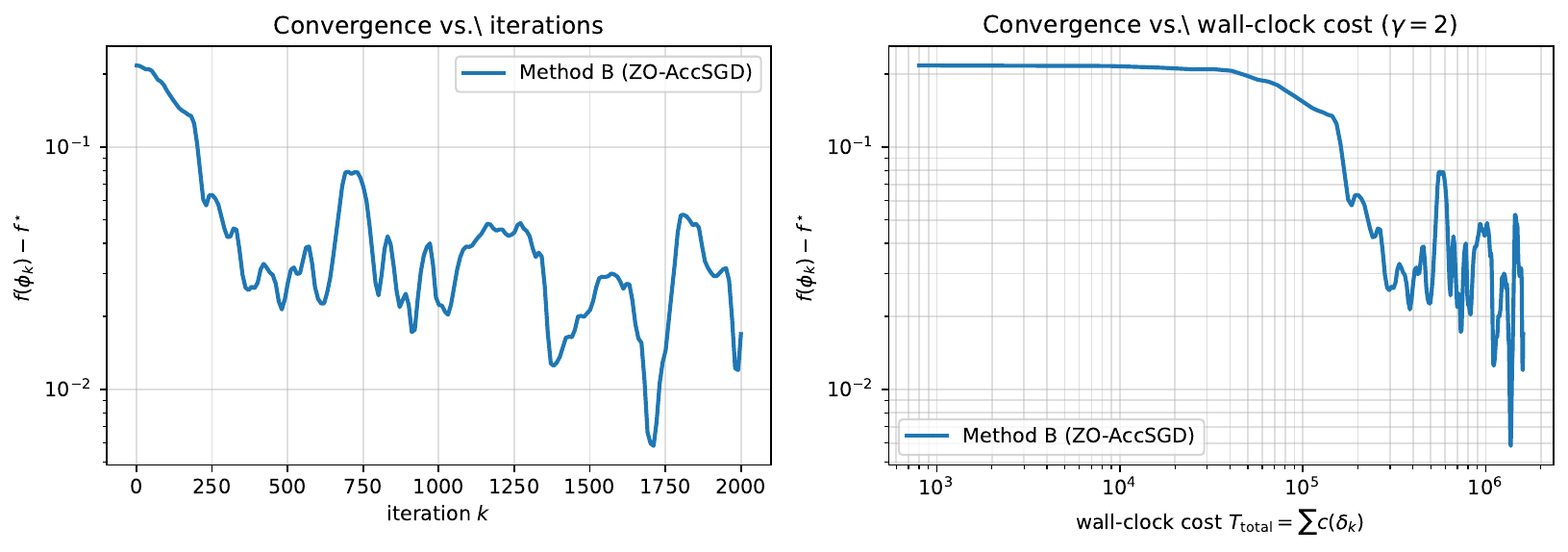}
  \caption{Convergence of the accelerated ZO-SGD on the synthetic
    supervised-PageRank instance. Left: trailing-mean residual vs.\ iteration
    $k$. Right: same vs.\ wall-clock cost $T_{\rm total}=\sum c(\delta_k)$ at
    $\gamma=2$.}
  \label{fig:pr-conv}
\end{figure}

\paragraph{Experiment 2 (the optimal fidelity $\delta^\star$).}
Sweeping $\delta$ on a log-grid from $5\!\cdot\!10^{-3}$ to $1$ at target
$\varepsilon=0.05$ (Fig.~\ref{fig:pr-delta}) reproduces the structure of
Sec.~\ref{ssec:pr-wallclock}: \textbf{(i)} for small $\delta$ the cost falls as
$T_{\rm total}\propto\delta^{-2}$ (the cost of the calls, iterations bounded);
\textbf{(ii)} the optimal $\delta^\star$ for $B=1$ sits just below the
divergence threshold ($\delta\approx0.24$), beyond which noise overwhelms
descent; \textbf{(iii)} raising $B$ to $4$ pushes that threshold right
($B=4$ still converges at $\delta=0.38$), the overbatching trade-off
of~\eqref{eq:pr-overbatch}.

\begin{figure}[t]
  \centering
  \includegraphics[width=0.7\textwidth]{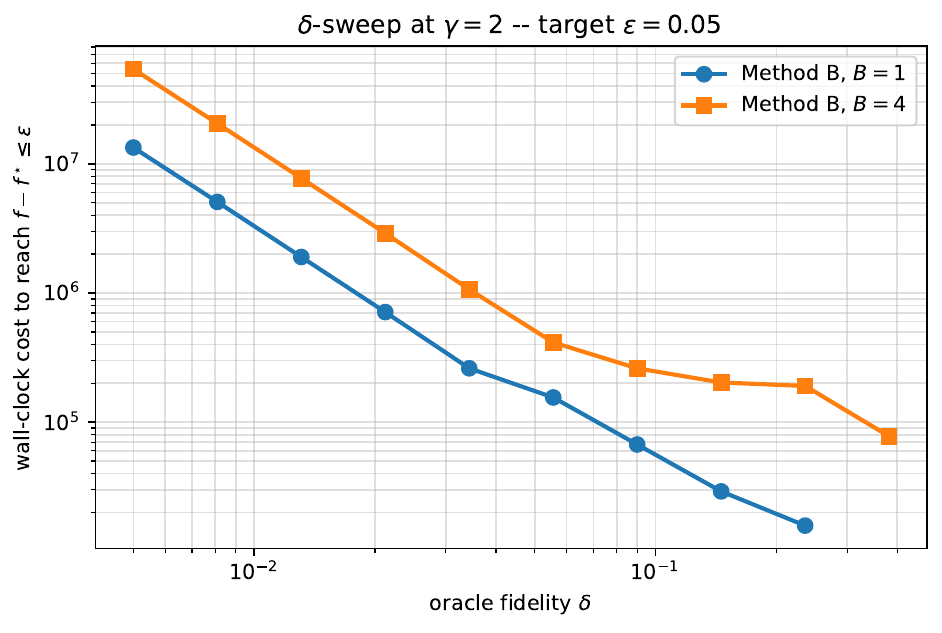}
  \caption{Wall-clock cost vs.\ fidelity $\delta$ at $\gamma=2$,
    $\varepsilon=0.05$. At small $\delta$ the cost scales as $\delta^{-2}$; the
    optimum sits just below the divergence threshold, matching the boundary
    behaviour of~\eqref{eq:pr-overbatch}. Larger $B$ shifts the threshold
    right.}
  \label{fig:pr-delta}
\end{figure}

\paragraph{Experiment 3 (flat batching at $\gamma=2$).}
Fixing $\delta\in\{0.02,0.1,0.5\}$ and sweeping $B\in\{1,\dots,32\}$ (with the
step size scaled by $\eta_B\propto1/(\delta/\sqrt B)$) yields
Fig.~\ref{fig:pr-B}: at $\delta=0.02$ batching is wasteful (cost grows linearly
in $B$); at $\delta=0.5$ the unbatched method does not converge at all, and the
cost is essentially flat across $B\in\{2,4,8,16\}$ --- the empirical signature
of $T^\star(B)\propto B^{1-\gamma/2}=B^0$ at $\gamma=2$.

\begin{figure}[t]
  \centering
  \includegraphics[width=0.7\textwidth]{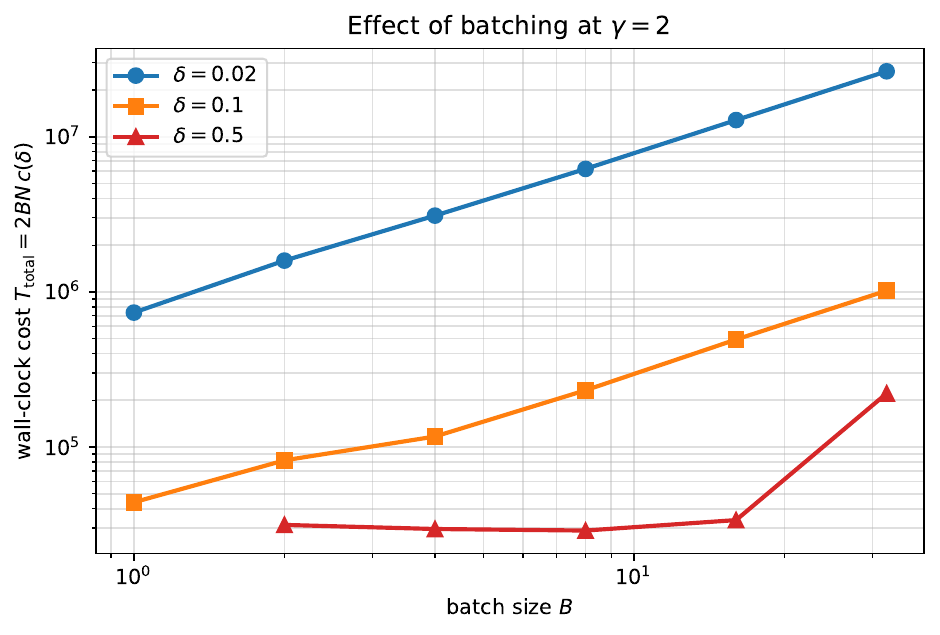}
  \caption{Effect of batch size $B$ at three fidelities, $\gamma=2$. At
    $\delta=0.02$ cost grows linearly in $B$. At $\delta=0.5$ (boundary),
    $B=1$ does not converge (no marker), and cost is flat-to-helpful for
    $B\in\{2,4,8,16\}$ --- overbatching at the predicted boundary.}
  \label{fig:pr-B}
\end{figure}

\paragraph{Experiment 4 (phase transition in $\gamma$).}
Replacing $c(\delta)=\delta^{-2}$ by $c(\delta)=\delta^{-\gamma}$ for
$\gamma\in\{0.5,1,1.5,2,2.5,3\}$ (noise distribution unchanged) and jointly
optimizing over $(B,\delta)$ gives Fig.~\ref{fig:pr-gamma}: at $\gamma=0.5$ the
optimum is interior ($\delta^\star\approx0.03$, $B^\star=1$), and as $\gamma$
grows it migrates to the boundary ($\delta^\star=1$, $B^\star=4$). The
migration occurs slightly below the asymptotic threshold $\gamma_{\rm crit}=2$:
this is expected, since the threshold is a small-$\varepsilon$, large-$N$
statement while our instance is mid-precision ($\varepsilon=0.05$ against an
initial gap of $0.22$). The qualitative prediction --- $\delta^\star$ jumping
to the boundary and $B^\star$ above $1$ as $\gamma$ rises --- is robustly borne
out.

\begin{figure}[t]
  \centering
  \includegraphics[width=\textwidth]{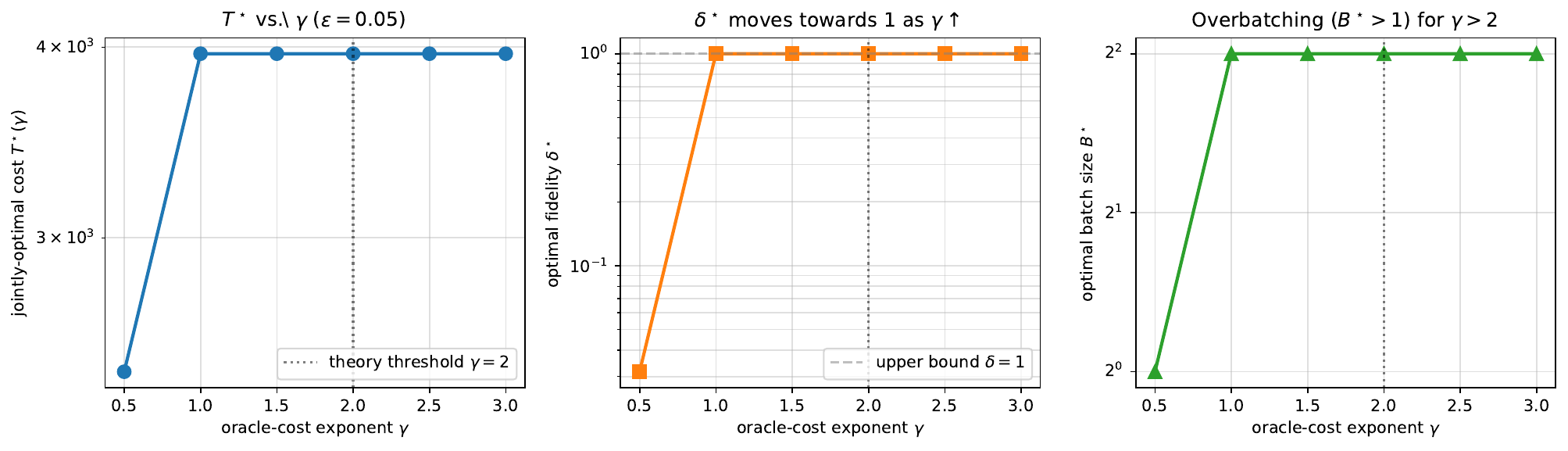}
  \caption{Phase transition in $\gamma$. Left: jointly optimized cost
    $T^\star(\gamma)$. Centre: $\delta^\star$ migrates from $\approx0.03$ at
    $\gamma=0.5$ to $1$ as $\gamma$ crosses the threshold. Right: $B^\star$
    jumps from $1$ to $4$ at the same transition (overbatching).}
  \label{fig:pr-gamma}
\end{figure}
\section{Discussion and limitations}
\label{sec:discussion}
\textbf{Discussion.} Our framework reveals that the oracle-cost exponent $\gamma$ behaves differently depending on the noise model. Under adversarial noise, it acts as a template selector: a phase transition at $\gamma=1$ makes non-accelerated GM wall-clock optimal over FGM, while optimal $N, B, \delta$ remain $\gamma$-independent. Under Tsybakov noise, $\gamma$ is a parameter tuner, setting a batching threshold ($\gamma=2$) that induces overbatching. Neither template dominates universally (Sec.~\ref{sec:tsybakov_regularization_comparison}); the choice depends on $\beta$ and $\gamma$ via $\gamma_{\text{crit}}(\beta)=\frac{3(\beta-1)}{3\beta-5}$. Finally, the gap $\Delta = p+\sigma\gamma-\rho(p+\gamma)$ in \eqref{eq:benifits} precisely diagnoses when adaptive fidelity yields polynomial (rather than constant) wall-clock gains, matching empirical observations.

\textbf{Limitations.} (i) We assume a polynomial oracle cost $c(\delta)\propto\delta^{-\gamma}$. While this drives the non-trivial trade-offs, real oracles may have fixed overheads or regime-dependent exponents. (ii) We assume zero switching costs for fidelity (Prop.~\ref{lemma:master}); explicit switching penalties remain unmodeled. (iii) Our closed-form bounds target the high-precision limit ($\varepsilon\to 0$); moderate $\varepsilon$ requires boundary analysis. (iv) Fidelity-separability is proven for first-order templates (GD, FGM, MD, SGD; Prop.~2) under both noise models; extending this to trust-region, surrogate, or second-order ZO methods is left open. (v) Optimal schedules require known problem constants ($L, \mu, \sigma_*$).Future research. Practical applications motivate extending our framework from absolute noise $\delta$ to composite noise models (supported empirically in Appx.~\ref{emperical gamma estimation}). 

\textbf{Future work} will explore composite adversarial noise, $|\hat{f}(x) - f(x)| \leq \beta (f(x) - f^*)^q + \delta$ \citep{vasin2025solving}, and Tsybakov noise under a strong growth condition, $\mathbb{E} (\hat{f}(x) - f(x))^2 \leq \rho (f(x) - f^*)^q + \delta^2$ \citep{vaswani2019fast}. Developing adaptive estimation for unknown problem constants is another key next step.

\bibliographystyle{unsrtnat}
\bibliography{references}

\appendix

\section{Experiments} \label{appendix:experiments}

\subsection{Estimating $\gamma$ for deep learning problems} \label{emperical gamma estimation}

\textbf{Hardware details:} Experiments were performed on Intel(R) Core(TM) i9-12900H: 14 cores,
32 GB RAM, NVIDIA RTX 3080 Ti Mobile GPU 16 GiB.

\textbf{Experiments details:} For most experiments below was used Adam~\citep{kingma2014adam} with learning rate $=0.001$.

We provide experiments for various deep learning problems. We will estimate the error as the calculation accuracy decreases, thus, at each iteration we will calculate loss using different floating point numbers -- $f_{\text{FP16}} (x^k)$ and $f_{\text{FP64}} f(x^k)$. Our goal is to estimate the $\gamma$ parameter in model $c(\delta) = \delta^{-\gamma}$, assuming that $f_{\text{FP64}}(x^k)$ real function value. Thus
\begin{equation*}
        \delta_k = | f_{\text{FP64}}(x^k) - f_{\text{FP16}}(x^k) |, \quad \delta^* = \max_{0 \leqslant k \leqslant N - 1} \delta_k, \quad \gamma_k = -\log_{\delta^*} (T_k^{\text{FP16}}),
\end{equation*}
where $T_k^{\text{FP16}}$ is the calculation time of loss using FP16.

In this experiments, the network is not trained via ZO methods, however, at each point of the trajectory, additional steps of the forward pass are performed. We will plot training curves - loss on train dataset, accuracy on test part and $\delta_k, T_k, \gamma_k$. Firstly we consider simple custom convolutional neural network and MNIST~\citep{deng2012mnist} dataset.

\begin{figure}[H]
    \centering
    \includegraphics[width=\linewidth]{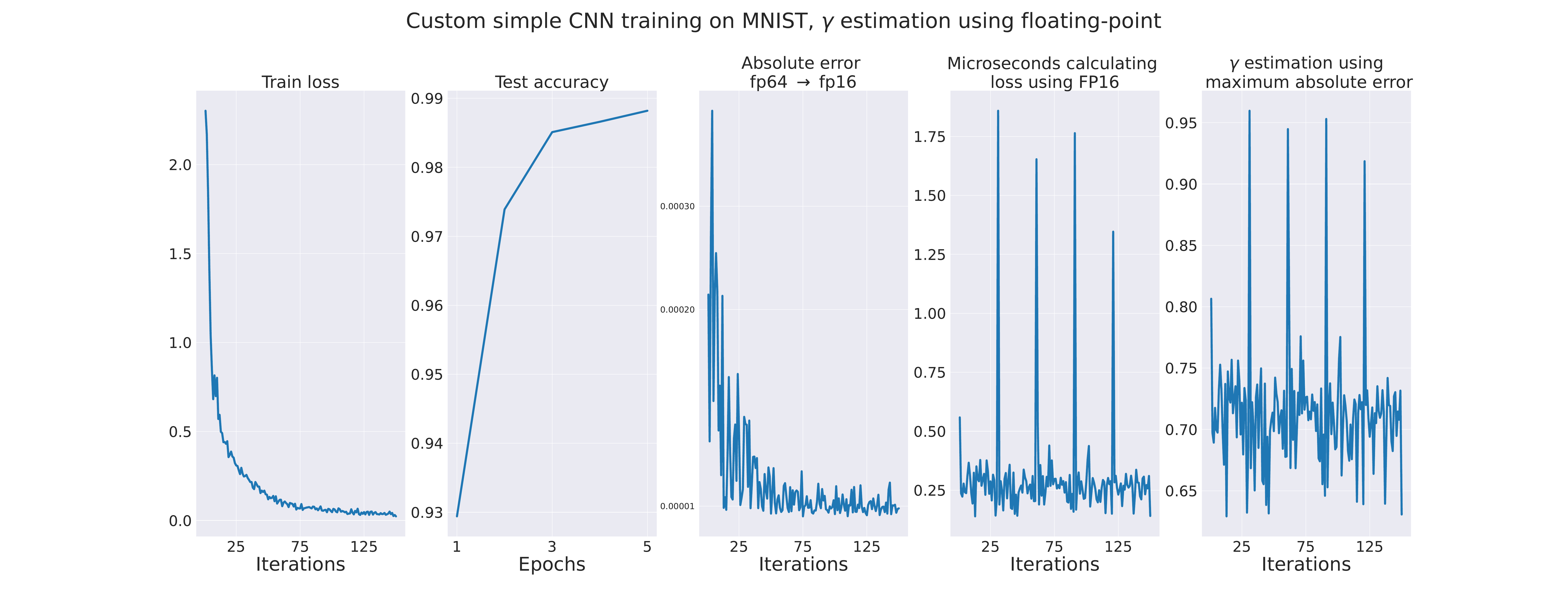}
    \caption{Simple custom CNN, MNIST, batch size -- 2048}
    \label{fig:mnist_gamma_estimation}
\end{figure}

On the Figure~\ref{fig:mnist_gamma_estimation} established, that $0.6 \leqslant \gamma \leqslant 1$. Note that the graph shows that the error decreases as the loss decreases, thus it is motivates to consider relative error $|\hat{f}(x) - f(x) | \leqslant \beta (f(x) - f^*)^{q}$. The following experiment uses the same idea for ResNet-18~\citep{he2016deep} and CIFAR-10~\citep{krizhevsky2009learning}.

\begin{figure}[H]
    \centering
    \includegraphics[width=\linewidth]{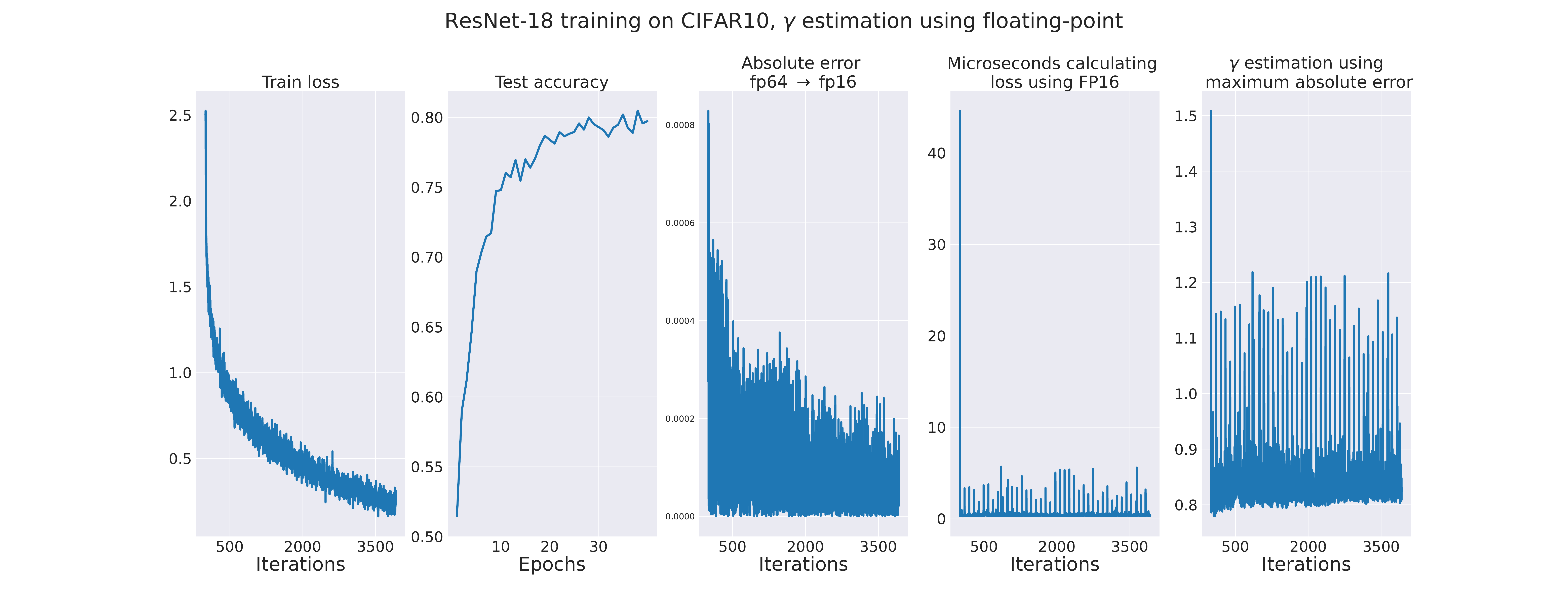}
    \caption{ResNet-18, CIFAR10, batch size -- 512}
    \label{fig:resnet18_gamma_estimation}
\end{figure}

Figure~\ref{fig:resnet18_gamma_estimation} demonstrates, that $0.7 \leqslant \gamma \leqslant 1.6$.


\subsection{Comparison of GM vs. FGM}
\label{exp:gm_fgm}
We compare GM and FGM on three strongly convex problems: (i) ridge regression on the UCI Superconductivity dataset with 81 features \footnote{\url{https://archive.ics.uci.edu/dataset/464/superconductivty+data}}, (ii) a synthetic quadratic objective, (iii) regularized logistic regression on a binarized version of the \texttt{digits} dataset \footnote{\url{https://scikit-learn.org/1.5/auto_examples/datasets/plot_digits_last_image.html}}. In all cases, the \red{...} demonstrates similar behaviour. So, below present only (i), and the rest is in Appx.~\ref{app:gm-fgm-details}.


In (i), we considered $f(w) = \frac{1}{2n}\|Xw-Y\|_2^2 + \frac{\lambda}{2}\|w\|_2^2$. We set the target accuracy  $\varepsilon = 10^{-8}$, and use the search of $\delta$ over 300 logarithmically spaced points between $10^{-15}$ and $10^{-2}$. Fig.~\ref{fig:gm_vs_fgm_superconductivity} presents the result; the empirical phase transition occurs around $\gamma \approx 0.9$, supporting the claim of Prop.~\ref{prop:igm_adversarial_uniform}
\begin{figure}[H]
    \centering
    \includegraphics[width=\textwidth]{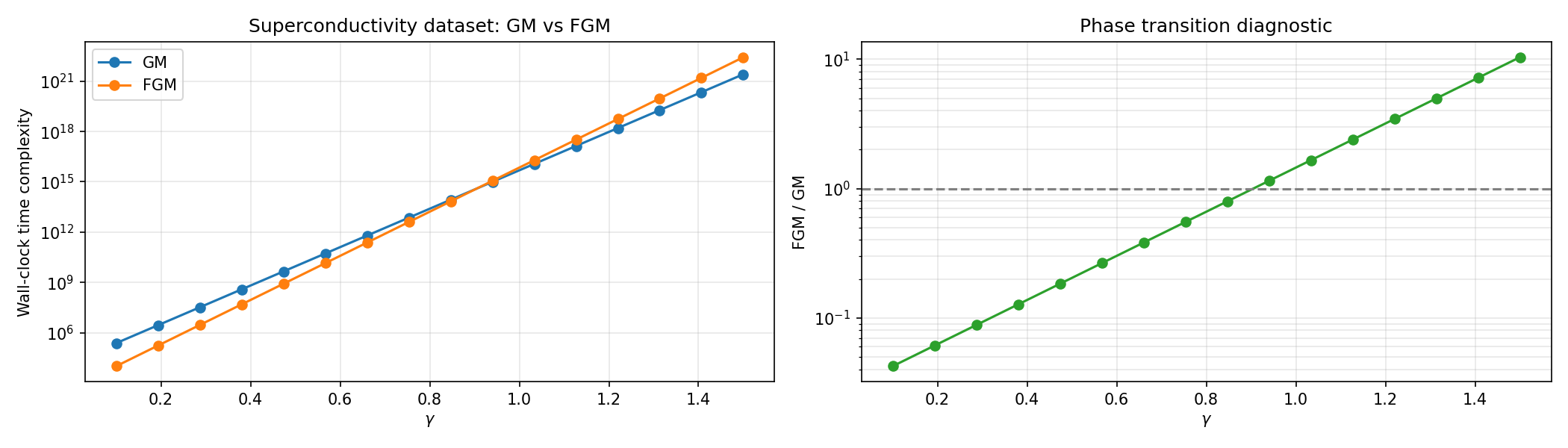}
    \caption{Comparison of the wall-clock time of GM and FGM on ridge regression over the UCI Superconductivity dataset. Left: optimal wall-clock time after tuning $\delta$. Right: the corresponding phase transition as a function of $\gamma$}
    \label{fig:gm_vs_fgm_superconductivity}
\end{figure}

\subsection{Additional details for the GM vs.\ FGM experiments}
\label{app:gm-fgm-details}

We compared standard gradient descent (GM) and the fast gradient method (FGM, implemented as STM in the code) on three problems: a synthetic quadratic objective, ridge regression, and regularized logistic regression. The performance criterion was the wall-clock time required to reach a prescribed target accuracy $\varepsilon$.

In all experiments, the parameter $\gamma$ was varied over 10 uniformly spaced values between $0.1$ and $1.5$. For each value of $\gamma$, we selected a single fixed noise level $\delta$ used throughout the whole run. The value of $\delta$ was chosen by grid search over 300 logarithmically spaced points between $10^{-15}$ and $10^{-2}$, with the search truncated at the corresponding theoretical threshold: $\delta \le \mu \varepsilon /(Ld)$ for GM and $\delta \le \mu^{3/2}\varepsilon /(L^{3/2}d)$ for FGM. For every candidate $\delta$, we ran the method and recorded the time needed to reach accuracy $\varepsilon$; the best such wall-clock time was reported.

For the synthetic quadratic experiment, the objective was $f(x) = \frac{1}{2}x^\top Qx$, where $Q$ is diagonal with eigenvalues ranging from $\mu=1$ to $L=10$, and the dimension is $d=20$. We used coordinate-wise forward finite differences, with gradient estimate $\tilde{\nabla}f(x)_i = \bigl(\hat f(x + h e_i) - \hat f(x)\bigr)/h$, where $h = 2\sqrt{\delta/L}$ was chosen according to the theory. The target accuracy was $\varepsilon = 10^{-6}$. The noisy oracle used a deterministic adversarial perturbation with sign given by $\operatorname{sign}(\sin(\sum_i x_i))$. The results are shown in Figure~\ref{fig:gm_vs_fgm_quadratic}.

For ridge regression, we considered $f(w) = \frac{1}{2n}\|Xw-Y\|_2^2 + \frac{\lambda}{2}\|w\|_2^2$ on the UCI Superconductivity dataset with 81 features.\footnote{\url{https://archive.ics.uci.edu/dataset/464/superconductivty+data}} The target accuracy was $\varepsilon = 10^{-8}$, and the search over $\delta$ used the same logarithmic grid as above. The results are shown in Figure~\ref{fig:gm_vs_fgm_superconductivity}; the empirical phase transition occurs around $\gamma \approx 0.9$.

For regularized logistic regression, we used $f(x) = \frac{1}{n}\sum_{i=1}^{n}\log\bigl(1+\exp(-y_i a_i^\top x)\bigr) + \frac{\lambda}{2}\|x\|^2$ on the \texttt{digits} dataset from \texttt{scikit-learn}, converted into a binary classification problem by grouping digits $0$--$4$ into one class and digits $5$--$9$ into the other.\footnote{\url{https://scikit-learn.org/1.5/auto_examples/datasets/plot_digits_last_image.html}} The target accuracy was $\varepsilon = 10^{-4}$, and $\delta$ was tuned over the same grid. The results are shown in Figure~\ref{fig:gm_vs_fgm_log_reg}; in this case, the phase transition appears around $\gamma \approx 0.6$.

\begin{figure}[H]
    \centering
    \includegraphics[width=\textwidth]{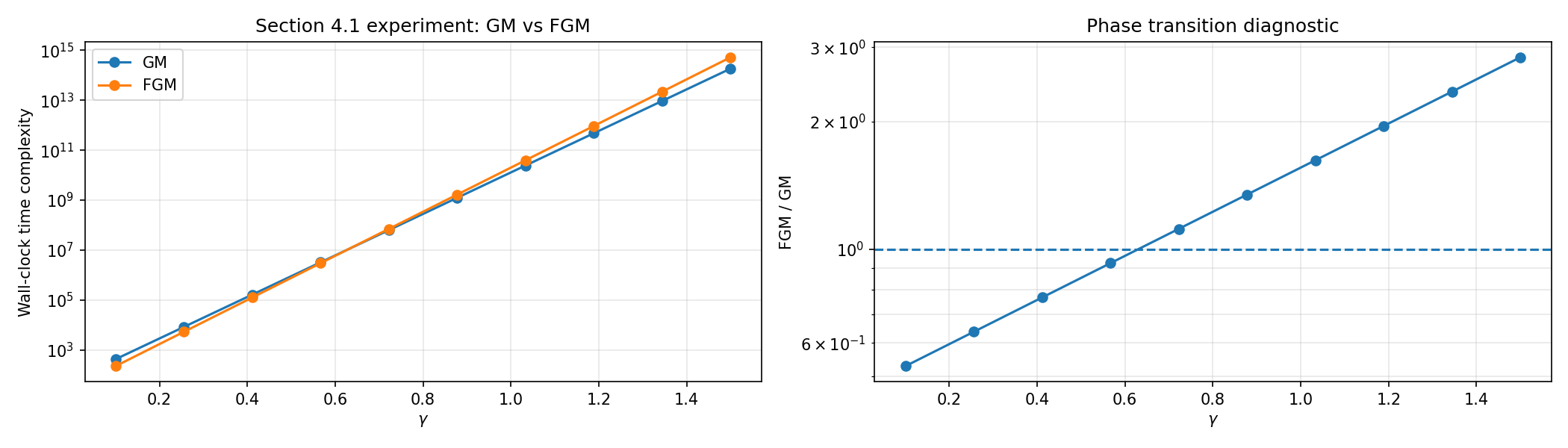}
    
    \caption{Comparison of GM and FGM on the quadratic objective.}
    \label{fig:gm_vs_fgm_quadratic}
\end{figure}

\begin{figure}[H]
    \centering
    \includegraphics[width=\textwidth]{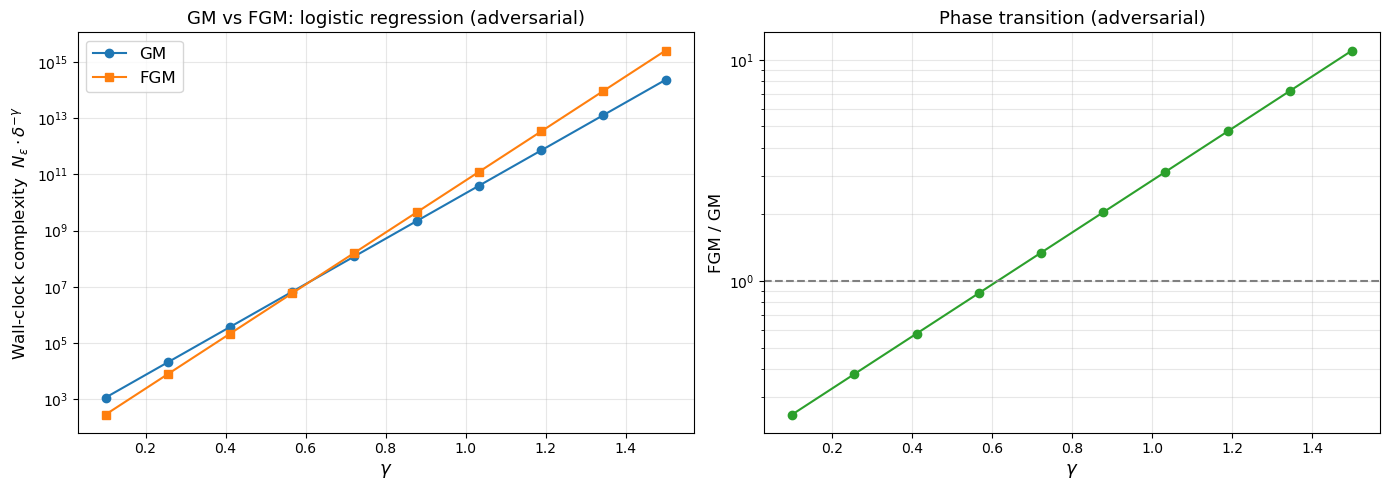}
    \caption{Comparison of the wall-clock time of GM and FGM on regularized logistic regression.}
    \label{fig:gm_vs_fgm_log_reg}
\end{figure}

\subsection{Comparison of fidelity allocation strategies}
\label{sec:exp_comparison}
This section illustrates the results from Sec.~\ref{sec:master_lemma}. We consider FGM on $f(x) = \frac{\mu}{2}x_1^2 + \frac{L}{2}x_2^2$, under assumption of oracle with adversarial noise (Assumption~\ref{asm:adv_noise}). We
compare two schedules: $\delta_k = \delta_0$ and $\delta_k = \delta_{k-1} \left(1 + \frac{\mu}{4L} + \sqrt{\frac{\mu}{2L}}\right)^{-\frac{1}{p+\gamma}}$ (motivated by Lemma~\ref{lemma:master}). Fig.~\ref{fig:adaptive_strategies_comparison} shows that the adaptive schedule consistently outperforms the fixed-noise baseline. Specifically, the optimal $\delta_0$ is larger for the adaptive strategy, and this strategy achieves lower wall-clock time across the range of target accuracies we tested. However it's important to note that the difference is not so significant which is consistent with the theory (Prop.~\ref{prop:geometric_schedule_gain}). Full details of the noise model and the tuning procedure are given in Appx.~\ref{app:delta-strategies-details}.

\begin{figure}[H]
    \centering
    \includegraphics[width=\textwidth]{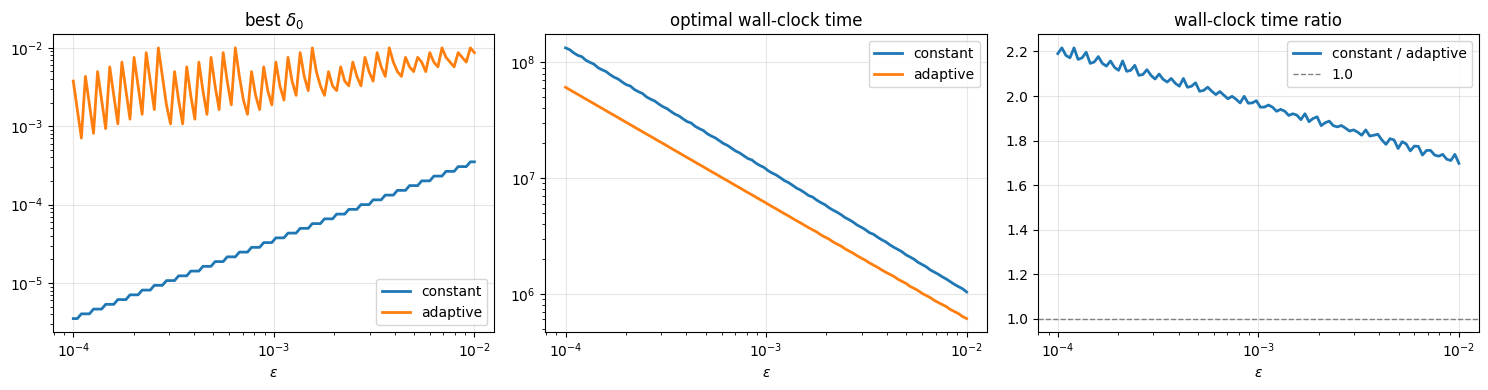}
    
    \caption{Comparison of fidelity allocation strategies on each iteration}
    \label{fig:adaptive_strategies_comparison}
\end{figure}

\subsection{Additional details for the scheduling experiments}
\label{app:delta-strategies-details}

In the second group of experiments, we considered the two-dimensional quadratic objective $f(x)=\frac{\mu}{2}x_1^2+\frac{L}{2}x_2^2$. The gradient was approximated using central finite differences:
\[
\tilde{\nabla}f(x)=\left(
\frac{\hat f(x+h e_1)-\hat f(x-h e_1)}{2h},
\frac{\hat f(x+h e_2)-\hat f(x-h e_2)}{2h}
\right).
\]

The noise model was greedy and adversarial. At every iteration, the next point $x_{k+1}$ depends on four noisy oracle values, corresponding to the points $x+h e_1$, $x-h e_1$, $x+h e_2$, and $x-h e_2$. Since each noisy value can vary independently within the interval $[f(\cdot)-\delta,\, f(\cdot)+\delta]$, the adversary has four degrees of freedom and chooses the perturbation that maximizes $f(x_{k+1})$ at the next iterate.

Because $x_{k+1}$ is obtained from the finite-difference estimate through an affine transformation, the resulting objective is a convex continuous function of these four perturbations. Therefore, its maximum over the box $[-\delta,\delta]^4$ is attained at a vertex, so it is sufficient to enumerate all $16$ vertices and choose the worst one.

Figure~\ref{fig:noise_convergence} compares the convergence behavior of GM and FGM under this worst-case noise model, under uniformly random noise from $[-\delta,\delta]$, and in the noiseless setting.

\begin{figure}[!t]
    \centering
    \includegraphics[width=\textwidth]{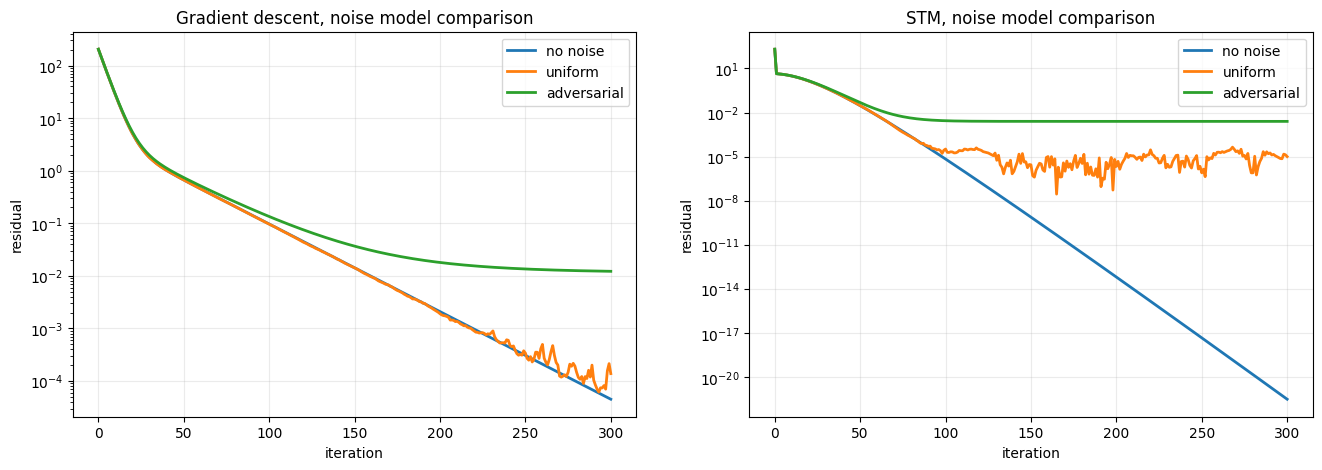}
    \caption{Comparison of convergence in function residual under different noise models.}
    \label{fig:noise_convergence}
\end{figure}

For this noise model, we considered the following strategies for choosing $\delta_k$ at each step of FGM: $\delta_k = \delta_0$, and
\[
\delta_k = \delta_{k-1} \left(1 + \frac{\mu}{4L} + \sqrt{\frac{\mu}{2L}}\right)^{-\frac{1}{1+\gamma}},
\]
with $\delta_0$ selected over a logarithmic grid from $10^{-8}$ to $10^{-2}$.

We now introduce an adaptive strategy motivated by the master lemma. Recall that the convergence rate of the accelerated method (FGM) \citep{vasin2023accelerated} with a noisy gradient, where the gradient is approximated using one of the schemes
\[
\tilde{\nabla}f(x)_i = \frac{\hat{f}(x+he_i)-\hat{f}(x)}{h}
\qquad \text{or} \qquad
\tilde{\nabla}f(x)_i = \frac{\hat{f}(x+he_i)-\hat{f}(x-he_i)}{2h},
\]
for $N$ iterations with varying noise levels, has the form
\[
f(x)-f(x^*) \lesssim LR^2\exp\left(-\frac{1}{2}\sqrt{\frac{\mu}{2L}}N\right) + \sum_{i=0}^N \frac{dL}{\mu} \left(1 + \frac{\mu}{4L} + \sqrt{\frac{\mu}{2L}}\right)^{i-N}\delta_i.
\]
According to the master lemma,
\[
\delta_i^*(N) = \left( \frac{\epsilon}{A_N} \right) \alpha_i^{-\frac{1}{1+\gamma}},
\quad
A_N := \sum_{j=1}^N \alpha_j^{\frac{\gamma}{1+\gamma}}.
\]
Therefore, in the accelerated setting,
\[
\frac{\delta_{i+1}^*(N)}{\delta_i^*(N)}
=
\left(\frac{\alpha_i}{\alpha_{i+1}}\right)^{\frac{1}{1+\gamma}}
=
\left(\frac{1}{1 + \frac{\mu}{4L} + \sqrt{\frac{\mu}{2L}}}\right)^{\frac{1}{1+\gamma}}.
\]

This relation motivates the adaptive strategy
\[
\delta_k = \delta_{k-1} \left(1 + \frac{\mu}{4L} + \sqrt{\frac{\mu}{2L}}\right)^{-\frac{1}{1+\gamma}}.
\]
As before, we tune $\delta_0$ over the same logarithmic grid. Figure~\ref{fig:adaptive_strategies_comparison}
shows three plots. The left panel illustrates how the optimal initial noise depends on the target accuracy. As expected, when $\delta$ decreases across iterations, larger values of the initial noise level $\delta_0$ become preferable. The middle panel shows that the adaptive strategy outperforms the strategy with fixed $\delta$ at every iteration. Finally, the right panel reports the ratio between the runtime of the strategy $\delta_k = \delta_0$ and that of the adaptive strategy. At the same time, this plot indicates that the gain from adaptivity remains rather modest, which is consistent with the theory:

\begin{proposition}[Logarithmic gain for geometrically decaying coefficients]
\label{prop:geometric_schedule_gain}
Fix $N$ and assume that the fidelity-separable channel has coefficients
\[
    \alpha_k(N,\Theta)=Cq^{N-k}, \qquad k=1,\ldots,N,
\]
where $C>0$ and
\[
    q=\frac{1}{1+\frac{\mu}{4L}+\sqrt{\frac{\mu}{2L}}}\in(0,1).
\]
Let
\[
    T_{\rm total}^{\rm var}(N)
\]
denote the optimal wall-clock time under the time-varying fidelity schedule of
Prop.~\ref{lemma:master}, and let
\[
    T_{\rm total}^{\rm unif}(N)
\]
denote the optimal wall-clock time under the uniform restriction
$\delta_1=\cdots=\delta_N$. If
\[
    N \asymp \sqrt{\frac{L}{\mu}}\,
    \log\!\left(\frac{LR^2}{\varepsilon}\right),
\]
then, in the high-precision regime,
\[
    \frac{T_{\rm total}^{\rm unif}(N)}
         {T_{\rm total}^{\rm var}(N)}
    =
    O\!\left(
    \log\!\left(\frac{LR^2}{\varepsilon}\right)
    \right).
\]
Equivalently, the time-varying fidelity schedule improves the best uniform
schedule by at most a logarithmic factor for this geometric coefficient profile.
\end{proposition}

\begin{proof}
For fixed $N$, Prop.~\ref{lemma:master} gives
\[
    T_{\rm total}^{\rm var}(N)
    =
    \epsilon^{-\gamma/p}
    \left(
        \sum_{k=1}^N
        \alpha_k^{\frac{\gamma}{p+\gamma}}
    \right)^{\frac{p+\gamma}{p}},
\]
where $\epsilon=\varepsilon-\mathcal E_0(N,\Theta)$ is the accuracy budget
allocated to the fidelity channel. By Cor.~\ref{corr:uniform}, the best uniform
schedule satisfies
\[
    T_{\rm total}^{\rm unif}(N)
    =
    N\epsilon^{-\gamma/p}
    \left(
        \sum_{k=1}^N\alpha_k
    \right)^{\gamma/p}.
\]
Hence the common factor $\epsilon^{-\gamma/p}$ cancels in the ratio. Put
\[
    s:=\frac{\gamma}{p+\gamma}\in(0,1).
\]
Since $\alpha_k=Cq^{N-k}$, the constant $C$ also cancels, and after changing
indices $r=N-k$ we obtain
\[
    \frac{T_{\rm total}^{\rm unif}(N)}
         {T_{\rm total}^{\rm var}(N)}
    =
    N
    \frac{
    \left(\sum_{r=0}^{N-1}q^r\right)^{\gamma/p}
    }{
    \left(\sum_{r=0}^{N-1}q^{sr}\right)^{(p+\gamma)/p}
    }.
\]
Using the geometric-sum identities,
\[
    \sum_{r=0}^{N-1}q^r
    =
    \frac{1-q^N}{1-q},
    \qquad
    \sum_{r=0}^{N-1}q^{sr}
    =
    \frac{1-q^{sN}}{1-q^s},
\]
we get
\[
    \frac{T_{\rm total}^{\rm unif}(N)}
         {T_{\rm total}^{\rm var}(N)}
    =
    N
    \frac{
        (1-q^N)^{\gamma/p}
        (1-q^s)^{(p+\gamma)/p}
    }{
        (1-q)^{\gamma/p}
        (1-q^{sN})^{(p+\gamma)/p}
    }.
\]
For the chosen value of $q$,
\[
    1-q
    =
    \frac{\frac{\mu}{4L}+\sqrt{\frac{\mu}{2L}}}
         {1+\frac{\mu}{4L}+\sqrt{\frac{\mu}{2L}}}
    =
    \Theta\!\left(\sqrt{\frac{\mu}{L}}\right)
\]
as $L/\mu\to\infty$. Moreover, since $s\in(0,1)$ is fixed,
\[
    1-q^s=\Theta(1-q)
    =
    \Theta\!\left(\sqrt{\frac{\mu}{L}}\right).
\]
Finally,
\[
    N \asymp \sqrt{\frac{L}{\mu}}
    \log\!\left(\frac{LR^2}{\varepsilon}\right)
\]
implies that $N(1-q)\asymp \log(LR^2/\varepsilon)$, so in the
high-precision regime $q^{sN}$ is bounded away from $1$ and
$(1-q^N)^{\gamma/p}\le 1$. Therefore,
\[
    \frac{T_{\rm total}^{\rm unif}(N)}
         {T_{\rm total}^{\rm var}(N)}
    =
    O\!\left(
        N
        \frac{(1-q)^{(p+\gamma)/p}}
             {(1-q)^{\gamma/p}}
    \right)
    =
    O\!\left(N(1-q)\right).
\]
Substituting $1-q=\Theta(\sqrt{\mu/L})$ yields
\[
    \frac{T_{\rm total}^{\rm unif}(N)}
         {T_{\rm total}^{\rm var}(N)}
    =
    O\!\left(
    \log\!\left(\frac{LR^2}{\varepsilon}\right)
    \right),
\]
as claimed.
\end{proof}

\section{Some definitions}

\begin{definition}[Smoothing kernel]
\label{def:smoothing_kernel} We assume that a kernel function $K : [-1, 1] \to \mathbb{R}$ is such that for a fixed $\beta \ge 2$: $\int_{-1}^1 K(r) \, dr = 0$, $ \int_{-1}^1 r K(r) \, dr = 1$, and $\int_{-1}^1 |r|^\beta |K(r)| \, dr < \infty$.
Additionally, if $\beta > 2$, we require $\int_{-1}^1 r^j K(r) \, dr = 0$ for all $j = 2, \ldots, \lfloor \beta \rfloor$.
Further, we set
\begin{equation}
    \label{eq:kappa}
    \kappa_\beta = \int |u|^\beta |K(u)| \, du, \quad \kappa = \int |K(u)|^2 \, du.
\end{equation}
\end{definition}

\section{Proofs from Sec.~\ref{sec:master_lemma}}
\label{appx:proofs_master}
\begin{proof}[Proof of Master Lemma~\ref{lemma:master}]
Since $\Phi(N, \Theta)$ does not depend on $\delta_1, \dots, \delta_N$, our goal is to minimize $\sum \delta_i^{-\gamma}$ subject to $\sum \alpha_i \delta_i^p = \varepsilon$. The Lagrangian is
$$ \mathcal{L} = \sum_{i=1}^N \delta_i^{-\gamma} + \lambda \left( \sum_{i=1}^N \alpha_i \delta_i^p - \varepsilon \right). $$
The stationarity condition $\frac{\partial \mathcal{L}}{\partial \delta_i} = 0$:
\begin{equation}
\label{eq:delta_opt_master}
-\gamma \delta_i^{-\gamma-1} + \lambda p \alpha_i \delta_i^{p-1} = 0 \implies \delta_i = \left(\frac{\gamma}{\lambda \alpha_i p}\right)^{\frac{1}{p+\gamma}}.
\end{equation}
Substituting the equation for $\delta_i$ into the $\lambda$-constraint, and then substituting $\delta_i^*$ into the objective function, yields \eqref{eq:time_adapt}.  
\end{proof}



\begin{lemma}
    \label{lemma:master_aux}
    Let $c(\delta) := \max \{\delta^{-\gamma}, 1\}$ for some $\gamma > 0$ and $\delta > 0$.
    Fix $N$ and assume that the bound $\mathcal{E}(N,\{\delta_k\}_{k\in [N]}, \Theta)$ admits a decomposition
    \[
    \mathcal{E}(N,\{\delta_k\}_{k\in [N]}, \Theta) = \Phi(N, \Theta) +
    \sum^N_{k=1}\alpha_k \delta^p_k
    \]
    with $\alpha_k$ being  depend on the choice of gradient estimator, noise model and optimization algorithm. Let $\mathcal{E}(N,\{\delta_k\}_{k\in [N]}, \Theta) \lesssim \varepsilon$.
    Denote
    \[
    K(N, m) := \left(\frac{\varepsilon - \sum^N_{i=m+1} \alpha_i}{\sum_{i=1}^m \alpha^{-\frac{\gamma}{p+\gamma}} } \right)^{\frac{1}{p}}.
    \]
    Then we get the following solution
    \[
    \delta_k(N) = 
    \begin{cases}
        K(N, m^*) \alpha^{-\frac{1}{p+\gamma}}_k, &~\text{if}~k\le m^*\\
        1, &~\text{if}~k > m^*,
    \end{cases}
    \]
    where $m^*$ is the smallest $m$ satisfying
    \[
    K(N, m)\alpha^{-\frac{1}{p+\gamma}}_m < 1 \quad \text{and} \quad K(m)\alpha^{-\frac{1}{p+\gamma}}_{m+1} \ge 1.
    \]
\end{lemma}
\begin{proof}
The Lagrangian is
$$ \mathcal{L} = \sum_{i=1}^N \delta_i^{-\gamma} + \lambda_0 \left( \sum_{i=1}^N \alpha_i \delta_i^p - \varepsilon \right) + \sum^N_{i=1} \lambda_i( \delta_i - 1) - \sum^N_{i=1} \tilde{\lambda}_i \delta_i. $$

If $0 < \delta_i < 1$, \eqref{eq:delta_opt_master} holds. Let 
\[
K:= \left(\frac{\gamma}{\lambda_0 p}\right)^{\frac{1}{p+\gamma}}.
\]
If
$\delta_i = 1$, then $\tilde{\lambda}_i = 0$ and 
$$-\gamma + \lambda_0\alpha_i p + \lambda_i = 0
\;\Rightarrow\;
\lambda_i = \gamma-\lambda_0\alpha_i p \ge 0
\;\Rightarrow\;
\lambda_0\alpha_i p \le \gamma.$$

Next, let us rearrange $\delta_{i}$ in increasing order, and let $m$ be s.t. $\delta_{m} \in (0, 1)$ and $\delta_{m+1} = 1$ ($m$ is not known so far). Thus,
$$ K^p \sum_{i=1}^m \alpha_i^{1 -\frac{p}{p+\gamma}}
+
\sum_{i=m+1}^N \alpha_i
\quad \le \varepsilon
\;\Rightarrow\; K(m) := \left(\frac{\varepsilon - \sum^N_{i=m+1} \alpha_i}{\sum_{i=1}^m \alpha^{-\frac{\gamma}{p+\gamma}} } \right)^{\frac{1}{p}}
$$
Finally, using the explicit form of $\alpha_i$, we have to find the smallest $m$, s.t.
\[
K(m)\alpha^{-\frac{1}{p+\gamma}}_m < 1 \quad \text{and} \quad K(m)\alpha^{-\frac{1}{p+\gamma}}_{m+1} \ge 1.
\]
\end{proof}

\begin{corollary}[Large $N$ full version]
\label{cor:large_N_explicit}
Under Prop.~\ref{lemma:master}, assume that, as $N\to\infty$,
\[
\mathcal E_0(N,\Theta)\asymp N^{-\beta},
\qquad
A_N(\gamma):=
\sum_{j=1}^N \alpha_j(N,\Theta)^{\gamma/(p+\gamma)}
\asymp (\gamma)N^\rho,
\]
with $\beta>0$ and $\rho>0$. Let
\[
D_\gamma:=\rho(p+\gamma)+\beta\gamma .
\]
Then the continuous interior optimizer is
\[
N^*
=
\left(
\frac{C_0D_\gamma}
{\rho(p+\gamma)\varepsilon}
\right)^{1/\beta},
\]
and
\[
\delta_k^*
=
\left(
\frac{\beta\gamma\,\varepsilon}
{D_\gamma C_A(\gamma)(N^*)^\rho}
\right)^{1/p}
\alpha_k(N^*,\Theta)^{-1/(p+\gamma)} .
\]

For the uniform schedule $\delta_k\equiv\delta$, assume that
\[
\widetilde A_N:=
\sum_{j=1}^N\alpha_j(N,\Theta)
\sim \widetilde C_A N^\sigma,
\]
where $\sigma\in\mathbb R$ and $p+\sigma\gamma>0$. Let
\[
\widetilde D_\gamma:=p+\sigma\gamma+\beta\gamma .
\]
Then the continuous interior optimizer under the uniform restriction is
\[
\widetilde N^*
=
\left(
\frac{C_0\widetilde D_\gamma}
{(p+\sigma\gamma)\varepsilon}
\right)^{1/\beta},
\quad
\delta^*
=
\left(
\frac{\beta\gamma\,\varepsilon}
{\widetilde D_\gamma\widetilde C_A(\widetilde N^*)^\sigma}
\right)^{1/p}.
\]
For integer horizons, the displayed values are continuous relaxations; the
integer optimizer is obtained by minimizing the same one-dimensional objective
over feasible $N\in\mathbb N$, which does not affect the leading
$\varepsilon$-scaling.
\end{corollary}

\begin{proof}
We first consider the nonuniform case. For a fixed feasible horizon $N$, define
\[
\eta_N:=\varepsilon-\mathcal E_0(N,\Theta)
=
\varepsilon-C_0N^{-\beta}.
\]
By Prop.~\ref{lemma:master}, the fixed-$N$ optimal fidelity allocation is
\[
\delta_k^*(N)
=
\left(
\frac{\eta_N}{A_N(\gamma)}
\right)^{1/p}
\alpha_k(N,\Theta)^{-1/(p+\gamma)},
\]
and the corresponding wall-clock cost is
\[
T^*(N)
\asymp
\eta_N^{-\gamma/p}A_N(\gamma)^{(p+\gamma)/p}.
\]
Using $A_N(\gamma)\sim C_A(\gamma)N^\rho$, the continuous relaxation is
\[
T^*(N)
\asymp
(\varepsilon-C_0N^{-\beta})^{-\gamma/p}
\bigl(C_A(\gamma)N^\rho\bigr)^{(p+\gamma)/p}.
\]
Equivalently,
\[
\log T^*(N)
=
-\frac{\gamma}{p}\log(\varepsilon-C_0N^{-\beta})
+
\frac{\rho(p+\gamma)}{p}\log N
+
\mathrm{const}.
\]
The first-order condition gives
\[
\frac{\rho(p+\gamma)}{pN}
=
\frac{\gamma}{p}
\frac{\beta C_0N^{-\beta-1}}
{\varepsilon-C_0N^{-\beta}}.
\]
Hence
\[
\rho(p+\gamma)(\varepsilon-C_0N^{-\beta})
=
\beta\gamma C_0N^{-\beta}.
\]
Therefore
\[
C_0N^{-\beta}
=
\frac{\rho(p+\gamma)}
{\rho(p+\gamma)+\beta\gamma}\varepsilon
=
\frac{\rho(p+\gamma)}{D_\gamma}\varepsilon,
\]
which yields
\[
N^*
=
\left(
\frac{C_0D_\gamma}
{\rho(p+\gamma)\varepsilon}
\right)^{1/\beta}.
\]
At this horizon,
\[
\eta_{N^*}
=
\varepsilon-C_0(N^*)^{-\beta}
=
\frac{\beta\gamma}{D_\gamma}\varepsilon.
\]
Substituting this and
\[
A_{N^*}(\gamma) \asymp C_A(\gamma)(N^*)^\rho
\]
into the fixed-$N$ allocation formula gives
\[
\delta_k^*
=
\left(
\frac{\beta\gamma\,\varepsilon}
{D_\gamma C_A(\gamma)(N^*)^\rho}
\right)^{1/p}
\alpha_k(N^*,\Theta)^{-1/(p+\gamma)} .
\]
In particular,
\[
N^* \asymp \varepsilon^{-1/\beta},
\qquad
\delta_k^*
\asymp 
\varepsilon^{(1+\rho/\beta)/p}
\alpha_k(N^*,\Theta)^{-1/(p+\gamma)}.
\]

We now consider the uniform restriction $\delta_k\equiv\delta$. For fixed $N$,
the error constraint becomes
\[
\mathcal E_0(N,\Theta)+\widetilde A_N\delta^p\le \varepsilon.
\]
Thus the largest admissible uniform fidelity level is
\[
\delta^*(N)
=
\left(
\frac{\varepsilon-C_0N^{-\beta}}
{\widetilde A_N}
\right)^{1/p}.
\]
The fixed-$N$ wall-clock cost is therefore
\[
T_{\rm unif}^*(N)
\asymp
N(\delta^*(N))^{-\gamma}
=
N
\left(
\frac{\widetilde A_N}
{\varepsilon-C_0N^{-\beta}}
\right)^{\gamma/p}.
\]
Using $\widetilde A_N\sim\widetilde C_A N^\sigma$, we get
\[
T_{\rm unif}^*(N)
\asymp
N
\left(
\frac{\widetilde C_A N^\sigma}
{\varepsilon-C_0N^{-\beta}}
\right)^{\gamma/p}.
\]
Equivalently,
\[
\log T_{\rm unif}^*(N)
=
\left(1+\frac{\sigma\gamma}{p}\right)\log N
-
\frac{\gamma}{p}\log(\varepsilon-C_0N^{-\beta})
+
\mathrm{const}.
\]
The condition $p+\sigma\gamma>0$ ensures that the continuous objective grows
for sufficiently large $N$, so the interior critical point gives the interior
minimizer. Differentiating gives
\[
1+\frac{\sigma\gamma}{p}
=
\frac{\beta\gamma}{p}
\frac{C_0N^{-\beta}}
{\varepsilon-C_0N^{-\beta}}.
\]
Hence
\[
C_0N^{-\beta}
=
\frac{p+\sigma\gamma}
{p+\sigma\gamma+\beta\gamma}\varepsilon
=
\frac{p+\sigma\gamma}{\widetilde D_\gamma}\varepsilon.
\]
Therefore
\[
\widetilde N^*
=
\left(
\frac{C_0\widetilde D_\gamma}
{(p+\sigma\gamma)\varepsilon}
\right)^{1/\beta},
\]
and
\[
\varepsilon-C_0(\widetilde N^*)^{-\beta}
=
\frac{\beta\gamma}{\widetilde D_\gamma}\varepsilon.
\]
Substituting into the fixed-$N$ formula for $\delta^*(N)$ yields
\[
\delta^*
=
\left(
\frac{\beta\gamma\,\varepsilon}
{\widetilde D_\gamma\widetilde C_A(\widetilde N^*)^\sigma}
\right)^{1/p}.
\]
Consequently,
\[
\widetilde N^*\asymp \varepsilon^{-1/\beta},
\qquad
\delta^*
\asymp 
\varepsilon^{(1+\sigma/\beta)/p}.
\]

The argument above covers the interior optimum only. If the resulting fidelity
level violates an admissible upper bound, e.g. $\delta^*>1$, then the optimizer
lies on the boundary of the fidelity set and the displayed interior formulas
do not apply.
\end{proof}

\subsection{Optimized nonuniform-to-uniform cost ratio.}
\label{app:optimized_ratio}

For the nonuniform allocation, Prop.~\ref{lemma:master} gives, for fixed $N$,
\[
T_{\rm total}(N)
\asymp
(\varepsilon-C_0N^{-\beta})^{-\gamma/p}
\bigl(C_A(\gamma)N^\rho\bigr)^{(p+\gamma)/p}.
\]
At the continuous optimum,
\[
N^*
=
\left(\frac{C_0D_\gamma}{\rho(p+\gamma)\varepsilon}\right)^{1/\beta},
\qquad
\varepsilon-C_0(N^*)^{-\beta}
=
\frac{\beta\gamma}{D_\gamma}\varepsilon,
\]
where $D_\gamma=\rho(p+\gamma)+\beta\gamma$. Hence
\[
T_{\rm total}(N^*)
\asymp
C_A(\gamma)^{(p+\gamma)/p}
\left(\frac{\beta\gamma}{D_\gamma}\right)^{-\gamma/p}
\left(\frac{C_0D_\gamma}{\rho(p+\gamma)}\right)^{
\frac{\rho(p+\gamma)}{\beta p}}
\varepsilon^{-D_\gamma/(\beta p)} .
\]

For the uniform allocation,
\[
T_{\rm total}^{\rm unif}(N)
\asymp
N
\left(
\frac{\widetilde C_A N^\sigma}
{\varepsilon-C_0N^{-\beta}}
\right)^{\gamma/p}.
\]
At the continuous optimum,
\[
\widetilde N^*
=
\left(
\frac{C_0\widetilde D_\gamma}
{(p+\sigma\gamma)\varepsilon}
\right)^{1/\beta},
\qquad
\varepsilon-C_0(\widetilde N^*)^{-\beta}
=
\frac{\beta\gamma}{\widetilde D_\gamma}\varepsilon,
\]
where $\widetilde D_\gamma=p+\sigma\gamma+\beta\gamma$. Therefore
\[
T_{\rm total}^{\rm unif}(\widetilde N^*)
\asymp
\widetilde C_A^{\gamma/p}
\left(\frac{\beta\gamma}{\widetilde D_\gamma}\right)^{-\gamma/p}
\left(\frac{C_0\widetilde D_\gamma}{p+\sigma\gamma}\right)^{
\frac{p+\sigma\gamma}{\beta p}}
\varepsilon^{-\widetilde D_\gamma/(\beta p)} .
\]
Dividing the two estimates yields
\[
R_{\rm opt}(\varepsilon,\gamma)
:=
\frac{T_{\rm total}(N^*)}
{T_{\rm total}^{\rm unif}(\widetilde N^*)}
\asymp
K_\gamma
\varepsilon^{(\widetilde D_\gamma-D_\gamma)/(\beta p)} .
\]
Since
\[
\widetilde D_\gamma-D_\gamma
=
p+\sigma\gamma-\rho(p+\gamma),
\]
we obtain
\[
R_{\rm opt}(\varepsilon,\gamma)
\asymp
K_\gamma
\varepsilon^{
\frac{p+\sigma\gamma-\rho(p+\gamma)}{\beta p}} .
\]
Thus $R_{\rm opt}(\varepsilon,\gamma)\to0$ as $\varepsilon\to0$ whenever
\[
\rho(p+\gamma)<p+\sigma\gamma .
\]

\section{Proof of Prop.~\ref{lem:master_decomp}}
\label{sec:proof_master_decomp}
The idea of the proof is as follows.  

\textbf{Proof sketch}
The estimate $g_k$ enters the per-step 
inequality only through the inner product $\langle g_k,v_k\rangle$ with an
$\mathcal{F}_{k-1}$-measurable vector $v_k$ (e.g.\ $v_k=x_k-x^{*}$ for GD) and through
the squared norm $\|g_k\|^2$. Taking the conditional expectation gives
$\E[\langle g_k,v_k\rangle\mid\mathcal{F}_{k-1}]
= \langle\nabla f(x_k),v_k\rangle + \langle b_k(x_k),v_k\rangle$;
Cauchy--Schwarz on the bias term produces
$\|b_k\|\cdot\|v_k\|$, while $\E[\|g_k\|^2\mid\mathcal{F}_{k-1}]\le V_k$
handles the squared-norm term. A Young-type inequality applied to any cross term involving $\|b_k\|\cdot\|v_k\|$ yields the quadratic
contribution $\|b_k\|^2$. Telescoping across iterations gives the three
sums in~\eqref{eq:err_accumulation_2}.

Table~\ref{tab:drift_coefficients} provides specific examples of error decomposition
\eqref{eq:err_accumulation_2}. Namely, $a_k$ multiplies the linear bias term,
$c_k$ multiplies the squared bias term, and $e_k$ multiplies the second-moment
term of the stochastic gradient estimator. These coefficients depend only on
the underlying first-order method and not on the particular zeroth-order
oracle.

This separation is useful because the oracle-specific estimates enter only
through bounds on the conditional bias $b_k$ and second moment $V_k$. 
In particular,
mirror descent accumulates the bias linearly through the regret inequality,
nonconvex SGD converts the bias-gradient cross term into a squared-bias
contribution, strongly convex SGD propagates the errors through the contraction
weights, and accelerated SGD amplifies them according to the acceleration
weights.

\begin{table}[ht!]
\centering
\caption{%
  Coefficients $a_k$, $c_k$, $e_k$ in the error decomposition~\eqref{eq:err_accumulation_2}.
  A dash ``---'' indicates the term is absent.
}
\label{tab:drift_coefficients}
\smallskip
\renewcommand{\arraystretch}{1.6}
\setlength{\tabcolsep}{5pt}
\begin{tabular}{@{}lccccl@{}}
\toprule
\textbf{Method}
  & $\boldsymbol{\mathrm{Gap}_N}$
  & $\boldsymbol{a_k}$
  & $\boldsymbol{c_k}$
  & $\boldsymbol{e_k}$
  & \textbf{Ref.} \\
\midrule
Mirror Descent (cvx)
  & $\E [f(\bar{x}_N)] - f^*$
  & $\tfrac{\alpha_k D_K}{S_N}$
  & ---
  & $\tfrac{\alpha_k^2}{2 S_N}$
  & Prop.~\ref{prop:mirror_descent} \\[6pt]
SGD (nonconvex)
  & $\E[\tfrac{1}{S_N}\displaystyle\sum_k \alpha_k \|\nabla f(x_k)\|^2]$
  & ---
  & $\tfrac{\alpha_k}{S_N}$
  & $\tfrac{L\alpha_k^2}{S_N}$
  & Prop.~\ref{prop:sgd_nonconvex} \\[6pt]
SGD (\(\mu\)-str. cvx)
  & \(f(x_N)-f^*\)
  & ---
  & \(\dfrac{L\rho_N\alpha_k}{\mu\rho_{k+1}}\)
  & \(\dfrac{L\rho_N\alpha_k^2}{\rho_{k+1}}\)
  & Prop.~\ref{prop:sgd_strongly_convex} \\[6pt]
Accelerated SGD (cvx)
&
\(f(x_N^{ag})-f^*\)
&
\(\dfrac{\gamma_kR}{\beta_N\gamma_N}\)
&
\(\dfrac{\gamma_k^2}{\beta_N\gamma_N}\)
&
\(\dfrac{\gamma_k^2}{\beta_N\gamma_N}\)
&
Prop.~\ref{prop:accelerated_sgd_biased} \\
\bottomrule
\end{tabular}
 
\smallskip
\begin{minipage}{0.92\linewidth}
\small
\emph{Notation.}
$S_N := \sum_{t=0}^{N-1}\alpha_t$;\;
$\rho_{k+1} := \prod_{s=0}^{k}(1-\alpha_s\mu)$;\;
$D_K := \sup_{x\in K}\|x - x^*\|$;\;
$R := \|x_0 - x^*\|$;\;
$\beta_N\gamma_N$ is the normalization of the accelerated scheme;\;
$B$ is the batch size.
All equalities hold up to universal constants.
\end{minipage}
\end{table}

\begin{proposition}[Mirror descent with biased stochastic gradients]
\label{prop:mirror_descent}
Let $K\subset\mathbb{R}^d$ be convex and compact, and let $x^*\in K$ be a
minimizer of a convex function $f$. Consider mirror descent
\[
x_{k+1}
=
\argmin_{x\in K}
\left\{
\alpha_k\langle \hat g_k,x\rangle
+
D_\psi(x,x_k)
\right\},
\]
where $\psi$ is $1$-strongly convex with respect to a norm $\|\cdot\|$ and
$D_\psi$ is the associated Bregman divergence. Suppose
\[
\mathbb{E}[\hat g_k\mid\mathcal F_k]=\nabla f(x_k)+b_k(x_k),
\qquad
\mathbb{E}[\|\hat g_k\|_*^2\mid\mathcal F_k]\le V_k(x_k).
\]
Let $D_K:=\sup_{x\in K}\|x-x^*\|$ and $S_N:=\sum_{k=0}^{N-1}\alpha_k$.
Then the averaged point
$\bar x_N:=S_N^{-1}\sum_{k=0}^{N-1}\alpha_k x_k$ satisfies
\[
\mathbb{E}[f(\bar x_N)-f^*]
\le
\frac{D_\psi(x^*,x_0)}{S_N}
+
\sum_{k=0}^{N-1}
\left[
\frac{\alpha_k D_K}{S_N}\mathbb{E}\|b_k(x_k)\|
+
\frac{\alpha_k^2}{2S_N}\mathbb{E}V_k(x_k)
\right].
\]
Thus, in~\eqref{eq:err_accumulation_2},
\[
a_k=\frac{\alpha_kD_K}{S_N},\qquad
c_k=0,\qquad
e_k=\frac{\alpha_k^2}{2S_N}.
\]
\end{proposition}

\begin{proof}
The standard mirror-descent three-point inequality gives, for every
$x\in K$,
\[
\alpha_k\langle \hat g_k,x_k-x\rangle
\le
D_\psi(x,x_k)-D_\psi(x,x_{k+1})
+
\frac{\alpha_k^2}{2}\|\hat g_k\|_*^2 .
\]
Taking $x=x^*$ and conditional expectation with respect to $\mathcal F_k$,
\[
\alpha_k
\langle \nabla f(x_k),x_k-x^*\rangle
\le
\mathbb{E}\!\left[
D_\psi(x^*,x_k)-D_\psi(x^*,x_{k+1})
\mid \mathcal F_k
\right]
+
\frac{\alpha_k^2}{2}V_k(x_k)
+
\alpha_k\langle b_k(x_k),x^*-x_k\rangle .
\]
By convexity,
$f(x_k)-f^*\le \langle \nabla f(x_k),x_k-x^*\rangle$, and by Cauchy--Schwarz,
\[
\langle b_k(x_k),x^*-x_k\rangle
\le
D_K\|b_k(x_k)\|.
\]
Summing over $k$, taking total expectation, and using convexity once more,
$f(\bar x_N)\le S_N^{-1}\sum_k\alpha_k f(x_k)$, yields the claim.
\end{proof}

\begin{proposition}[Nonconvex SGD with biased stochastic gradients]
\label{prop:sgd_nonconvex}
Let $f$ be $L$-smooth and consider the SGD recursion
\[
x_{k+1}=x_k-\alpha_k\hat g_k,
\]
where
\[
\mathbb{E}[\hat g_k\mid\mathcal F_k]=\nabla f(x_k)+b_k(x_k),
\qquad
\mathbb{E}[\|\hat g_k\|^2\mid\mathcal F_k]\le V_k(x_k).
\]
Let $S_N:=\sum_{k=0}^{N-1}\alpha_k$. Then
\[
\frac{1}{S_N}\sum_{k=0}^{N-1}\alpha_k
\mathbb{E}\|\nabla f(x_k)\|^2
\lesssim
\frac{f(x_0)-f_{\inf}}{S_N}
+
\sum_{k=0}^{N-1}
\left[
\frac{\alpha_k}{S_N}\mathbb{E}\|b_k(x_k)\|^2
+
\frac{L\alpha_k^2}{S_N}\mathbb{E}V_k(x_k)
\right].
\]
Thus, in~\eqref{eq:err_accumulation_2},
\[
a_k=0,\qquad
c_k=\frac{\alpha_k}{S_N},\qquad
e_k=\frac{L\alpha_k^2}{S_N},
\]
up to universal constants.
\end{proposition}

\begin{proof}
By $L$-smoothness,
\[
f(x_{k+1})
\le
f(x_k)
-\alpha_k\langle \nabla f(x_k),\hat g_k\rangle
+
\frac{L\alpha_k^2}{2}\|\hat g_k\|^2 .
\]
Taking conditional expectation gives
\[
\mathbb{E}[f(x_{k+1})\mid\mathcal F_k]
\le
f(x_k)
-\alpha_k\|\nabla f(x_k)\|^2
-\alpha_k\langle \nabla f(x_k),b_k(x_k)\rangle
+
\frac{L\alpha_k^2}{2}V_k(x_k).
\]
Using Young's inequality,
\[
-\langle \nabla f(x_k),b_k(x_k)\rangle
\le
\frac12\|\nabla f(x_k)\|^2+\frac12\|b_k(x_k)\|^2,
\]
we obtain
\[
\frac{\alpha_k}{2}\|\nabla f(x_k)\|^2
\le
f(x_k)-\mathbb{E}[f(x_{k+1})\mid\mathcal F_k]
+
\frac{\alpha_k}{2}\|b_k(x_k)\|^2
+
\frac{L\alpha_k^2}{2}V_k(x_k).
\]
Summing and dividing by $S_N$ gives the stated bound, after absorbing
absolute constants.
\end{proof}

\begin{proposition}[Strongly convex SGD with biased stochastic gradients]
\label{prop:sgd_strongly_convex}
Let \(f\) be \(\mu\)-strongly convex and \(L\)-smooth, and let
\(x^*=\arg\min_x f(x)\). Consider
\[
    x_{k+1}=x_k-\alpha_k\widehat g_k,
    \qquad 0<\alpha_k\le \frac1L,
\]
where
\[
    \mathbb E[\widehat g_k\mid \mathcal F_k]
    =
    \nabla f(x_k)+b_k(x_k),
    \qquad
    \mathbb E[\|\widehat g_k\|^2\mid \mathcal F_k]
    \le V_k(x_k).
\]
Define
\[
    \rho_0:=1,
    \qquad
    \rho_{k+1}:=\prod_{s=0}^{k}(1-\mu\alpha_s).
\]
Then
\[
\mathbb E\|x_N-x^*\|^2
\lesssim
\rho_N\|x_0-x^*\|^2
+
\rho_N
\sum_{k=0}^{N-1}
\left[
    \frac{\alpha_k}{\mu\rho_{k+1}}
    \mathbb E\|b_k(x_k)\|^2
    +
    \frac{\alpha_k^2}{\rho_{k+1}}
    \mathbb E V_k(x_k)
\right].
\]
Consequently,
\[
\mathbb E[f(x_N)-f^*]
\lesssim
L\rho_N\|x_0-x^*\|^2
+
\sum_{k=0}^{N-1}
\left[
    \frac{L\rho_N\alpha_k}{\mu\rho_{k+1}}
    \mathbb E\|b_k(x_k)\|^2
    +
    \frac{L\rho_N\alpha_k^2}{\rho_{k+1}}
    \mathbb E V_k(x_k)
\right].
\]
Thus, in the decomposition \((7)\), the corresponding coefficients are
\[
    a_k=0,
    \qquad
    c_k=\frac{L\rho_N\alpha_k}{\mu\rho_{k+1}},
    \qquad
    e_k=\frac{L\rho_N\alpha_k^2}{\rho_{k+1}},
\]
up to universal constants.
\end{proposition}

\begin{proof}
Using the update rule,
\[
\|x_{k+1}-x^*\|^2
=
\|x_k-x^*\|^2
-2\alpha_k\langle \widehat g_k,x_k-x^*\rangle
+\alpha_k^2\|\widehat g_k\|^2 .
\]
Taking conditional expectation gives
\[
\mathbb E[\|x_{k+1}-x^*\|^2\mid \mathcal F_k]
\le
\|x_k-x^*\|^2
-2\alpha_k\langle \nabla f(x_k),x_k-x^*\rangle
-2\alpha_k\langle b_k(x_k),x_k-x^*\rangle
+\alpha_k^2 V_k(x_k).
\]
Since \(f\) is \(\mu\)-strongly convex and \(x^*\) is its minimizer,
\[
\langle \nabla f(x_k),x_k-x^*\rangle
\ge
\mu\|x_k-x^*\|^2 .
\]
For the bias term, Young's inequality gives
\[
2\alpha_k
\bigl|\langle b_k(x_k),x_k-x^*\rangle\bigr|
\le
\mu\alpha_k\|x_k-x^*\|^2
+
\frac{\alpha_k}{\mu}\|b_k(x_k)\|^2 .
\]
Therefore,
\[
\mathbb E[\|x_{k+1}-x^*\|^2\mid \mathcal F_k]
\le
(1-\mu\alpha_k)\|x_k-x^*\|^2
+
\frac{\alpha_k}{\mu}\|b_k(x_k)\|^2
+
\alpha_k^2 V_k(x_k).
\]
Taking total expectation and unrolling the recursion yields
\[
\mathbb E\|x_N-x^*\|^2
\le
\rho_N\|x_0-x^*\|^2
+
\rho_N
\sum_{k=0}^{N-1}
\left[
    \frac{\alpha_k}{\mu\rho_{k+1}}
    \mathbb E\|b_k(x_k)\|^2
    +
    \frac{\alpha_k^2}{\rho_{k+1}}
    \mathbb E V_k(x_k)
\right].
\]
Finally, by \(L\)-smoothness and \(\nabla f(x^*)=0\),
\[
f(x_N)-f^*
\le
\frac L2\|x_N-x^*\|^2.
\]
Multiplying the previous display by \(L/2\) and absorbing universal constants gives the claimed bound.
\end{proof}

\begin{proposition}[Accelerated SGD with biased stochastic gradients]
\label{prop:accelerated_sgd_biased}
Consider the accelerated stochastic scheme used in the ZO-AccSGD analysis,
with weights \((\beta_k,\gamma_k)\). Let the estimator \(g_k\) satisfy
\[
    \mathbb E[g_k\mid \mathcal F_k]
    =
    \nabla f(y_k)+b_k(y_k),
    \qquad
    \mathbb E[\|g_k\|^2\mid \mathcal F_k]
    \le V_k(y_k).
\]
Assume that in the accelerated Lyapunov recursion the gradient error enters as
\[
\beta_N\gamma_N\bigl(f(x_N^{ag})-f^*\bigr)
\le
R_0+
\sum_{k=0}^{N-1}
\left[
    \gamma_k
    \langle g_k-\nabla f(y_k),x^*-z_k\rangle
    +
    \gamma_k^2
    \|g_k-\nabla f(y_k)\|^2
\right],
\]
where \(R_0\) is independent of the oracle fidelity. If
\(\|z_k-x^*\|\le R\) along the trajectory, then
\[
\mathbb E[f(x_N^{ag})-f^*]
\le
\widetilde E_0(N,\Theta)
+
\sum_{k=0}^{N-1}
\left[
    \frac{\gamma_k R}{\beta_N\gamma_N}
    \mathbb E\|b_k(y_k)\|
    +
    \frac{\gamma_k^2}{\beta_N\gamma_N}
    \mathbb E\|b_k(y_k)\|^2
    +
    \frac{\gamma_k^2}{\beta_N\gamma_N}
    \mathbb E V_k(y_k)
\right].
\]
Thus, in the decomposition \((7)\), the corresponding coefficients are
\[
    a_k=\frac{\gamma_k R}{\beta_N\gamma_N},
    \qquad
    c_k=\frac{\gamma_k^2}{\beta_N\gamma_N},
    \qquad
    e_k=\frac{\gamma_k^2}{\beta_N\gamma_N}.
\]
\end{proposition}

\begin{proof}
Let
\[
    \varepsilon_k
    :=
    g_k-\mathbb E[g_k\mid \mathcal F_k].
\]
Then
\[
    \mathbb E[\varepsilon_k\mid \mathcal F_k]=0,
    \qquad
    g_k-\nabla f(y_k)
    =
    b_k(y_k)+\varepsilon_k.
\]
For the linear term in the accelerated recursion, taking conditional
expectation gives
\[
\mathbb E\!\left[
    \langle g_k-\nabla f(y_k),x^*-z_k\rangle
    \mid \mathcal F_k
\right]
=
\langle b_k(y_k),x^*-z_k\rangle.
\]
Hence, by Cauchy--Schwarz and the assumption \(\|z_k-x^*\|\le R\),
\[
\mathbb E\!\left[
    \langle g_k-\nabla f(y_k),x^*-z_k\rangle
    \mid \mathcal F_k
\right]
\le
R\|b_k(y_k)\|.
\]

For the quadratic term, we first note that
\[
\begin{aligned}
\mathbb E\!\left[
    \|g_k-\nabla f(y_k)\|^2
    \mid \mathcal F_k
\right]
&=
\mathbb E\!\left[
    \|b_k(y_k)+\varepsilon_k\|^2
    \mid \mathcal F_k
\right]
\\
&=
\|b_k(y_k)\|^2
+
\mathbb E[\|\varepsilon_k\|^2\mid \mathcal F_k],
\end{aligned}
\]
because
\(\mathbb E[\varepsilon_k\mid\mathcal F_k]=0\).
Moreover,
\[
\begin{aligned}
\mathbb E[\|\varepsilon_k\|^2\mid \mathcal F_k]
&=
\mathbb E\!\left[
    \|g_k-\mathbb E[g_k\mid\mathcal F_k]\|^2
    \mid \mathcal F_k
\right]
\\
&=
\mathbb E[\|g_k\|^2\mid \mathcal F_k]
-
\|\mathbb E[g_k\mid\mathcal F_k]\|^2
\\
&\le
\mathbb E[\|g_k\|^2\mid \mathcal F_k]
\le
V_k(y_k).
\end{aligned}
\]
Therefore,
\[
\mathbb E\!\left[
    \|g_k-\nabla f(y_k)\|^2
    \mid \mathcal F_k
\right]
\le
\|b_k(y_k)\|^2+V_k(y_k).
\]

Taking conditional expectation in the accelerated Lyapunov recursion and using
the two bounds above yields
\[
\beta_N\gamma_N
\mathbb E[f(x_N^{ag})-f^*]
\le
\widetilde R_0
+
\sum_{k=0}^{N-1}
\left[
    \gamma_k R\mathbb E\|b_k(y_k)\|
    +
    \gamma_k^2\mathbb E\|b_k(y_k)\|^2
    +
    \gamma_k^2\mathbb E V_k(y_k)
\right],
\]
where \(\widetilde R_0\) is independent of the oracle fidelity. Dividing by
\(\beta_N\gamma_N\) and absorbing \(\widetilde R_0/(\beta_N\gamma_N)\) into
\(\widetilde E_0(N,\Theta)\) gives the claim.
\end{proof}

\medskip

\section{Proof of Prop.~\ref{prop:kernel_bias_variance}}
\label{sec:proof_example_kernel_estimator}
As an example, we consider a two-point oracle.
\begin{lemma}
\label{lem:oracle_bv}
Let $\hat f(x)=f(x)+\xi_t(x)$ and consider the two-point estimator (Def.~\ref{def:two_point_estimator})
with bandwidth $h_t$ and batch size $B$.
Let $x^{\pm}_{t} = x_{t, i}\pm h_t r_{t,i}\zeta_{t,i}$ and set
\begin{equation}
\label{eq:annoying_notations}
\xi_{t,i}^{+} := \xi^{+}_t(x^{+}_{t, i}),
\quad 
\xi_{t,i}^{-} := \xi^{-}_t(x^{-}_{t, i}), \quad \Delta_{t,i}:=\xi_{t,i}^+ -\xi_{t,i}^-.
\end{equation}
Then
\[
\|b_t\|_* \;\lesssim\; \kappa_\beta L_\beta h_t^{\beta-1}
\;+\;\frac{d}{h_t}\E^{1/2}[\Delta_{t,1}^2\mid\mathcal F_t], \quad V_t \;\lesssim\; \frac{1}{B}\E[\|g_{t,1}\|_*^2\mid\mathcal F_t]
\;+\;\frac{d^2}{B h_t^2}\E[\Delta_{t,1}^2\mid\mathcal F_t],
\]
where 
\[
g_{t,1} := \frac{d}{2h_t} \Big( f(x^{+}_{t, 1}) - f(x^{-}_{t, 1}) \Big) K(r_1) \zeta_1.
\]
\end{lemma}
\begin{proof}
We use \eqref{eq:annoying_notations} and write
\[
\hat{g}(x_{t}, h_t, r_i, \zeta_i) := \underbrace{\frac{d}{2h_t} \Big( f(x^{+}_{t, i}) - f(x^{-}_{t, i}) \Big) K(r_i) \zeta_i}_{g_{i, t}:= } + \frac{d}{2h_t} K(r_i) \zeta_i \Delta_{t, i}.
\]
Then, if we use batching with the batch size $B$, we get
\[
\hat{\mathbf{g}}(x_t) := \frac{1}{B} \sum^B_{i=1} \hat{g}(x_{t}, h_t, r_i, \zeta_i).
\]
Next, we note that
\[
\mathbb{E}[\hat{\mathbf{g}}(x_t) | \mathcal{F}_t
] -
\nabla f(x_t)
= \mathbb{E}[g_{1, t} | \mathcal{F}_t
]  - \nabla f(x_t) +  \frac{d}{2h_t}  \mathbb{E}[K(r)\zeta\Delta_{1, t}| \mathcal{F}_t
].
\]
Using the result from~\cite{bychkov2024accelerated}, to control the first term in the r.h.s., and applying Cauchy–Schwarz inequality to the second term, we get
\[
\left \|\mathbb{E}[\hat{\mathbf{g}}(x_t) | \mathcal{F}_t
] -
\nabla f(x_t) \right\| \lesssim \kappa_{\beta} L_{\beta} h^{\beta - 1}_t + \frac{d}{h_t}\mathbb{E}[\Delta^2_{1, t} | \mathcal{F}_t].
\]
Thus, 
\[
b_t := \kappa_{\beta} L_{\beta} h^{\beta - 1}_t + \frac{d}{h_t}\mathbb{E}^{1/2}[\Delta^2_{1, t} | \mathcal{F}_t].
\]
Similarly we get
\[
V_t
:=
\frac{2}{B}\,\E[\|g_{t,1}\|_*^2\mid\mathcal F_t]
+
\frac{d^2}{B\,h_t^2}\,\E[\Delta_{t,1}^2\mid\mathcal F_t].
\]
\end{proof}

Under either noise model (Definitions~\ref{asm:tsybakov_noise}
and~\ref{asm:adv_noise}),
$\E[\Delta_k^{2}\mid\mathcal{F}_{k-1}]\lesssim\delta_k^{2}$, so
we get
\begin{equation}
\label{eq:kernel_fidelity_bounds}
\|b_k(x_k)\|\;\lesssim\;\kappa_\beta L\,h^{\beta-1}+\frac{d}{h}\,\delta_k,
\qquad
V_k(x_k)\;\lesssim\;G^{2}+\frac{d^{2}}{h^{2}}\,\delta_k^{2},
\end{equation}
where $G^{2}$ bounds
$\E[|f(x_k^{+})-f(x_k^{-})|^{2}\mid\mathcal{F}_{k-1}]$ (for instance,
via Lipschitz continuity of $f$).

\section{Wall-clock complexity for adversarial noise missing proofs and remarks} \label{section:proof intermidiate}

Let $x_N$ be an estimate obtained after $N$ steps and let the corresponding approximation error be $\mathcal{E}(N, \delta, \red{\Theta})$. \citet{dvurechensky2016stochastic} and \citet{devolder2013intermediate} ensure that if $f$ is strongly convex (Asm.~\ref{asm:regularity}, $\mu > 0$), the bound is
\begin{equation*}
    f(x_N) - f^* 
    \lesssim \mathcal{E}(N, \delta, \red{\Theta})  := LR^2 \exp\left(-\left(\frac{\mu}{L}\right)^{\frac{1}{p}} N\right) + \left(\frac{L}{\mu}\right)^{\frac{2p-1}{p}}  d\delta,
\end{equation*}
where $R := \|x_0 - x^*\|$. Applying Corollary~\ref{corr:uniform} for convergence above (case) one can obtain result provided in Proposition~\ref{prop:igm_adversarial_uniform} (case $\mu > 0$).

\begin{remark}  \label{remark:const delta adversarial}
    Analysis of \eqref{eq:complexity_igm} demonstrates that if $\gamma < 1$, the wall-clock time complexity is optimal (minimal) for $p=2$. This corresponds to FGM. If $\gamma \ge 1$, the optimal wall-clock time complexity is achieved at $p =1$. This corresponds to GM. Thus, the upper bound on the wall-clock time complexity is tighter for GM; so,  accelerated methods can be wall-clock inferior to non-accelerated schemes.
\end{remark}

Furthermore, if $f$ is convex (i.e., Assumption~\ref{asm:regularity} with $\mu = 0$ holds), convergence can be expressed as:
\begin{equation}
\label{eq:err_convex}
    f(x_N) - f^* 
\lesssim \mathcal{E}(N, \delta, \red{\Theta}) := \frac{LR^2}{N^p} + \sqrt{d\delta L}\tilde{R}_N + N^{p-1}d\delta,
\end{equation}
where $\tilde{R}_N = \max\limits_{k \leqslant N - 1} \|x^k - x^* \|_2$ is a parameter depending on the convergence trajectory of the chosen algorithm. Using stopping criteria similar to those proposed in~\cite{vasin2023accelerated} or setting $\tilde{R}_N = O(R)$ and Corollary~\ref{corr:uniform} the second part of result of the Proposition~\ref{prop:igm_adversarial_uniform} can be obtained.

\textit{\textbf{Proof of Proposition}~\ref{prop:igm_adversarial_uniform}} (case $\mu = 0$)
\begin{proof}
    Convergence bound~\eqref{eq:err_convex} can be expressed as:
    \begin{equation*}
        \frac{LR^2}{N^p} + 2 \max \lbrace \sqrt{d\delta L} R, N^{p-1}d\delta \rbrace.
    \end{equation*}
    Then we can apply Corollary~\ref{corr:uniform} to both components of error accumulation. Estimating iterations amount:
    \begin{equation*}
        \frac{LR^2}{N^p} \leqslant \frac{\varepsilon}{3} \Rightarrow N = O \left( \left(\frac{3 LR^2}{\varepsilon} \right)^{\nicefrac{1}{p}} \right).
    \end{equation*}
    1. Applying Corollary~\ref{corr:uniform} to component $\sqrt{d\delta L} R$:
    \begin{eqnarray*}
        T_{\mathrm{total}}^{(1)} & \leqslant & N \left( \frac{2 \varepsilon}{3} \right)^{-2 \gamma} \left(2 \sqrt{d L} R \right)^{2 \gamma} \leqslant \left(\frac{3 LR^2}{\varepsilon} \right)^{\nicefrac{1}{p}} \left( \frac{9}{4 \varepsilon^2} \right)^{\gamma} \left(4 d LR^2 \right)^{\gamma} \\
        & = & \left(\frac{3 LR^2}{\varepsilon} \right)^{\nicefrac{1}{p}} \left( \frac{9 d LR^2}{\varepsilon^2} \right)^{\gamma} = O \left( \left( \frac{\varepsilon}{9 d} \right)^{-\gamma} \left( \frac{LR^2}{\varepsilon} \right)^{\gamma + \frac{1}{p}} \right).
    \end{eqnarray*}
    2. Applying Corollary~\ref{corr:uniform} to component $N^{p-1} d \delta$:
    \begin{eqnarray*}
        T_{\mathrm{total}}^{(2)} & \leqslant & N \left( \frac{2 \varepsilon}{3} \right)^{-\gamma} \left(N^{p - 1} d  \right)^{\gamma} \leqslant \left(\frac{3 LR^2}{\varepsilon} \right)^{\nicefrac{1}{p}} \left( \frac{3}{2 \varepsilon} \right)^{\gamma} \left( \left(\frac{3 LR^2}{\varepsilon} \right)^{\nicefrac{\gamma (p - 1)}{p}} \right) d^{\gamma} \\ 
        & = & O \left( \left( \frac{2 \varepsilon}{3 d} \right)^{-\gamma} \left( \frac{3 LR^2}{\varepsilon} \right)^{\gamma + \frac{1 - \gamma}{p}} \right).
    \end{eqnarray*}
    Then we can use maximum of $T_{\mathrm{total}}^{(1)}$ and $T_{\mathrm{total}}^{(2)}$ for total estimation of wall clock time for $\mathcal{E}(N, \delta, \Theta) \leqslant \varepsilon$ solution. Note, that we can assume, that $\varepsilon \leqslant LR^2$, since $L$-smoothness (Assumptions~\ref{asm:regularity}) implies $f(x) - f^* \leqslant \tfrac{1}{2} LR^2$.  Thus we can estimate $T_{\mathrm{total}}$:
    \begin{eqnarray*}
         T_{\mathrm{total}}
         & = & O \left( \max \left \lbrace T_{\mathrm{total}}^{(1)}, T_{\mathrm{total}}^{(2)} \right \rbrace \right)
         \\
         & = & O \left( \max \left \lbrace
          \left( \frac{\varepsilon}{9 d} \right)^{-\gamma} \left( \frac{LR^2}{\varepsilon} \right)^{\gamma + \frac{1}{p}},
         \left( \frac{2 \varepsilon}{3 d} \right)^{-\gamma} \left( \frac{3 LR^2}{\varepsilon} \right)^{\gamma + \frac{1 - \gamma}{p}} \right \rbrace \right)
         \\
         & = & O \left( \left( \frac{\varepsilon}{9 d} \right)^{-\gamma} \cdot \max \left \lbrace
         \left( \frac{LR^2}{\varepsilon} \right)^{\gamma + \frac{1}{p}},
         \left( \frac{3 LR^2}{\varepsilon} \right)^{\gamma + \frac{1 - \gamma}{p}} \right \rbrace \right)
         \\
         & = & O \left( \left( \frac{\varepsilon}{9 d} \right)^{-\gamma} 
         \left( \frac{3LR^2}{\varepsilon} \right)^{\gamma + \frac{1}{p}} \right)
    \end{eqnarray*}
\end{proof}

\begin{remark}
    The convex case result of Proposition~\ref{prop:igm_adversarial_uniform} indicates that for any $\gamma > 0$ the optimal choice of method is FGM (i.e., $p=2$). Note that full gradient approximation requires $d$ queries to the oracle \citep{gasnikov2023randomized}. Consequently, under sequential queries, the bounds~\eqref{eq:complexity_igm} scale by a factor of $d$. However, the wall-clock time complexity remains invariant under parallel oracle queries.
\end{remark}



In papers~\cite{stonyakin2020inexact, stonyakin2021inexact} were obtained results for fast gradient method and $(\lbrace \delta_k \rbrace_{k = 0}^{N - 1}, L, \mu)$ inexact models described at~\cite{devolder2013intermediate}. To study wall-clock complexity for time-varying $\delta_k$ we introduce the following intermediate gradient method.

\begin{algorithm}[t]
    \caption{IGM (Intermediate Gradient Method).} 
    \label{alg re-agm}
    \begin{algorithmic}[1]
        \Require Starting point $x^0$, number of steps $N$, $L$ -- smoothness parameter, $\mu$ -- strong convexity parameter, $\nu$ -- intermediate parameter.
        \State \textbf{Set} $u^0 = x^0$.
        \State \textbf{Set} $h = \frac{1}{4L}$.
        \State \textbf{Set} $s = \left(1 + \frac{1}{4} \left(\frac{\mu}{2L} \right)^{\nu} \right)$.
        \State \textbf{Set} $m = \left(1 - \frac{1}{4} \left(\frac{\mu}{2L} \right)^{\nu} \right)$.
        \State \textbf{Set} $q = \frac{\mu}{16 L}$.
        \State \textbf{Set} $\omega = \frac{(m - s) + \sqrt{(s - m)^2 +4mq }}{2 m}$.
        \For{$k = 0 \dots N - 1$}
            \State $y^k = \frac{\omega u^k + x^k}{1 + \omega}$.
            \State $u^{k + 1} = (1 - \omega) u^{k} + \omega y^k - \frac{2 \omega}{\mu} \widetilde{\nabla} f(y^{k})$.
            \State $x^{k + 1} = y^{k} - h \widetilde{\nabla} f(y^{k})$.
        \EndFor
        \State \Return \textbf{Output:} $x^N$.
    \end{algorithmic}
\end{algorithm}

        
         
         


We will establish convergence of Algorithm~\ref{alg re-agm} with inexact gradient $\wtgg f$, satisfying:
\begin{equation} \label{def:abs grad error}
    (\forall x \in \mathbb{R}^d) \quad \widetilde{\nabla} f(x) = \nabla f(x) + \zeta_a(x), \quad \| \zeta_a(x) \|_2 \leqslant \Delta.
\end{equation}

Following \citep[p. 83]{nesterov2018lectures}, let us introduce a parameterized set of functions $\Psi$, its element defined for $c \in \mathbb{R}, \kappa \in \mathbb{R}^{++},$ and $u \in \mathbb{R}^n$,  as follows
\begin{equation} \label{qudaratic family}
    \psi(x | c, \kappa, u) = c + \frac{\kappa}{2} \| x - u \|_2^2, \quad \forall x \in \mathbb{R}^n. 
\end{equation}

According to~\cite{nesterov2018lectures}, we can mention to the following useful properties of the class $\Psi$.

\begin{lemma} \label{psi comb}
Let $\psi_1, \psi_2 \in \Psi$. Then $\forall \eta_1, \eta_2, c_1, c_2 \in \mathbb{R}, \forall \kappa_1, \kappa_2 \in \mathbb{R}^{++}$, and $\forall u, v \in \mathbb{R}^n$, we have 
\[
    \eta_1 \psi_1(x|c_1, \kappa_1, u) + \eta_2 \psi_2(x|c_2, \kappa_2, v) = \psi_3(x|c_3, \kappa_3, w), 
\]
where 
\begin{equation*}
    \begin{gathered}
        c_3 = \eta_1 c_1 + \eta_2 c_2 + \frac{\eta_1\eta_2\kappa_1\kappa_2}{2(\eta_1\kappa_1 + \eta_2\kappa_2)} \|u  - v \|_2^2, \\
        \kappa_3 = \eta_1\kappa_1 + \eta_2\kappa_2, \quad w = \frac{\eta_1\kappa_1 u +\eta_2\kappa_2 v}{\eta_1\kappa_1 + \eta_2\kappa_2}.
    \end{gathered}
\end{equation*}
\end{lemma}

\begin{lemma} \label{conv lemma}
Let $0 < \lambda < 1$, $A > 0$, $\psi_0, \psi \in \Psi$, $z \in \mathbb{R}^n$ such that 
\begin{equation*}
f(z) \leqslant \min_{x \in \mathbb{R}^n} \psi(x) + A,
\end{equation*}
and 
\begin{equation*}
     \psi(x) \leqslant \lambda \psi_0(x) + (1 - \lambda) f(x), \quad \forall x \in \mathbb{R}^n.
\end{equation*}
Then
\begin{equation*}
    f(z) - f^* \leqslant \lambda (\psi_0(x^*) - f^*) + A. 
\end{equation*}
\end{lemma}
\begin{proof}
    \begin{eqnarray*}
        f(z) - f^*
        & \leqslant & \min_{x \in \mathbb{R}^n} \psi(x) + A - f^*
        \\
        & \leqslant & \psi(x^*) - f^* + A \leqslant \lambda \psi_0(x^*) + (1 - \lambda) f(x^*) - f^* + A
        \\
        & = & \lambda ( \psi_0(x^*) - f^* ) + A. 
    \end{eqnarray*}
\end{proof}

\begin{lemma} \label{fenchel for accelerated}
    Let $\widetilde{\nabla} f$ satisfies error condition~\eqref{def:abs grad error}, $0 < \mu \leqslant L$, then:
    \begin{equation*}
        \begin{gathered}
            \forall \nu \geqslant 0 \quad \|\wtgg f(x) \|_2^2 \geqslant \left(1 - \frac{1}{4} \left(\frac{\mu}{2L} \right)^{\nu} \right) \|\nabla f(x) \|_2^2 - \left(4\left(\frac{2L}{\mu} \right)^{\nu} - 1 \right) \Delta^2, \\
            \|\wtgg f(x) \|_2^2 \leqslant\left(1 + \frac{1}{4} \left( \frac{\mu}{2L} \right) ^{\nu} \right) \|\nabla f(x) \|_2^2 + \left(1 + 4\left(\frac{2L}{\mu}\right)^{\nu} \right) \Delta^2.
        \end{gathered}
    \end{equation*}
\end{lemma}
\begin{proof}
    \begin{eqnarray*}
        \|\wtgg f(x) \|_2^2 & = & \|\nabla f(x) + \zeta_a(x) \|_2^2 = \|\nabla f(x) \|_2^2 + 2 \langle \nabla f(x), \zeta_a(x) \rangle + \| \zeta_a(x) \|_2^2 \\
        & \overset{\text{Fenchel ineq.}}{\leqslant} & \left(1 + \frac{1}{\lambda} \right) \| \nabla f(x) \|_2^2 + \left(1 + \lambda \right) \| \zeta_a(x) \|_2^2 \\
        & \overset{\lambda = 4 \left(2L / \mu \right)^{\nu}}{\leqslant} &
        \left(1 + \frac{1}{4} \left( \frac{\mu}{2L} \right) ^{\nu} \right) \|\nabla f(x) \|_2^2 + \left(1 + 4\left(\frac{2L}{\mu}\right)^{\nu} \right) \Delta^2.
    \end{eqnarray*}
    Similarly for the lower bound.
\end{proof}

\begin{lemma} \label{accelerated quadratic estimation}
    Let $\wtgg f$ satisfies error condition~\eqref{def:abs grad error}, $0 < \mu \leqslant L, 0 < \nu$, then:
    \begin{eqnarray*}
        \frac{\mu}{4} \left\|x - z + \frac{2}{\mu} \widetilde{\nabla} f(z) \right \|_2^2
        & \leqslant & \langle \nabla f(z), x - z \rangle + \frac{\mu}{2} \|x - z \|_2^2 \\
        & + & \frac{1}{\mu} \left(1 + \frac{1}{4} \left(\frac{\mu}{2L} \right)^{\nu} \right) \|\nabla f(z) \|_2^2 \\
        & + & \frac{1}{\mu} \left(4\left(\frac{2L}{\mu} \right)^{\nu} + 2 \right) \Delta^2, \\
        \frac{\mu}{4} \left\|x - z + \frac{2}{\mu} \widetilde{\nabla} f(z) \right \|_2^2
        & \geqslant & \langle \nabla f(z), x - z \rangle \\
        & + & \frac{1}{\mu} \left(1 - \frac{1}{4} \left(\frac{\mu}{2L} \right)^{\nu} \right) \|\nabla f(z) \|_2^2 \\
        & - & \frac{4}{\mu} \left(\frac{2L}{\mu} \right)^{\nu} \Delta^2.
    \end{eqnarray*}
\end{lemma}
\begin{proof}
We will find bounds for each term of sum:
\begin{equation*}
    \frac{\mu}{4} \left\|x - z + \frac{2}{\mu} \widetilde{\nabla} f(z) \right \|_2^2 = \frac{\mu}{4} \|x - z \|_2^2 + \langle \wtgg f(z), x - z \rangle + \frac{1}{\mu} \|\wtgg f(z) \|_2^2.
\end{equation*}
Linear form:
\begin{eqnarray*}
    \langle \wtgg f(z), x - z \rangle
    & \leqslant &  \langle \nabla f(z), x - z \rangle + \frac{1}{\mu} \Delta^2 + \frac{\mu}{4} \|x - z \|_2^2.
\end{eqnarray*}
\begin{eqnarray*}
    \langle \wtgg f(z), x - z \rangle
    & \geqslant &  \langle \nabla f(z), x - z \rangle - \frac{1}{\mu} \Delta^2 - \frac{\mu}{4} \|x - z \|_2^2.
\end{eqnarray*}
Quadratic term:
\begin{equation*}
    \|\wtgg f(z) \|_2^2 \overset{\text{Lemma }\ref{fenchel for accelerated}}{\leqslant} \left(1 + \frac{1}{4} \left(\frac{\mu}{2L} \right)^{\nu} \right) \|\nabla f(z) \|_2^2 + \left(4\left(\frac{2L}{\mu} \right)^{\nu} + 1 \right) \Delta^2.
\end{equation*}
\begin{equation*}
    \|\wtgg f(z) \|_2^2 \overset{\text{Lemma }\ref{fenchel for accelerated}}{\geqslant} \left(1 - \frac{1}{4} \left(\frac{\mu}{2L} \right)^{\nu} \right) \|\nabla f(z) \|_2^2 - \left(4\left(\frac{2L}{\mu} \right)^{\nu} - 1 \right) \Delta^2.
\end{equation*}
Then we can sum up estimations above:
\begin{eqnarray*}
    \frac{\mu}{4} \left\|x - z + \frac{2}{\mu} \widetilde{\nabla} f(z) \right \|_2^2 
    & \leqslant & \langle \nabla f(z), x - z \rangle + \frac{\mu}{2} \|x - z \|_2^2 \\
    & + & \frac{1}{\mu} \left(1 + \frac{1}{4} \left(\frac{\mu}{2L} \right)^{\nu} \right) \|\nabla f(z) \|_2^2 \\
    & + & \frac{1}{\mu} \left(4\left(\frac{2L}{\mu} \right)^{\nu} + 1 \right) \Delta^2 + \frac{\Delta^2}{\mu} \\
    & \leqslant & \langle \nabla f(z), x - z \rangle + \frac{\mu}{2} \|x - z \|_2^2 \\
    & + & \frac{1}{\mu} \left(1 + \frac{1}{4} \left(\frac{\mu}{2L} \right)^{\nu} \right) \|\nabla f(z) \|_2^2 \\
    & + & \frac{1}{\mu} \left(4\left(\frac{2L}{\mu} \right)^{\nu} + 2 \right) \Delta^2.
\end{eqnarray*}
\begin{eqnarray*}
    \frac{\mu}{4} \left\|x - z + \frac{2}{\mu} \widetilde{\nabla} f(z) \right \|_2^2
    & \geqslant & \langle \nabla f(z), x - z \rangle \\
    & + & \frac{1}{\mu} \left(1 - \frac{1}{4} \left(\frac{\mu}{2L} \right)^{\nu} \right) \|\nabla f(z) \|_2^2 \\
    & - & \frac{4}{\mu} \left(\frac{2L}{\mu} \right)^{\nu} \Delta^2.
\end{eqnarray*}

\end{proof}

\begin{lemma} \label{lower bound}
Let $f$ satisfies~\eqref{asm:regularity} function and $\widetilde{\nabla}{f}$ satisfies~\eqref{def:abs grad error}. Then for any $z \in \mathbb{R}^n$ and $\nu > 0$, there is $\phi_{z, \nu} \in \Psi$, such that: 
\begin{equation*}
    \begin{gathered}
        \phi_{z, \nu}(x) = c_{\phi} + \frac{\kappa_{\phi}}{2} \|x - u_{\phi} \|_2^2 \leqslant f(x), \quad \forall x \in \mathbb{R}^n. \\
        c_{\phi} := f(z) - \frac{1}{\mu} \left(1 + \frac{1}{4} \left(\frac{\mu}{2L} \right)^{\nu} \right) \|\nabla f(z) \|_2^2 - \frac{1}{\mu} \left(4\left(\frac{2L}{\mu} \right)^{\nu} + 2 \right) \Delta^2, \\
        \kappa_{\phi} := \mu / 2, \\
        u_{\phi} := z - \frac{2}{\mu} \widetilde{\nabla}f(z).
    \end{gathered}
\end{equation*}
\end{lemma}
\begin{proof}
Proof follows directly from Lemma~\ref{accelerated quadratic estimation} and strong convexity Def.~\eqref{asm:regularity}.
\end{proof}

Now, for Algorithm~\ref{alg re-agm}, we have $u^0 = x^0, x^0 \in \mathbb{R}^d$. Let $u^k$ be defined in Algorithm~\ref{alg re-agm} and $\omega$ defined in Algorithm~\ref{alg re-agm}. Let us define the following sequences:
\begin{align}\label{function seq}
    &  \lambda_0 = 1, \quad {c_0 = f(x^0)}, \quad \lambda_{k+1} = (1 - \omega) \lambda_k, \quad \forall k \geqslant0,
    \\& \psi_k(x| c_k, \mu / 2, u^k) = c_k + \frac{\mu}{4} \|x - u^k\|_2^2, \quad \forall k \geqslant 0
    \label{psi^k_Alg1}
\end{align}

Also $\phi_{y^k, \nu}$ is the function that is obtained from Lemma \ref{lower bound}. We define, recursively
\begin{align} \label{psi seq def}
    \psi_{k + 1}(x | c_{k + 1}, \mu / 2, u^{k + 1}) := (1 - \omega) \psi_k(x | c_k, \mu / 2, u^k) +  \omega \phi_{y^k, \nu}(x), \quad \forall k \geqslant 0.
\end{align}
Error accumulation sequence also required to be defined:
\begin{equation} \label{error accumulation seq}
    E_0 = 0, \quad E_{k + 1} = (1 - \omega) E_k + 3 \left(\frac{\mu}{L} \right)^{1 - 2 \nu} \frac{\Delta_k^2}{\mu}, \quad \| \wtgg f(y^k) - \nabla f(y^k) \|_2 \leqslant \Delta_k.
\end{equation}

The following Lemma motivates the definition of function $\phi_{y^k, \nu}$.

\begin{lemma} \label{upper bound for lemma}
Let $\{\lambda_k\}_{k \geqslant 0}$ be a sequence defined in \eqref{function seq} and $\{\psi_k\}_{k \geqslant 0}$ be the corresponding functions~\eqref{psi seq def}. Then 
\begin{equation*}
    \psi_k(x) \leqslant \lambda_k \psi_0(x) + (1 - \lambda_k) f(x), \quad \forall k \geqslant 0 \;\; \text{and} \;\; \forall x \in \mathbb{R}^n.
\end{equation*}

\end{lemma}
\begin{proof}
By the principle of induction. For $k = 0$, the inequality is obvious. Assume that inequality takes place for $k$. From \eqref{psi seq def}, and Lemma \ref{lower bound}, we get 
\begin{align*}
    \psi_{k + 1}(x)  &= (1 - \omega) \psi_k(x) + \omega \phi_{y^k, \nu}(x) 
    \\& \leqslant (1 - \omega) \left( (1 - \lambda_k) f(x) + \lambda_k \psi_0(x) \right) + \omega f(x) \\
    & = \lambda_{k + 1} \psi_0(x) + (1 - \lambda_{k + 1}) f(x).
\end{align*}
\end{proof}

\begin{lemma} \label{w solution estimation}
Let variables $s, m$ defined at Algorithm~\ref{alg re-agm}, $\omega$ is the largest root of the quadratic equation
\begin{equation}\label{omega lemma def}
    m \omega^2 + (s - m) \omega - q = 0,
\end{equation}
where $q = \frac{\mu}{16 L}, 0 \leqslant \nu \leqslant \nicefrac{1}{2}$. Then:

\begin{equation*}
     \frac{1}{8} \left(\frac{\mu}{2L} \right)^{1 - \nu} \leqslant \omega \leqslant \frac{1}{4} \left(\frac{\mu}{2L} \right)^{1 - \nu}.
\end{equation*}

\end{lemma}

\begin{proof}
By solving the equation \eqref{omega lemma def} and choosing the largest root we get
\begin{equation} \label{omega quadsol}
    \omega = \frac{(m - s) + \sqrt{(s - m)^2 + 4mq }}{2 m}.
\end{equation}

Let us assume that $\omega > \frac{1}{4} \left(\frac{\mu}{2L} \right)^{1 - \nu} $. Then

\begin{eqnarray*}
    m \omega^2 + (s - m) \omega - q 
    & > & \left(1 - \frac{1}{4} \left(\frac{\mu}{2L} \right)^{\nu} \right) \frac{1}{16} \left(\frac{\mu}{2L} \right)^{2(1 - \nu)} \\
    & + & \frac{1}{2} \left(\frac{\mu}{2L} \right)^{\nu} \cdot \frac{1}{4} \left(\frac{\mu}{2L} \right)^{1 - \nu} - \frac{1}{16} \frac{\mu}{L} > 0.
\end{eqnarray*}

which is contradiction with \eqref{omega lemma def}. Hence $\omega \leqslant \left(\frac{\mu}{2L} \right)^{1 - \nu}$.

One can note, that $m < 1$ (from definition). So we can prove lower bound for $\omega$. Let us prove by contradiction. To that end, we assume $\omega < \frac{1}{8} \left(\frac{\mu}{L}\right)^{1 - \nu}$. Then:

\begin{align*}
    m \omega^2 + (s - m) \omega - q 
    & < \frac{1}{64} \left(\frac{\mu}{2L}\right)^{2(1 - \nu)} + \frac{1}{2} \left(\frac{\mu}{2L} \right)^{\nu} \cdot \frac{1}{8} \left(\frac{\mu}{2L}\right)^{1 - \nu} - \frac{\mu}{16 L} \\
    & \overset{0 \leqslant \nu \leqslant \nicefrac{1}{2}}{\leqslant} \frac{1}{128} \frac{\mu}{L} + \frac{1}{32} \frac{\mu}{L} - \frac{1}{16} \frac{\mu}{L} < 0.
\end{align*}
Thus, we come to a contradiction because $\omega$ is a root of the equation $m \omega^2 + (s - m) \omega - q = 0$.

\end{proof}

\begin{lemma} \label{GD step}
    Let $f$ satisfies~\eqref{asm:regularity}, $\widetilde{\nabla} f$ satisfies noise condition~\eqref{def:abs grad error},
    \begin{equation*}
        f(x^{k + 1}) \leqslant f(y^k) - \frac{1}{16 L} \|\nabla f(y^k) \|_2^2 + \frac{3}{16 L} \Delta_k^2.
    \end{equation*}
    Where:
    \begin{equation} \label{gd const def}
        \begin{gathered}
            h = \frac{1}{4 L}.
        \end{gathered}
    \end{equation}
\end{lemma}
\begin{proof}
    Using smoothness of function $f$ we get:
    \begin{equation*}
        f(x^{k + 1}) \leqslant f(y^k) + \langle \nabla f(y^k), x^{k + 1} - y^k \rangle + \frac{L}{2} \| x^{k + 1} - y^k \|_2^2.
    \end{equation*}
    Let us estimate terms.
    Linear form:
    \begin{align*}
        \langle \nabla f(y^k), x^{k + 1} - y^k \rangle 
        & = -h \langle \nabla f(y^k), \wtgg f(y^k) \rangle \\
        & = -h \bigg (
        \| \nabla f(y^k) \|_2^2 + \langle \nabla f(y^k), \zeta_a(y^k) \rangle \bigg ) \\
        & \overset{\text{Fenchel ineq.}}{\leqslant}
        -h \bigg( \|\nabla f(y^k) \|_2^2 - 
        \frac{\lambda}{2} \|\nabla f(y^k) \|_2^2 - \frac{1}{2 \lambda} \Delta_k^2 \bigg) \\
        &  \overset{\lambda = 1}{\leqslant} -h \bigg( \frac{1}{2} \| \nabla f(y^k) \|_2^2  -
        \frac{1}{2} \Delta_k^2 \bigg).
    \end{align*}
    Quadratic term:
    \begin{eqnarray*}
        \| x^{k + 1} - y^k \|_2^2 
        & = & h^2 \| \widetilde{\nabla} f(y^k) \|_2^2 
        \leqslant 2 h^2 \left(\|\nabla f(y^k) \|_2^2 + \Delta_k^2 \right). 
    \end{eqnarray*}
    Combining this inequalities:
    \begin{eqnarray*}
        f(x^{k + 1}) 
        & \leqslant & f(y^k)
        - h \left( \frac{1}{2} \| \nabla f(y^k) \|_2^2 - \frac{1}{2} \Delta_k^2 \right) \\ 
        & + & L h^2 \left( \|\nabla f(y^k) \|_2^2 + \Delta_k^2 \right) \\
        & = & f(y^k) + h \|\nabla f(y^k) \|_2^2 \left(L h - \frac{1}{2} \right) \\ 
        & + & h \Delta_k^2 \left( \frac{1}{2} + h L \right) \\
        & \overset{h \text{ definition}}{=} & f(y^k) - \frac{1}{16 L}  \|\nabla f(y^k) \|_2^2 + \frac{3}{16 L} \Delta_k^2.
    \end{eqnarray*}
\end{proof}

\begin{lemma} \label{c > f lemma}
Let $f$ satisfies~\eqref{asm:regularity}, $\widetilde{\nabla} f$ satisfies~\eqref{def:abs grad error}, $\{x^k\}_{k \geqslant 0}$ be a sequence of points generated by Algorithm~\ref{alg re-agm}, and $c_k\, \forall k \geqslant 0$ be a minimal value of the function $\psi_k$ corresponding to Algorithm~\ref{alg re-agm}, $0 \leqslant \nu \leqslant \nicefrac{1}{2}$. Then $c_k \geqslant f(x^k) - E_k \, \forall k \geqslant 0$, where $E_k$ defined at~\eqref{error accumulation seq}.
\end{lemma}
\begin{proof}
We will use the principle of mathematical induction to prove the statement of this lemma. For $k = 0$, it is obvious that $c_0 \geqslant f(x^0)$. Now, let us assume that $c_k \geqslant f(x^k) - E_k$, and prove the statement for $k+1$. For this, from Lemma \ref{psi comb}, Def.~\eqref{psi seq def} and definition of $\phi_{y^k}$ from Lemma~\ref{lower bound} we have 

\begin{eqnarray}\label{c combination}
    c_{k + 1}
    & = & (1 - \omega) c_k
    \\
    & + & \omega \Bigg( f(y^k) - \frac{s}{\mu} \|\nabla f(y^k) \|_2^2
    - \frac{1}{\mu} \left(4 \left(\frac{2L}{\mu} \right)^{\nu} + 2 \right) \Delta_k^2 \Bigg )
    \\
    & + & \frac{\omega (1 - \omega) \mu}{4} \left \|y^k - u^k - \frac{2}{\mu} \widetilde{\nabla} f(y^k) \right \|_2^2.
\end{eqnarray}

From Lemma~\ref{accelerated quadratic estimation} we can estimate last term:
\begin{eqnarray*}
    \frac{\omega (1 - \omega) \mu}{4} \left \|y^k - u^k - \frac{2}{\mu} \widetilde{\nabla} f(y^k) \right \|_2^2 
    & \geqslant & \omega (1 - \omega) \langle \nabla f(y^k), u^k - y^k \rangle \\
    & + & \frac{m \omega (1 - \omega)}{\mu} \|\nabla f(y^k) \|_2^2
    \\
    & - & \frac{4 \omega (1 - \omega)}{\mu} \left(\frac{2L}{\mu} \right)^{\nu} \Delta_k^2.
\end{eqnarray*}
Combining the inequalities above and grouping the corresponding terms:
\begin{eqnarray*}
    c_{k + 1}
    & \geqslant & (1 - \omega) c_k + \omega f(y^k) + \omega (1 - \omega) \langle \nabla f(y^k), u^k - y^k \rangle
    \\
    & - & \frac{1}{\mu} \left( \omega s - \omega (1 - \omega) m \right) \|\nabla f(y^k) \|_2^2
    \\
    & \overset{(i)}{-} & \frac{\omega}{\mu} \left(8\left(\frac{2L}{\mu} \right)^{\nu} + 4 \right) \Delta_k^2, 
\end{eqnarray*}
$(i)$ can be obtained from $0 \leqslant \omega < 1$ (Lemma~\ref{omega quadsol}):
\begin{equation*}
    \omega \left(4 \left(\frac{2L}{\mu} \right)^{\nu} + 2 \right) + 4 \omega (1 - \omega) \left(\frac{2L}{\mu} \right)^{\nu} \leqslant \omega \left(8\left(\frac{2L}{\mu} \right)^{\nu} + 2 \right).
\end{equation*}
From convexity of function $f$ and induction hypothesis:
\begin{eqnarray*}
    (1 - \omega) c_k 
    & \geqslant & (1 - \omega) \left( f(x^k) - E_k \right)
    \\
    & \geqslant & (1 - \omega) \left( f(y^k) + \langle \nabla f(y^k), x^k - y^k \rangle \right) - (1 - \omega) \xi_k,
\end{eqnarray*}
then, using definition $y^k$ we can eliminate the linear form:
\begin{equation*}
     \langle \nabla f(y^k), \omega (1 - \omega) (u^k - y^k) + (1 - \omega) (x^k - y^k) \rangle = 0
\end{equation*}
and continue $c_{k + 1}$ estimation:
\begin{eqnarray*}
    c_{k + 1}
    & \geqslant & f(y^k) - (1 - \omega) E_k
    \\
    & - & \frac{1}{\mu} \left( \omega s - \omega (1 - \omega) m \right) \|\nabla f(y^k) \|_2^2
    \\
    & - & \frac{2\omega}{\mu} \left(4\left(\frac{2L}{\mu} \right)^{\nu} + 1 \right) \Delta_k^2, 
\end{eqnarray*}
From $\omega$ definition one can simplify the coefficient of the gradient norm:
\begin{eqnarray*}
    c_{k + 1}
    & \geqslant & f(y^k) - (1 - \omega) E_k - \frac{q}{\mu} \|\nabla f(y^k) \|_2^2
    - \frac{2\omega}{\mu} \left(4\left(\frac{2L}{\mu} \right)^{\nu} + 1 \right) \Delta_k^2
    \\
    & \geqslant & f(y^k) - (1 - \omega) E_k
    - \frac{1}{16 L} \|\nabla f(y^k) \|_2^2
    - \frac{2\omega}{\mu} \left(4\left(\frac{2L}{\mu} \right)^{\nu} + 1 \right) \Delta_k^2.
    \\
    & \overset{\text{Lemma}~\ref{w solution estimation}}{\geqslant} & f(y^k) - (1 - \omega) E_k
    - \frac{1}{16 L} \|\nabla f(y^k) \|_2^2
    - \left(\frac{\mu}{2L} \right)^{1 - 2 \nu} \frac{5 \Delta_k^2}{2 \mu}.
\end{eqnarray*}
Since $x^{k + 1} = y^k - h \wtgg f(y^k)$, where $h$ selected consistently with the Lemma~\ref{GD step} we can obtain:
\begin{eqnarray*}
    c_{k + 1}
    & \geqslant & f(x^{k + 1}) - (1 - \omega) E_k
    - \frac{3}{16 L} \Delta_k^2 - \left(\frac{\mu}{2L} \right)^{1 - 2 \nu} \frac{5 \Delta_k^2}{2 \mu} \\
    & = & f(x^{k + 1}) - (1 - \omega) E_k - \frac{\mu}{L} \cdot \frac{3 \Delta_k^2}{16\mu} - \left(\frac{\mu}{2L} \right)^{1 - 2 \nu} \frac{5 \Delta_k^2}{2 \mu} \\
    & \geqslant & f(x^{k + 1}) - (1 - \omega) E_k - 3 \left(\frac{\mu}{L} \right)^{1 - 2 \nu} \frac{\Delta_k^2}{\mu} \\ 
    & \overset{E_k \text{ definition}}{=} & f(x^{k + 1}) - E_{k + 1}.
\end{eqnarray*}

\end{proof}

Finally we able to prove convergence of Algorithm~\ref{alg re-agm}.

\begin{theorem} \label{igm convergence}
Let $f$ be an $L$-smooth and $\mu$-strongly convex function (Assumption~\ref{asm:regularity}), $\widetilde{\nabla} f$ satisfies error condition~\eqref{def:abs grad error}. Then Algorithm~\ref{alg re-agm} generates $x^N$, s.t. 
\begin{eqnarray*}
    f(x^N) - f^* & \leqslant &  \left(1 - \frac{1}{8} \left(\frac{\mu}{2L}\right)^{1 - \nu} \right)^N \left(f(x^0) - f^* + \frac{\mu}{4} R^2 \right)
    \\
    & + & 3 \left( \frac{\mu}{L} \right)^{1 - 2\nu} \sum_{k = 0}^{N - 1} \left(1 - \frac{1}{8} \left(\frac{\mu}{2L}\right)^{1 - \nu} \right)^{N - k - 1} \frac{\Delta_k^2}{\mu}
    \\
    & \leqslant & \left(1 - \frac{1}{16} \left(\frac{\mu}{L}\right)^{1 - \nu} \right)^N LR^2
    \\
    & + & 3 \left( \frac{\mu}{L} \right)^{1 - 2\nu} \sum_{k = 0}^{N - 1} \left(1 - \frac{1}{16} \left(\frac{\mu}{L}\right)^{1 - \nu} \right)^{N - k - 1} \frac{\Delta_k^2}{\mu}
\end{eqnarray*}
where $R := \|x^0 - x^* \|_2$.
\end{theorem}
\begin{proof}
From Lemma~\ref{c > f lemma} we obtain:
\begin{equation*}
    f(x^N) \leqslant c_N + E_N \quad \text{ see denotation~\eqref{error accumulation seq} }. 
\end{equation*}
Applying Lemma~\ref{upper bound for lemma} and then Lemma~\ref{conv lemma}:
\begin{eqnarray*}
    f(x^N) - f^*
    & \leqslant &
    \lambda_N (\psi_0(x^*) - f^*) + E_N \\
    & \overset{\eqref{error accumulation seq}}{\leqslant} & \lambda_N \left(f(x^0) - f^* + \frac{\mu}{4}R^2 \right)
    \\
    & + & 3 \left( \frac{\mu}{L} \right)^{1 - 2\nu} \sum_{k = 0}^{N - 1} (1 - \omega)^{N - k - 1} \frac{\Delta_k^2}{\mu}
    \\
    & \overset{\text{Lemma}~\ref{w solution estimation}}{\leqslant} & \left(1 - \frac{1}{8} \left(\frac{\mu}{2L}\right)^{1 - \nu} \right)^N \left(f(x^0) - f^* + \frac{\mu}{4} R^2 \right)
    \\
    & + & 3 \left( \frac{\mu}{L} \right)^{1 - 2\nu} \sum_{k = 0}^{N - 1} \left(1 - \frac{1}{8} \left(\frac{\mu}{2L}\right)^{1 - \nu} \right)^{N - k - 1} \frac{\Delta_k^2}{\mu}
    \\
    & \leqslant & \left(1 - \frac{1}{16} \left(\frac{\mu}{L}\right)^{1 - \nu} \right)^N LR^2
    \\
    & + & 3 \left( \frac{\mu}{L} \right)^{1 - 2\nu} \sum_{k = 0}^{N - 1} \left(1 - \frac{1}{16} \left(\frac{\mu}{L}\right)^{1 - \nu} \right)^{N - k - 1} \frac{\Delta_k^2}{\mu}.
\end{eqnarray*}

\end{proof}

So we can move to Prop.~\ref{prop:igm_adversarial_various}.

\textit{\textbf{Proof of Proposition}~\ref{prop:igm_adversarial_various} (case $\mu > 0$)}
\begin{proof}
    Using Theorem~\ref{igm convergence} and substitute $\nu = 1 - \frac{1}{p}, \Delta_k = 2\sqrt{d L \delta_k}$:
    \begin{eqnarray*}
    f(x^N) - f^* \leqslant \left(1 - \frac{1}{16} \left(\frac{\mu}{L}\right)^{\frac{1}{p}} \right)^N LR^2
    + 6 d \left( \frac{\mu}{L} \right)^{2\frac{1 - p}{p}} \sum_{k = 0}^{N - 1} \left(1 - \frac{1}{16} \left(\frac{\mu}{L}\right)^{\frac{1}{p}} \right)^{N - k - 1} \delta_k.
    \end{eqnarray*}
    We will apply Master Lemma~\ref{lemma:master} for convergence above. In designations of Master Lemma:
    \begin{equation*}
        \alpha_k = 6 d \left( \frac{\mu}{L} \right)^{2\frac{1 - p}{p}} \left(1 - \frac{1}{16} \left(\frac{\mu}{L}\right)^{\frac{1}{p}} \right)^{N - k - 1}, \quad 0 \leqslant k \leqslant N - 1.
    \end{equation*}
    Number of iterations to guarantee $\mathcal{E}_0(N, \Theta) \leqslant \nicefrac{\varepsilon}{2}$:
    \begin{equation*}
        N \leqslant 16 \left(\frac{L}{\mu}\right)^{\frac{1}{p}} \ln \left( \frac{2 LR^2}{\varepsilon} \right) + 1 \leqslant 17 \left(\frac{L}{\mu}\right)^{\frac{1}{p}} \ln \left( \frac{2 LR^2}{\varepsilon} \right) = O \left( \left(\frac{L}{\mu}\right)^{\frac{1}{p}} \ln \left( \frac{LR^2}{\varepsilon} \right) \right). 
    \end{equation*}
    Then, using Master Lemma~\ref{lemma:master} we must choose:
    \begin{eqnarray*}
        A_N & = & \sum_{k = 0}^{N - 1} \alpha_k^{\frac{\gamma}{1 + \gamma}} = \left(6 d \right)^{\frac{\gamma}{1 + \gamma}} \left( \frac{\mu}{L} \right)^{\frac{2\gamma(1 - p)}{p(1 + \gamma)}} \sum_{k = 0}^{N - 1} \left(1 - \frac{1}{16} \left(\frac{\mu}{L}\right)^{\frac{1}{p}} \right)^{(N - k - 1) \frac{\gamma}{1 + \gamma}} 
        \\
        & = & \left(6 d \right)^{\frac{\gamma}{1 + \gamma}} \left( \frac{\mu}{L} \right)^{2 \frac{2\gamma(1 - p)}{p(1 + \gamma)}} \frac{1 - \left(1 - \frac{1}{16} \left(\frac{\mu}{L}\right)^{\frac{1}{p}} \right)^{\frac{\gamma}{1 + \gamma}N}}{1 - \left(1 - \frac{1}{16} \left(\frac{\mu}{L}\right)^{\frac{1}{p}} \right)^{\frac{\gamma}{1 + \gamma}}}
        \\
    \end{eqnarray*}
    Hence estimation for Wall-clock time can be provided:
    \begin{equation*}
        T_{\mathrm{total}}^{\mathrm{IGM}} = \left(\frac{\varepsilon}{2} \right)^{-\gamma} A_N^{1 + \gamma}
        \leqslant
        \left( \frac{\varepsilon}{12 d} \right)^{-\gamma} \left(\frac{\mu}{L}\right)^{2\frac{\gamma(1 - p)}{p}} 
        \left(\frac{1 - \left(1 - \frac{1}{16} \left(\frac{\mu}{L}\right)^{\frac{1}{p}} \right)^{\frac{\gamma}{1 + \gamma} N}}{1 - \left(1 - \frac{1}{16} \left(\frac{\mu}{L}\right)^{\frac{1}{p}} \right)^{\frac{\gamma}{1 + \gamma}}} \right)^{1 + \gamma}.
    \end{equation*}

    There are two cases.

    \textit{Case 1.} $\gamma \approx 0$ (sufficient small).

    \begin{equation*}
        A_N = \sum_{k = 0}^{N - 1} \alpha_k^{\frac{\gamma}{1 + \gamma}} \overset{\alpha_k \leqslant 1}{\leqslant} N^{1 + \gamma}.
    \end{equation*}

    \begin{eqnarray*}
        T_{\mathrm{total}}^{\mathrm{IGM}} & \leqslant & \left( \frac{\varepsilon}{12 d} \right)^{-\gamma} \left(\frac{\mu}{L}\right)^{\frac{2\gamma(1 - p)}{p}} N^{1 + \gamma} \leqslant 17^{1 + \gamma} \left( \frac{\varepsilon}{12 d} \right)^{-\gamma} \left(\frac{L}{\mu}\right)^{2\gamma + \frac{1 - \gamma}{p}} \ln^{1 + \gamma} \left( \frac{2 LR^2}{\varepsilon} \right) \\
        & = & \widetilde{O} \left( \left( \frac{\varepsilon}{12 d} \right)^{-\gamma} \left(\frac{L}{\mu}\right)^{2\gamma + \frac{1 - \gamma}{p}} \right).
    \end{eqnarray*}


    \textit{Case 2.} $\gamma \not \approx 0$ (enough big).

    \begin{equation*}
        \left(\frac{1 - \left(1 - \frac{1}{16} \left(\frac{\mu}{L}\right)^{\frac{1}{p}} \right)^{\frac{\gamma}{1 + \gamma} N}}{1 - \left(1 - \frac{1}{16} \left(\frac{\mu}{L}\right)^{\frac{1}{p}} \right)^{\frac{\gamma}{1 + \gamma}}} \right)^{1 + \gamma}
        \leqslant \left( \frac{1}{\frac{1}{16} \frac{\gamma}{1 + \gamma} \left(\frac{\mu}{L}\right)^{\frac{1}{p}}} \right)^{1 + \gamma} \leqslant \left(16 \frac{1 + \gamma}{\gamma} \right)^{1 + \gamma} \left(\frac{L}{\mu}\right)^{\frac{1 + \gamma}{p}}.
    \end{equation*}

    Then $T_{\mathrm{total}}^{\mathrm{IGM}}$ can be estimated as:

    \begin{eqnarray*}
        T_{\mathrm{total}}^{\mathrm{IGM}} & \leqslant & \left( \frac{\varepsilon}{12 d} \right)^{-\gamma} \left(\frac{\mu}{L}\right)^{\frac{2\gamma(1 - p)}{p}} \left(16 \frac{1 + \gamma}{\gamma} \right)^{1 + \gamma} \left(\frac{L}{\mu}\right)^{\frac{1 + \gamma}{p}} 
        \\
        & = & O \left((1 + \gamma^{-1}) \left( \frac{\varepsilon}{12 d} \right)^{-\gamma} \left(\frac{L}{\mu}\right)^{\frac{1 - \gamma}{p} + 2\gamma} \right).
    \end{eqnarray*}

\end{proof}

\begin{remark}
    Same for Remark~\ref{remark:const delta adversarial} one can note, that for $\gamma < 1$, the wall-clock time complexity is minimal for $p=2$, and optimal for $p = 1$, if $\gamma \ge 1$.
\end{remark}

\textit{\textbf{Proof of Proposition}~\ref{prop:igm_adversarial_various} (convex case)}
\begin{proof}
    We will use regularization technique described at~\citet{vasin2023accelerated} (Remark 4.2). We introduce:
    \begin{equation*}
        \begin{gathered}
            f_{\mu}(x | x^0) = f(x) + \frac{\mu}{2} \|x - x^0 \|_2^2.
            x_\mu^* \text{ - solution for problem: } f_{\mu}(\cdot | x^0) \to \min_{x \in \mathbb{R}^n}, \\
            f_{\mu}^* = f_{\mu}(x^*_{\mu} | x^0), \\
            R_\mu = \|x_{\mu}^* - x^0 \|_2, \\
            \wtgg f_{\mu}(x) = \wtgg f(x) + \mu (x - x^0).
        \end{gathered}
    \end{equation*}
    Note, that $f_{\mu}$ is $\mu$-strongly convex, since $f$ is convex. We will apply result of Proposition~\ref{prop:igm_adversarial_various} to obtain $\mathcal{E}_{\mu}(N,\{\delta_k\}_{k\in[N]},\Theta) \leqslant \frac{\varepsilon}{2}$, where $\mathcal{E}_{\mu}$ corresponding error bound for $f_{\mu}$. We chose $\mu = \frac{\varepsilon}{R^2}$, thus:
    \begin{equation*}
        T_{\mathrm{total}}^{\mathrm{IGM}} = \begin{cases}
            \widetilde{O} \left( \left( \frac{\varepsilon}{24 d} \right)^{-\gamma} \left(\frac{2 LR^2}{\varepsilon}\right)^{\frac{1 - \gamma}{p} + 2\gamma} \right), & ~\text{if}~~\gamma \approx 0,\\
            O \left( (1 + \gamma^{-1}) \left( \frac{\varepsilon}{24 d} \right)^{-\gamma} \left(\frac{2 LR^2}{\varepsilon}\right)^{\frac{1 - \gamma}{p} + 2\gamma} \right) & ~\text{if}~~\gamma \not\approx 0.
        \end{cases}
    \end{equation*}
    Then we can estimate $\mathcal{E}(N,\{\delta_k\}_{k\in[N]},\Theta)$ for original function $f$:
    \begin{eqnarray*}
        \mathcal{E}(N,\{\delta_k\}_{k\in[N]},\Theta) = f(x^N) - f^*
        & \leqslant & f(x^N) + \frac{\mu}{2} \|x^N - x^0 \|_2^2 - f^* \\
        & = & f_{\mu}(x^N | x^0) - f^* = f_{\mu}(x^N | x^0) - f_{\mu}(x^* | x^0) + \frac{\mu}{2} R^2 \\
        & \leqslant & f_{\mu}(x^N | x^0) - f_{\mu}^* + \frac{\varepsilon}{2} \leqslant \varepsilon.
    \end{eqnarray*}
\end{proof}

\section{Proof of Proposition~\ref{prop:adv2}}
\label{app:adv_design}

We derive the optimal wall-clock design for the adversarial-noise model
directly from the bound in Sec.~\ref{sec:adv_noise}. Set
\[
    T:=BN,
    \qquad
    p:=\frac{\beta-1}{\beta}.
\]
The wall-clock objective is $  T_{\rm total}=T\delta^{-\gamma}.$

\textbf{Step 1: optimizing out the smoothing radius.}

The adversarial bound contains two bias terms depending on \(h\) and
\(\delta\), $\kappa_\beta Lh^{\beta-1}$ and $\frac{d\delta}{h}$. For fixed \(\delta\), we balance them:
\[
    \kappa_\beta Lh^{\beta-1}
    \asymp
    \frac{d\delta}{h}.
\]
Thus
\[
    h^\beta
    =
    \frac{d\delta}{\kappa_\beta L},
    \qquad
    h^*(\delta)
    =
    \left(
        \frac{d\delta}{\kappa_\beta L}
    \right)^{1/\beta}.
\]

With this choice,
\[
    \kappa_\beta L(h^*)^{\beta-1}
    \asymp
    \frac{d\delta}{h^*}
    \asymp
    a_1\delta^p,
\quad
    a_1
    :=
    (\kappa_\beta L)^{1/\beta}d^{(\beta-1)/\beta}.
\]
Similarly,
\[
    \kappa_\beta^2L(h^*)^{2(\beta-1)}
    \asymp
    \frac{d^2\delta^2}{L(h^*)^2}
    \asymp
    a_2\delta^{2p},\quad
    a_2
    :=
    \kappa_\beta^{2/\beta}
    L^{\frac{2}{\beta}-1}
    d^{\frac{2(\beta-1)}{\beta}}.
\]

Therefore, after substituting \(h=h^*(\delta)\), the error bound becomes
\[
\mathcal E(N,T,\delta)
\lesssim
\mathcal E_0(N,T)
+
\frac{R}{\sqrt T}
\left(
    c_1\delta^{1/(2\beta)}
    +
    c_2\delta^{p/2}
\right)
+
2Ra_1\delta^p
+
2Na_2\delta^{2p},
\]
where
\[
    \mathcal E_0(N,T)
    :=
    \frac{LR^2}{N^2}
    +
    \frac{LR^2}{T}
    +
    \frac{Rc_0}{\sqrt T},
    \qquad
    c_0^2:=\kappa d\sigma_*.
\]

\textbf{Step 2: the \(\delta\)-independent part.}

The \(\delta\)-independent part contains two horizons, $N$ and $ T = BN$. With respect to \(N\),
\[
    \mathcal E_{0,N}(N)
    =
    \frac{LR^2}{N^2},
    \qquad
    \mathcal E_{0,N}(N)\asymp N^{-2}.
\]
Thus the \(N\)-exponent is
\[
    \boxed{\beta_0^{(N)}=2.}
\]

With respect to \(T\), the large-\(T\) part is
\[
    \frac{LR^2}{T}
    +
    \frac{Rc_0}{\sqrt T}.
\]
In the high-precision regime and for \(c_0>0\), the term
\(Rc_0/\sqrt T\) dominates \(LR^2/T\). Hence
\[
    \mathcal E_{0,T}(T)\asymp T^{-1/2},
    \qquad
    \boxed{\beta_0^{(T)}=\frac12.}
\]
The constraint $\frac{Rc_0}{\sqrt T}\lesssim \varepsilon$
gives
\[
    T\gtrsim
    T_{\min}
    :=
    \frac{R^2c_0^2}{\varepsilon^2}
    =
    \frac{R^2\kappa d\sigma_*}{\varepsilon^2}.
\]

\textbf{Step 3: fidelity channels.}

After balancing \(h\), the active fidelity channels are $2Ra_1\delta^p$ and $2Na_2\delta^{2p}.$ For fixed \(N\), the accuracy constraint imposes
\[
    \delta
    \lesssim
    \delta_D
    :=
    \left(
        \frac{\varepsilon}{2Ra_1}
    \right)^{1/p},
    \quad
    \delta
    \lesssim
    \delta_E(N)
    :=
    \left(
        \frac{\varepsilon}{2Na_2}
    \right)^{1/(2p)}.
\]
Since $    T_{\rm total}=T\delta^{-\gamma}$, 
is decreasing in \(\delta\) for every \(\gamma>0\), the optimal fidelity
is the largest admissible one:
\[
    \delta^*(N)
    =
    \min\{\delta_D,\delta_E(N)\}.
\]

The transition between the two fidelity channels occurs when $    \delta_D=\delta_E(N).$
Equivalently,
\[
    \left(
        \frac{\varepsilon}{2Ra_1}
    \right)^{1/p}
    =
    \left(
        \frac{\varepsilon}{2Na_2}
    \right)^{1/(2p)}.
\]
Solving for \(N\) gives
\[
    N_{\rm cr}
    =
    \frac{2R^2a_1^2}{a_2\varepsilon}.
\]
Using $\frac{a_1^2}{a_2}=L,$
we obtain
\[
    \boxed{
    N_{\rm cr}
    \asymp
    \frac{LR^2}{\varepsilon}.
    }
\]

Thus, in the large-\(N\) regime \(N\gtrsim N_{\rm cr}\), the accumulated
quadratic channel is active:
\[
    \delta^*(N)
    =
    \left(
        \frac{\varepsilon}{2Na_2}
    \right)^{1/(2p)}.
\]

\textbf{Step 4: optimizing \(N\).}

In the large-\(N\) regime,
\[
    T_{\rm total}(N,T)
    =
    T
    \left(
        \frac{2Na_2}{\varepsilon}
    \right)^{\gamma/(2p)}.
\]
For fixed \(T\), this is increasing in \(N\). Therefore the optimal
choice is the smallest admissible \(N\),
\[
    \boxed{
    N^*
    \asymp
    N_{\rm cr}
    \asymp
    \frac{LR^2}{\varepsilon}.
    }
\]

Substituting \(N^*=N_{\rm cr}\) into the active fidelity constraint gives
\[
    \delta^*
    =
    \left(
        \frac{\varepsilon}{2a_2N_{\rm cr}}
    \right)^{1/(2p)}
    \asymp
    \left(
        \frac{\varepsilon^2}{a_2LR^2}
    \right)^{1/(2p)}.
\]
Expanding \(a_2\), we obtain
\[
    \boxed{
    \delta^*
    \asymp
    d^{-1}
    \kappa_\beta^{-1/(\beta-1)}
    L^{-1/(\beta-1)}
    R^{-\beta/(\beta-1)}
    \varepsilon^{\beta/(\beta-1)}.
    }
\]
The corresponding smoothing radius is
\[
    \boxed{
    h^*
    =
    h^*(\delta^*)
    =
    \left(
        \frac{d\delta^*}{\kappa_\beta L}
    \right)^{1/\beta}
    \asymp
    \kappa_\beta^{-1/(\beta-1)}
    L^{-1/(\beta-1)}
    R^{-1/(\beta-1)}
    \varepsilon^{1/(\beta-1)}.
    }
\]

\textbf{Step 5: optimizing \(T=BN\).}

The \(\delta\)-independent term \(Rc_0/\sqrt T\) gives
\[
    T\gtrsim
    T_{\min}
    =
    \frac{R^2c_0^2}{\varepsilon^2}.
\]
Since \(T_{\rm total}=T\delta^{-\gamma}\) is increasing in \(T\), the
optimal choice is
\[
    \boxed{
    T^*
    \asymp
    T_{\min}
    =
    \frac{R^2\kappa d\sigma_*}{\varepsilon^2}.
    }
\]

Therefore
\[
    \boxed{
    B^*
    =
    \frac{T^*}{N^*}
    \asymp
    \frac{\kappa d\sigma_*}{L\varepsilon}.
    }
\]

\textbf{Step 6: optimized wall-clock cost.}

The optimized wall-clock cost is
$ T_{\rm total}^*
    =
    T^*(\delta^*)^{-\gamma}.$ Using the expressions above,
\[
    T^*
    \asymp
    \frac{R^2\kappa d\sigma_*}{\varepsilon^2},
\]
and
\[
    (\delta^*)^{-\gamma}
    \asymp
    d^\gamma
    \kappa_\beta^{\gamma/(\beta-1)}
    L^{\gamma/(\beta-1)}
    R^{\gamma\beta/(\beta-1)}
    \varepsilon^{-\gamma\beta/(\beta-1)}.
\]
Hence
\[
    \boxed{
    T_{\rm total}^*
    \asymp
    \kappa\sigma_*
    d^{1+\gamma}
    \kappa_\beta^{\gamma/(\beta-1)}
    L^{\gamma/(\beta-1)}
    R^{2+\gamma\beta/(\beta-1)}
    \varepsilon^{-\left(2+\gamma\beta/(\beta-1)\right)}.
    }
\]
Equivalently, hiding problem-dependent constants,
\[
    T_{\rm total}^*
    =
    O\!\left(
        d^{1+\gamma}
        \varepsilon^{-\left(
            2+\frac{\gamma\beta}{\beta-1}
        \right)}
    \right).
\]

\textbf{Step 7: batching regimes.}

Since
\[
    B^*=\frac{T^*}{N^*},
\]
the batching regime is determined by comparing \(B^*\) and \(N^*\), or
equivalently comparing \(T^*\) with \((N^*)^2\). We have
\[
    T^*
    \asymp
    \frac{R^2\kappa d\sigma_*}{\varepsilon^2},
\quad
    (N^*)^2
    \asymp
    \frac{L^2R^4}{\varepsilon^2}.
\]
Therefore
\[
    B^*\le N^*
    \quad\Longleftrightarrow\quad
    T^*\le (N^*)^2
    \quad\Longleftrightarrow\quad
    \kappa d\sigma_*
    \lesssim
    L^2R^2.
\]
Similarly,
\[
    B^*>N^*
    \quad\Longleftrightarrow\quad
    \kappa d\sigma_*
    \gtrsim
    L^2R^2.
\]
Thus the batching transition is
\[
    \boxed{
    \kappa d\sigma_*
    \asymp
    L^2R^2.
    }
\]

\section{Proof of Proposition~\ref{prop:acc_tsyb_design}}
\label{app:acc_tsyb_design}
Throughout this proof we use the bounded-trajectory assumption
\[
    \widetilde R
    =
    \max_t\{\|z_t-x^*\|,\|z_t-y_t\|\}
    =
    O(R).
\]
Thus the bias term
\(\widetilde R\kappa_\beta L_\beta h^{\beta-1}\) is of order
\(R\kappa_\beta L_\beta h^{\beta-1}\).

\textbf{Small-batch regime: \(B\le 4\kappa d\).} 

Set $Q:=BN .$ Then we get
\[
E(Q,B,\delta,h)
\lesssim
\underbrace{
\frac{\kappa^2d^2LR^2}{Q^2}
}_{\mathcal E_0(Q)}
+
\underbrace{
\frac{Q}{\kappa Lh^2}\delta^2
}_{\text{Tsybakov channel}}
+
S_{\rm sm}(Q,B,h),
\]
where the smoothing remainder is
\[
S_{\rm sm}(Q,B,h)
=
\frac{QLh^2}{\kappa d}
+
R\kappa_\beta L_\beta h^{\beta-1}
+
\frac{Q}{B}
\frac{(\kappa_\beta L_\beta h^{\beta-1})^2}{L}.
\]
Hence $ \mathcal E_0(Q)
    =
    \kappa^2d^2LR^2\, Q^{-2},
    \qquad \mathcal E_0(Q) \sim Q^{-2}.$ For a fixed admissible \(h\), the active wall-clock subproblem is
\[
    \min_{Q,\delta} Q\delta^{-\gamma}
    \qquad
    \text{s.t.}
    \qquad
    \frac{A}{Q^2}+CQ\delta^2 \lesssim \varepsilon,
\]
with $ A:=\kappa^2d^2LR^2, C:=\frac{1}{\kappa Lh^2}.$ For fixed \(Q\), the largest feasible \(\delta\) is optimal:
\[
    \delta^2(Q)
    =
    \frac{\varepsilon-A/Q^2}{CQ}.
\]
Substituting this into \(T=Q\delta^{-\gamma}\) gives
\[
    T(Q)
    =
    C^{\gamma/2}
    Q^{1+\gamma/2}
    \left(
        \bar\varepsilon-\frac{A}{Q^2}
    \right)^{-\gamma/2}.
\]
Let $ s:=\frac{A/Q^2}{\varepsilon}\in(0,1).$ Then $    Q=\left(\frac{A}{\varepsilon s}\right)^{1/2},$ and, up to \(s\)-independent factors,
\[
    T(s)
    \propto
    s^{-(2+\gamma)/4}(1-s)^{-\gamma/2}.
\]
The first-order condition yields
\[
    s^*
    =
    \frac{2+\gamma}{2+3\gamma}.
\]
Therefore
\[
    \frac{A}{(Q^*)^2}
    =
    \frac{2+\gamma}{2+3\gamma}\varepsilon,
    \qquad
    CQ^*(\delta^*)^2
    =
    \frac{2\gamma}{2+3\gamma}\varepsilon.
\]
Thus
\[
    Q^*
    =
    \kappa dR\sqrt L\,
    \varepsilon^{-1/2}
    \left(
        \frac{2+3\gamma}{2+\gamma}
    \right)^{1/2},
\]
and
\[
    \delta^*
    =
    hL^{1/4}(dR)^{-1/2}
    \varepsilon^{3/4}
    G(\gamma),
\]
where
\[
    G(\gamma)
    :=
    \frac{(2\gamma)^{1/2}(2+\gamma)^{1/4}}
         {(2+3\gamma)^{3/4}}.
\]

It remains to choose \(h\). Since
\[
    T^*
    =
    Q^*(\delta^*)^{-\gamma}
    \propto h^{-\gamma},
\]
the optimal smoothing radius is the largest one allowed by
\(S_{\rm sm}(Q^*,B,h)\lesssim\varepsilon\). Therefore
\[
    h_{\rm sm}^*(B)
    =
    \min\{H_1,H_2,H_3(B)\},
\]
where
\[
    H_1
    \asymp
    \left(
        \frac{\varepsilon\kappa d}{Q^*L}
    \right)^{1/2}
    \asymp
    \varepsilon^{3/4}R^{-1/2}L^{-3/4},
\quad
    H_2
    \asymp
    \left(
        \frac{\varepsilon}{R\kappa_\beta L_\beta}
    \right)^{1/(\beta-1)},
\]
and
\[
    H_3(B)
    \asymp
    \left(
        \frac{\varepsilon LB}{Q^*\kappa_\beta^2L_\beta^2}
    \right)^{1/(2\beta-2)}
    \asymp
    \left(
        \frac{B\varepsilon^{3/2}\sqrt L}
        {\kappa dR\kappa_\beta^2L_\beta^2}
    \right)^{1/(2\beta-2)}.
\]
Consequently,
\[
    T_{\rm sm}^*(B)
    \asymp
    \kappa
    d^{1+\gamma/2}
    R^{1+\gamma/2}
    L^{1/2-\gamma/4}
    \varepsilon^{-1/2-3\gamma/4}
    \bigl(h_{\rm sm}^*(B)\bigr)^{-\gamma}
    C_\gamma,
\]
where \(C_\gamma\) depends only on \(\gamma\).

The active terms depend on \(N\) and \(B\) only through \(Q=BN\).
Thus batching does not improve the leading active trade-off itself.
It can only affect the smoothing cap through \(H_3(B)\). In the regime
where \(H_1\) or \(H_2\) is active, one may take \(B^*=1\) and
\(N^*=Q^*\). If \(H_3(B)\) is active, increasing \(B\) can enlarge the
admissible smoothing radius, but only up to the boundary
\(B\le 4\kappa d\).

\textbf{Large-batch regime: \(B>4\kappa d\).}

In this case, we get
\[
E(N,B,\delta,h)
\lesssim
\underbrace{
\frac{LR^2}{N^2}
}_{\mathcal E_0(N)}
+
\underbrace{
\frac{N\kappa d^2}{BLh^2}\delta^2
}_{\text{Tsybakov channel}}
+
S_{\rm lg}(N,B,h),
\]
where
\[
S_{\rm lg}(N,B,h)
=
\frac{N\kappa dLh^2}{B}
+
R\kappa_\beta L_\beta h^{\beta-1}
+
\frac{N(\kappa_\beta L_\beta h^{\beta-1})^2}{L}.
\]
Hence
\[
    \mathcal E_0(N)
    =
    LR^2\,N^{-2},
    \qquad \mathcal E_0(N)\sim N^{-2}
\]
For fixed \(B\) and admissible \(h\), the active subproblem is
\[
    \min_{N,\delta} NB\delta^{-\gamma}
    \qquad
    \text{s.t.}
    \qquad
    \frac{A}{N^2}
    +
    C_BN\delta^2
    \le \varepsilon,
\]
where $A:=LR^2,$ $
    C_B:=\frac{\kappa d^2}{BLh^2}.$
Repeating the one-dimensional optimization above gives
\[
    N^*
    =
    R\sqrt L\,
    \varepsilon^{-1/2}
    \left(
        \frac{2+3\gamma}{2+\gamma}
    \right)^{1/2},
\]
and
\[
    \delta^*(B,h)
    =
    B^{1/2}
    hL^{1/4}
    (\kappa d^2R)^{-1/2}
    \varepsilon^{3/4}
    G(\gamma).
\]
Substituting this into \(T=NB\delta^{-\gamma}\) yields
\[
    T_{\rm lg}^*(B,h)
    \asymp
    B^{1-\gamma/2}
    \kappa^{\gamma/2}
    d^\gamma
    R^{1+\gamma/2}
    L^{1/2-\gamma/4}
    \varepsilon^{-1/2-3\gamma/4}
    h^{-\gamma}
    C_\gamma .
\]
Thus the dependence on \(B\) is
\[
    T_{\rm lg}^*(B,h)\propto B^{1-\gamma/2}.
\]

If \(0<\gamma<2\), increasing \(B\) is not beneficial at the active
trade-off level, so the optimum is attained at the smallest admissible
large batch and is covered, up to constants, by the small-batch regime.
If \(\gamma=2\), the leading dependence on \(B\) is flat. If
\(\gamma>2\), then \(T_{\rm lg}^*(B,h)\) decreases with \(B\) as long
as the interior solution has \(\delta^*(B,h)<1\). Therefore the optimal
large-batch strategy increases \(B\) until the boundary $    \delta^*=1$
is reached.

\textbf{Boundary solution for \(\gamma>2\).}

At the boundary \(c(\delta^*)=1\), so the active problem becomes
\[
    \min_{N,B} NB
    \qquad
    \text{s.t.}
    \qquad
    \frac{LR^2}{N^2}
    +
    \frac{N\kappa d^2}{BLh^2}
    \le \varepsilon.
\]
For fixed \(N\), the smallest feasible \(B\) is
\[
    B(N)
    =
    \frac{N\kappa d^2}
    {Lh^2(\varepsilon-LR^2/N^2)}.
\]
Therefore
\[
    T(N)
    =
    NB(N)
    =
    \frac{\kappa d^2N^2}
    {Lh^2(\varepsilon-LR^2/N^2)}.
\]
Minimizing in \(N\) gives $    \frac{LR^2}{(N^*)^2}
    \asymp
    \varepsilon,$
and hence
\[
    N^*
    \asymp
    R\sqrt L\,\varepsilon^{-1/2}.
\]
The corresponding batch size is
\[
    B^*
    \asymp
    \frac{\kappa d^2N^*}{\varepsilon Lh^2}
    \asymp
    \kappa d^2R L^{-1/2}
    \varepsilon^{-3/2}h^{-2}.
\]
Thus
\[
    T^*
    =
    N^*B^*
    \asymp
    \kappa d^2R^2
    \varepsilon^{-2}h^{-2}.
\]

It remains to choose \(h\). Substituting \(N^*\) and \(B^*(h)\) into
\(S_{\rm lg}(N,B,h)\lesssim\varepsilon\) gives
\[
    h_{\rm lg}^*
    =
    \min\{K_1,K_2,K_3\},
\]
where
\[
    K_1
    \asymp
    d^{1/4}L^{-1/2},
\quad
    K_2
    \asymp
    \left(
        \frac{\varepsilon}{R\kappa_\beta L_\beta}
    \right)^{1/(\beta-1)},
\quad
    K_3
    \asymp
    \left(
        \frac{\varepsilon^{3/2}\sqrt L}
        {R\kappa_\beta^2L_\beta^2}
    \right)^{1/(2\beta-2)}.
\]
Since \(T^*\propto h^{-2}\), the optimal choice is
\(h=h_{\rm lg}^*\), yielding
\[
    T_{\rm overbatch}^*
    \asymp
    \kappa d^2R^2
    \varepsilon^{-2}
    \bigl(h_{\rm lg}^*\bigr)^{-2}.
\]
This proves the large-batch overbatching regime for \(\gamma>2\).

\section{Proof of Proposition~\ref{prop:akhavan_walltime_batched}}
\label{app:proof_tsybakov_batch}
\label{app:akhavan_batched_bound}
\label{app:akhavan_batched_proof}

We first justify the batched extension used in
\eqref{eq:akhavan_bound_batched}. Let
\[
    \bar g_t=\frac1B\sum_{i=1}^B g_t^{(i)}
\]
be the average of \(B\) conditionally independent copies of the
\(\ell_2\)-randomized estimator. Starting from the proof of Corollary~19
in~\citet{akhavan2024gradient}, batching only changes the oracle-noise
contributions: the deterministic and smoothing-bias terms remain unchanged,
while the Tsybakov variance level is divided by \(B\). Indeed, writing
\[
    g_t=s_t+n_t,
    \qquad
    n_t=\frac{d}{2h_t}(\xi_t-\xi'_t)\zeta_tK(r_t),
\]
the independent averaging gives
\[
    \mathbb E\left[
        \left\|
        \frac1B\sum_{i=1}^B n_t^{(i)}
        \right\|^2
        \,\middle|\,x_t
    \right]
    =
    \frac1B\mathbb E[\|n_t\|^2\mid x_t].
\]
Thus every occurrence of the Tsybakov variance level \(\delta^2\) in the
noise-dependent terms is replaced by \(\delta^2/B\). Applying the same
strongly convex argument as in~\citet{akhavan2024gradient} with this
replacement gives \eqref{eq:akhavan_bound_batched}.

We now solve the leading-order wall-clock problem. Put
\[
    a:=\frac{\beta-1}{\beta},
    \qquad
    C_0:=\frac{dL^2R^2}{\mu},
    \qquad
    C_A:=\frac{L_\beta^{2/\beta}}{\mu}d^{2a}.
\]
Ignoring the lower-order residual channel, the active constraint is
\[
    \frac{C_0}{N}+C_A\delta^{2a}(BN)^{-a}\le \varepsilon .
\]
Let
\[
    Q:=BN,
    \qquad
    \eta_N:=\varepsilon-\frac{C_0}{N},
    \qquad
    S_N:=\left(\frac{C_A}{\eta_N}\right)^{1/a}.
\]
For fixed \(N\) and \(\delta\), feasibility requires \(\eta_N>0\), and the
smallest feasible total number of oracle calls is
\[
    Q^*(N,\delta)=\max\{N,S_N\delta^2\}.
\]
Thus, after optimizing over \(Q\), the fixed-\(N\) objective is
\[
    T_N(\delta)=\max\{N,S_N\delta^2\}\delta^{-\gamma},
    \qquad 0<\delta\le1.
\]

If \(0<\gamma<2\), then \(T_N(\delta)\) is minimized at the balancing point
\[
    S_N(\delta^*)^2=N.
\]
Hence
\[
    (\delta^*)^2
    =
    N\left(\frac{\eta_N}{C_A}\right)^{1/a},
    \qquad
    Q^*(N,\delta^*)=N,
    \qquad
    B^*=\frac{Q^*}{N}=1.
\]
Substituting this \(\delta^*\) into the cost gives
\[
    T(N)
    =
    N^{1-\gamma/2}
    C_A^{\gamma/(2a)}
    \left(
        \varepsilon-\frac{C_0}{N}
    \right)^{-\gamma/(2a)}.
\]
Minimizing the last display over \(N>C_0/\varepsilon\) yields
\[
    N^*
    =
    \frac{C_0}{\varepsilon}
    \frac{2\beta-2+\gamma}{(\beta-1)(2-\gamma)}.
\]
Consequently,
\[
    \delta^*
    \asymp
    d^{-1/2}L R\,
    \mu^{\frac{1}{2(\beta-1)}}
    L_\beta^{-\frac{1}{\beta-1}}
    \varepsilon^{\frac{1}{2(\beta-1)}}
    C_\beta(\gamma),
\]
where
\[
    C_\beta(\gamma)
    =
    \left(
        \frac{\gamma\beta}{2\beta-2+\gamma}
    \right)^{\frac{\beta}{2(\beta-1)}}
    \left(
        \frac{2\beta-2+\gamma}{(\beta-1)(2-\gamma)}
    \right)^{1/2}.
\]
Therefore
\[
    T_{\rm total}^*
    \asymp
    d^{1+\gamma/2}
    L^{2-\gamma}
    R^{2-\gamma}
    L_\beta^{\frac{\gamma}{\beta-1}}
    \mu^{-1-\frac{\gamma}{2(\beta-1)}}
    \varepsilon^{-1-\frac{\gamma}{2(\beta-1)}}
    C_\beta(\gamma).
\]

If \(\gamma\ge2\), then the fixed-\(N\) objective is minimized, up to
leading order, at the largest admissible fidelity level \(\delta^*=1\).
Indeed, on the region \(S_N\delta^2\le N\), one has
\[
    T_N(\delta)=N\delta^{-\gamma},
\]
which decreases with \(\delta\). On the region \(S_N\delta^2\ge N\), one has
\[
    T_N(\delta)=S_N\delta^{2-\gamma},
\]
which is decreasing for \(\gamma>2\) and constant for \(\gamma=2\). Hence
\(\delta^*=1\) is leading-order optimal for all \(\gamma\ge2\).

Substituting \(\delta=1\) into \(Q^*(N,\delta)\), we obtain
\[
    Q^*(N,1)
    =
    \max\left\{
        N,
        \left(
            \frac{C_A}{\varepsilon-C_0/N}
        \right)^{1/a}
    \right\}.
\]
Thus the remaining one-dimensional problem is
\[
    \min_{N>C_0/\varepsilon}
    \max\left\{
        N,
        \left(
            \frac{C_A}{\varepsilon-C_0/N}
        \right)^{1/a}
    \right\}.
\]
The first term is increasing in \(N\), while the second term is decreasing in
\(N\). Therefore the continuous leading-order optimum is attained at the
balancing point
\[
    N
    =
    \left(
        \frac{C_A}{\varepsilon-C_0/N}
    \right)^{1/a}.
\]
Equivalently,
\[
    \varepsilon N-C_0
    =
    C_A N^{1-a}.
\]
Since \(1-a=1/\beta\), the high-precision solution satisfies
\[
    N^*
    \asymp
    C_A^{1/a}\varepsilon^{-1/a}.
\]
At the balancing point, \(Q^*=N^*\), and hence the sequential leading-order
optimizer can be chosen with
\[
    B^*=\frac{Q^*}{N^*}=1.
\]
Therefore
\[
    T_{\rm total}^*
    =
    Q^*
    \asymp
    C_A^{1/a}\varepsilon^{-1/a}.
\]
Substituting
\[
    C_A=\frac{L_\beta^{2/\beta}}{\mu}d^{2a},
    \qquad
    \frac1a=\frac{\beta}{\beta-1},
\]
gives
\[
    N^*
    \asymp
    T_{\rm total}^*
    \asymp
    d^2
    L_\beta^{2/(\beta-1)}
    \mu^{-\beta/(\beta-1)}
    \varepsilon^{-\beta/(\beta-1)},
    \qquad
    B^*=1.
\]

It remains to verify that the residual channel is lower-order. Let
\(A_{\rm main}\) and \(A_{\rm res}\) denote the main noise channel and the
residual channel in \eqref{eq:akhavan_bound_batched}. Then
\[
    \frac{A_{\rm res}}{A_{\rm main}}
    =
    L^2
    L_\beta^{-4/\beta}
    d^{4/\beta-1}
    \delta^{-2(\beta-2)/\beta}
    B^{(\beta-2)/\beta}
    N^{-2/\beta}.
\]
For \(0<\gamma<2\), we have
\[
    B^*=1,
    \qquad
    N^*\asymp\varepsilon^{-1},
    \qquad
    \delta^*\asymp\varepsilon^{1/(2(\beta-1))},
\]
and hence
\[
    \frac{A_{\rm res}}{A_{\rm main}}
    \asymp
    \varepsilon^{1/(\beta-1)}
    \to0.
\]
For \(\gamma\ge2\), the sequential leading-order optimizer has
\[
    \delta^*=1,
    \qquad
    B^*=1,
    \qquad
    N^*\asymp
    \varepsilon^{-\beta/(\beta-1)}.
\]
Therefore
\[
    \frac{A_{\rm res}}{A_{\rm main}}
    \asymp
    \varepsilon^{2/(\beta-1)}
    \to0.
\]
Thus the residual channel is lower-order in the high-precision regime.

Finally, although the sequential leading-order optimizer can be chosen with
\(B^*=1\), for \(\gamma\ge2\) the same total-work order can be redistributed
across batches. Taking \(\delta=1\), any choice satisfying
\[
    N\gtrsim \frac{C_0}{\varepsilon},
    \qquad
    Q=BN\gtrsim C_A^{1/a}\varepsilon^{-1/a}
\]
achieves the same leading accuracy. Hence one may reduce the sequential depth
to \(N\asymp C_0/\varepsilon\) by increasing
\[
    B\asymp
    \frac{C_A^{1/a}\varepsilon^{-1/a}}{C_0/\varepsilon},
\]
without changing the leading total oracle work.

\subsection{Details for regularization and comparison}
\label{app:regularization_comparison}


Let \(f\) be convex and assume that an optimal solution \(x^*\) exists with
\[
    \|x_0-x^*\|\le R.
\]
For \(\mu>0\), define
\[
    f_\mu(x)
    :=
    f(x)+\frac{\mu}{2}\|x-x_0\|^2.
\]
Then \(f_\mu\) is \(\mu\)-strongly convex. Let \(x_\mu^*\) be its minimizer.
Since
\[
    f_\mu(x_\mu^*)\le f_\mu(x^*),
\]
we have
\[
    f_\mu(x_\mu^*)-f^*
    \le
    \frac{\mu}{2}\|x^*-x_0\|^2
    \le
    \frac{\mu R^2}{2}.
\]
Therefore, if
\[
    f_\mu(\hat x)-f_\mu(x_\mu^*)\le \frac{\varepsilon}{2},
\]
then
\[
    f(\hat x)-f^*
    \le
    f_\mu(\hat x)-f^*
    \le
    \frac{\varepsilon}{2}+\frac{\mu R^2}{2}.
\]
Choosing
\[
    \mu=\frac{\varepsilon}{R^2}
\]
gives \(f(\hat x)-f^*\le\varepsilon\). Thus it is enough to solve the
regularized problem to accuracy \(O(\varepsilon)\) with
\[
    \mu\asymp\frac{\varepsilon}{R^2}.
\]

We now substitute this value of \(\mu\) into the strongly convex Tsybakov
baseline used by \textsc{Regularized ZO-SGD}. For \(0<\gamma<2\),
\[
    T_{\rm sc}
    \asymp
    d^{1+\gamma/2}
    L^{2-\gamma}
    R^{2-\gamma}
    L_\beta^{\gamma/(\beta-1)}
    \mu^{-1-\frac{\gamma}{2(\beta-1)}}
    \varepsilon^{-1-\frac{\gamma}{2(\beta-1)}}.
\]
Since
\[
    \mu^{-1-\frac{\gamma}{2(\beta-1)}}
    \asymp
    R^{2+\gamma/(\beta-1)}
    \varepsilon^{-1-\frac{\gamma}{2(\beta-1)}},
\]
we get
\[
    T_{\rm reg}
    \asymp
    d^{1+\gamma/2}
    L^{2-\gamma}
    L_\beta^{\gamma/(\beta-1)}
    R^{4-\gamma+\gamma/(\beta-1)}
    \varepsilon^{-2-\gamma/(\beta-1)}.
\]
Hence the \(\varepsilon\)-exponent is
\[
    p_{\rm reg}=2+\frac{\gamma}{\beta-1}.
\]

For \(\gamma\ge2\),
\[
    T_{\rm sc}
    \asymp
    d^2
    L_\beta^{2/(\beta-1)}
    \mu^{-\beta/(\beta-1)}
    \varepsilon^{-\beta/(\beta-1)}.
\]
Substituting \(\mu\asymp\varepsilon/R^2\) yields
\[
    T_{\rm reg}
    \asymp
    d^2
    L_\beta^{2/(\beta-1)}
    R^{2\beta/(\beta-1)}
    \varepsilon^{-2\beta/(\beta-1)}
    =
    d^2
    L_\beta^{2/(\beta-1)}
    R^{2\beta/(\beta-1)}
    \varepsilon^{-2-\frac{2}{\beta-1}}.
\]

We now compare this with \textsc{Accelerated ZO-SGD}. In the small-batch
regime,
\[
    T_{\rm acc}
    \asymp
    \kappa
    d^{1+\gamma/2}
    R^{1+\gamma/2}
    L^{1/2-\gamma/4}
    \varepsilon^{-1/2-3\gamma/4}
    \left(h_{\rm sm}^*(B)\right)^{-\gamma}.
\]
Here
\[
    h_{\rm sm}^*(B)=\min\{H_1,H_2,H_3(B)\},
\]
where
\[
    H_1\asymp \varepsilon^{3/4},
    \qquad
    H_2\asymp \varepsilon^{1/(\beta-1)},
\]
and
\[
    H_3(B)
    \asymp
    \left(
        \frac{B\varepsilon^{3/2}}{d}
    \right)^{1/(2\beta-2)}.
\]
Since the bound is decreasing in \(h_{\rm sm}^*(B)\), we take the largest
admissible batch size in the small-batch regime:
\[
    B\asymp4\kappa d.
\]
Then
\[
    H_3(B)
    \asymp
    \varepsilon^{3/(4(\beta-1))}
\]
up to \(\kappa\)-dependent constants. Since
\[
    \frac{3}{4(\beta-1)}
    <
    \frac1{\beta-1},
\]
we have \(H_2\lesssim H_3(B)\) for sufficiently small \(\varepsilon\).
Therefore \(H_3(B)\) is not active for the leading \(\varepsilon\)-exponent,
and
\[
    h_{\rm sm}^*(B)
    \asymp
    \varepsilon^s,
    \qquad
    s:=\max\left\{\frac34,\frac1{\beta-1}\right\}.
\]
Consequently,
\[
    T_{\rm acc}
    \asymp 
    d^{1+\gamma/2}
    \varepsilon^{-\left(\frac12+\frac{3\gamma}{4}+\gamma s\right)}.
\]
Thus
\[
    p_{\rm acc}
    =
    \frac12+\frac{3\gamma}{4}+\gamma s.
\]

If \(\beta\le7/3\), then \(s=1/(\beta-1)\). Hence
\[
    p_{\rm acc}
    =
    \frac12+\frac{3\gamma}{4}+\frac{\gamma}{\beta-1},
\]
and
\[
    p_{\rm reg}-p_{\rm acc}
    =
    \frac34(2-\gamma)>0.
\]
Thus \(p_{\rm acc}<p_{\rm reg}\) for all \(0<\gamma<2\).

If \(\beta>7/3\), then \(s=3/4\), so
\[
    p_{\rm acc}
    =
    \frac12+\frac{3\gamma}{2}.
\]
The accelerated method is better when
\[
    \frac12+\frac{3\gamma}{2}
    <
    2+\frac{\gamma}{\beta-1}.
\]
Solving this inequality gives
\[
    \gamma
    <
    \gamma_{\rm crit}(\beta)
    :=
    \frac{3(\beta-1)}{3\beta-5}.
\]
For \(\beta>7/3\), one has
\[
    1<\gamma_{\rm crit}(\beta)<2.
\]
Therefore, for \(\beta>7/3\), \textsc{Accelerated ZO-SGD} has the smaller
\(\varepsilon\)-exponent when
\(\gamma<\gamma_{\rm crit}(\beta)\), while \textsc{Regularized ZO-SGD} has
the smaller \(\varepsilon\)-exponent when
\(\gamma_{\rm crit}(\beta)<\gamma<2\).

Finally, in the regime \(\gamma\ge2\), \textsc{Accelerated ZO-SGD} gives
\[
    T_{\rm acc}
    \asymp 
    d^2\varepsilon^{-2}\left(h_{\rm lg}^*\right)^{-2}.
\]
Using
\[
    h_{\rm lg}^*\asymp \varepsilon^{1/(\beta-1)}
\]
for the leading \((d,\varepsilon)\)-scaling, we get
\[
    T_{\rm acc}
    \asymp
    d^2\varepsilon^{-2-\frac{2}{\beta-1}}.
\]
This matches the leading \((d,\varepsilon)\)-scaling of
\textsc{Regularized ZO-SGD}:
\[
    T_{\rm reg}
    \asymp 
    T_{\rm acc}
    \asymp
    d^2\varepsilon^{-2-\frac{2}{\beta-1}}.
\]
The difference is in sequential depth:
\textsc{Accelerated ZO-SGD} uses
\[
    N_{\rm acc}\asymp R\sqrt L\,\varepsilon^{-1/2},
\]
whereas \textsc{Regularized ZO-SGD} has deterministic sequential scale
\[
    N_{\rm reg}
    \asymp
    dL^2R^4\varepsilon^{-2}.
\]


\end{document}